\documentclass[11pt]{amsart}
\usepackage[backref]{hyperref}
\usepackage{amsmath,ascmac,color,graphicx,amsfonts,amssymb,latexsym,amsthm,enumerate,comment,bm,pifont,mathtools,stmaryrd,tikz-cd,cite}
\usepackage[T1]{fontenc}
\usepackage{lmodern}
\usepackage{tikz}
\def\Z{\mathbb{Z} } 
 
\def\Q{\mathbb{Q} } 
\def\R{\mathbb{R} } 

\def\N{\mathbb{N} }
\def\F{\mathbb{F} }
\def\FF{\mathcal{F}}
\def\AA{\mathcal{A}}
\def\BB{\mathcal{B}}

\def\vp{\varpi}
\def\aut{\mathrm{Aut}}
\def\gal{\mathrm{Gal}}
\def\deg{\mathrm{deg}}
\def\out{\mathrm{Out}}
\def\inn{\mathrm{Inn}}
\def\id{\mathrm{id}}

\theoremstyle{plain} 
\newtheorem{theorem}{Theorem}[section]
\newtheorem{proposition}[theorem]{Proposition}
\newtheorem{lemma}[theorem]{Lemma}
\newtheorem{corollary}[theorem]{Corollary}

\theoremstyle{definition}
\newtheorem{definition}{Definition}
\newtheorem{example}{Example}

\newtheorem{question}{Question}
\newtheorem{remark}{Remark}

\title[Galois Theory and Automorphism Groups of SFTs]{Galois Theory for Subshifts of Finite Type and Representations of Automorphism Groups}

\author{Kazutoyo Iketake}

\address{Mathematics, Electronics and Informatics Division, Graduate School of Science and Engineering, Saitama University, 255 Shimo-Okubo, Sakura-ku, Saitama-shi, Saitama 338-8570, Japan}



\date{May 27, 2026}

\email{iketake2115@gmail.com}

\begin{document}

\begin{abstract}
    The purpose of this paper is to constructively develop a Galois theory on irreducible shifts of finite type (SFTs) and to analyze the automorphism groups of SFTs using this framework. Let $X$ and $Y$ be irreducible SFTs. We demonstrate that a factor map $Y \to X$ satisfying algebraic conditions exhibits Galois-theoretic behavior. Specifically, we prove that a Galois correspondence, analogous to those known in the Galois theory of fields and covering spaces in topology, holds for irreducible SFTs. Furthermore, we introduce the absolute Galois group and its cohomology as conjugacy invariants for irreducible SFTs. Finally, we construct a representation of the automorphism group of an irreducible SFT into this cohomology to extract information from the automorphism group, whose structure is not yet fully understood.
\end{abstract}

\maketitle

\tableofcontents

\section{Introduction}

\subsection{Background}
The automorphism group of a subshift of finite type (SFT) has been studied in conjunction with the strong shift equivalence problem \cite{williams_shift_equiv}, constituting one of the central topics in symbolic dynamics. SFTs arise naturally as a class of dynamical systems via Markov partitions of hyperbolic dynamical systems, and their conjugacy problem represents a fundamental subject of inquiry in dynamical systems.
The conjugacy problem for SFTs was reduced by Williams \cite{williams_shift_equiv} to the strong shift equivalence problem, which is known to be closely related to several fields outside dynamical systems, such as topological quantum field theory \cite{TQFT}. The automorphism group $\aut(X)$ of an SFT $X$, consisting of all shift-commuting homeomorphisms on $X$, describes the internal symmetries of the dynamical system and serves as an invariant of its strong shift equivalence class. By the Curtis-Lyndon-Hedlund theorem, $\aut(X)$ can be identified with the group of all reversible cellular automata on $X$, thereby possessing connections to computer science and coding theory. Consequently, elucidating the algebraic structure of $\aut(X)$ is motivated by both the dynamical perspective and the viewpoint of interdisciplinary applications. \\
\indent In recent years, it has been demonstrated through the works of \cite{cyr_franks_kra_petite, cyr_kra2,cyr_kra1,donoso,pavlov_schmieding} and others that under conditions of low word complexity or zero entropy, the automorphism group $\aut(X)$ of a subshift $X$ is highly constrained and exhibits a relatively simple structure. In contrast, for SFTs, which serve as prototypical systems with positive entropy, it is known from \cite{Hedlund,boyle_lind_rudolph} that the automorphism group possesses an algebraically unwieldy structure, containing free groups and arbitrary finite groups as subgroups. To unravel the structure of the automorphism groups of SFTs and to understand the conjugacy problem, various algebraic invariants, such as dimension groups \cite{dimensiongroup} and Bowen-Franks groups \cite{bowen_franks}, have been constructed based on shift equivalence \cite{williams_shift_equiv}. Furthermore, by focusing on the properties of permutations induced by automorphisms on periodic orbits, group-theoretic invariants such as the sign homomorphism \cite{BoyleKrieger} have been introduced. Many of these existing invariants ensure computability and tractability through matrix calculations and have been extensively applied to the classification problem of dynamical systems. On the other hand, it is known that when the automorphism group acts on these abelian invariants, a non-negligible kernel arises where the action is trivial. For instance, the internal structure of the inert subgroup, defined as the kernel of the representation into the dimension group, often evades capture by existing invariants. Consequently, non-abelian structures, typified by the commutator subgroup $[\aut(X), \aut(X)]$, are sometimes compressed into this inert subgroup. Therefore, a natural need arises for a framework to directly extract the non-abelian information of the automorphism group hidden within the inert subgroup and the commutator subgroup. \\
\indent The purpose of this paper is to constructively establish an analogue of the theory of Galois extensions in field theory and covering spaces in topology on irreducible SFTs, and to introduce algebraic invariants for capturing the non-abelian structure of the automorphism group. It has long been known that there exist analogies between dynamical systems and number theory, typified by the correspondence between periodic orbits and prime ideals. In the book by Parry and Pollicott \cite{parry1990zeta}, the existence of Galois-theoretic phenomena in dynamical systems is suggested; however, their Galois-theoretic approach was primarily directed towards analyzing zeta functions and the distribution of orbits rather than establishing a purely algebraic theoretical framework. Thus, building upon the works of \cite{Nasu,Jung}, this paper demonstrates that the fundamental theorem of Galois theory holds by regarding constant-to-one maps satisfying certain conditions as analogues of finite Galois extensions in SFTs. Furthermore, based on this framework, we construct conjugacy invariants such as the absolute Galois group and Galois cohomology, and by considering representations of the automorphism group $\aut(X)$ into them, we aim to capture the non-abelian structure of the automorphism group.

\subsection{Main Results}
Let $X$ and $Y$ be irreducible SFTs. A map $\vp: Y \to X$ is called a factor map if it is a continuous and shift-commuting surjection. A factor map $\vp$ is said to be unramified if it is finite-to-one and constant-to-one, and we denote the cardinality of its fibers by $\deg(\vp)$. We define the relative automorphism group $\aut(Y/X,\vp)$, a subgroup of $\aut(Y)$, by
\[
\aut(Y/X,\vp)=\{\varphi\in\aut(Y)\mid \vp\circ\varphi=\vp\}.
\]
When $\#\aut(Y/X,\vp)=\deg(\vp)$ holds, $\vp$ is called a Galois factor, and we denote $\aut(Y/X,\vp)$ by $\gal(Y/X,\vp)$, which is referred to as the Galois group of $\vp$. An intermediate factor of a Galois factor $\vp: Y \to X$ is a triple $(Z, \alpha, \beta)$ consisting of an irreducible SFT $Z$ and unramified factor maps $\alpha$ and $\beta$ satisfying $\beta \circ \alpha = \vp$. Let $\mathcal{I}(Y/X,\vp)$ denote the set of all topological conjugacy classes of intermediate factors, and $\mathcal{S}(\gal(Y/X,\vp))$ denote the set of all subgroups of $\gal(Y/X,\vp)$. The following is our first main theorem, which is analogous to the fundamental theorem of Galois theory.

\begin{theorem}\label{mainthm1}
    There is a one-to-one correspondence between the two sets $\mathcal{I}(Y/X,\vp)$ and $\mathcal{S}(\gal(Y/X,\vp))$. 
\end{theorem}

This theorem is proved in Subsection~\ref{subsection_pf_of_fundamental_thm}. In Subsection~\ref{subsection_pf_of_fundamental_thm}, we also show that other properties holding in the Galois theory of fields and covering spaces analogously hold here, such as the necessary and sufficient conditions for a subgroup of the Galois group to be normal, and conditions regarding subgroup inclusions. In Section~\ref{subsection_ex_of_gal}, we provide several examples of determining the structure of Galois factors using Theorem~\ref{mainthm1}.\\
\indent Let $x_0\in X$, and consider the pointed space $(X,x_0)$. We denote the set of topological conjugacy classes of pointed Galois factors over $(X,x_0)$ by $\mathcal{I}(X,x_0)$. The set $\mathcal{I}(X,x_0)$ naturally admits a partial order, forming a directed set. For each $i\in \mathcal{I}(X,x_0)$, let $G_i$ be its Galois group, and we define its projective limit as $G_{X,x_0}=\varprojlim G_i$. It is shown in Section~\ref{section_Abs_galois_group} that $G_{X,x_0}$ is independent of the base point $x_0$ up to isomorphism, and we simply denote it by $G_X$. 
The absolute Galois group $G_X$ is equipped with a natural profinite topology. A profinite group is a topological group described as the projective limit of finite groups endowed with the discrete topology. Viewed as topological spaces, profinite spaces are contained in the class of spectral spaces in algebraic geometry; in particular, Hausdorff spectral spaces coincide with profinite spaces. We can consider the first cohomology group of a profinite group $H^1(G_X,\F_p)$, where $p$ is a prime number and $\F_p=\Z/p\Z$. 
The following is our second main theorem.

\begin{theorem}\label{mainthm2}
    The absolute Galois group $G_X$ and its first cohomology $H^1(G_X,\F_p)$ are topological conjugacy invariants. 
\end{theorem}

\indent While the absolute Galois group and its cohomology can be used as conjugacy invariants, they also provide a stage for representing $\aut(X)$. 
Indeed, the Galois cohomology $H^1(G_X,\F_p)$ naturally possesses the structure of an $\F_p$-vector space, and a representation into the automorphism group, $\Upsilon_p:\aut(X)\to \mathrm{GL}(H^1(G_X,\F_p))$, is constructed in Section~\ref{section_gal_cohom}. Our third main theorem concerns how much information this representation can capture. 
Let $m\ge1$ be an integer, and let $\mathcal{P}_m$ denote the set of all $m$-periodic orbits. An automorphism $f$ on $X$ can be regarded as a permutation on $\mathcal{P}_m$. We define a homomorphism $\mathrm{OS}_m:\aut(X)\to\{-1,1\}$ that associates an automorphism with the sign of this permutation. This is called the sign homomorphism, pioneered by the works of Boyle and Krieger \cite{BoyleKrieger} and others. The third main theorem below asserts that the representation $\Upsilon_p$ contains the information of the sign homomorphism.

\begin{theorem}
    There exists a subspace $W \subseteq H^1(G_X,\F_p)$ invariant under the representation such that for any $f\in\aut(X)$, the following hold:
    \begin{enumerate}
        \item The matrix representation of 
        \[
        \overline{\Upsilon_p}(f)_m:H^1(G_X,\F_p)/W\to H^1(G_X,\F_p)/W
        \] with respect to a certain basis of $H^1(G_X,\F_p)/W$ is exactly the permutation matrix of the permutation on $\mathcal{P}_m$ induced by $f$. 
        \item The equality
        \[
        \det(\overline{\Upsilon_p}(f)_m) \equiv \mathrm{OS}_m(f) \pmod p
        \]
        holds.
        \item The characteristic polynomial, determinant, and trace of $\overline{\Upsilon_p}(\cdot)_m:\aut(X)\to \mathrm{GL}(H^1(G_X,\F_p)/W)$ are conjugacy invariants of automorphisms. 
    \end{enumerate}
    Here, $\overline{\Upsilon_p}(\cdot)_m$ is the quotient representation of $\Upsilon_p$ with respect to $W$. 
\end{theorem}

This theorem is described in Subsection~\ref{permu_of_per_orb} and Subsection~\ref{subsection_application_conj}. In these subsections, the subspace $W$ is explicitly constructed. Furthermore, Subsection~\ref{subsection_application_conj} provides an example of distinguishing elements of the commutator subgroup by using the trace of $\overline{\Upsilon_p}(\cdot)_m$.

\subsection{Organization of the paper}
In Section~\ref{section_pre}, we recall some facts and prove some lemmas that follow immediately from general theory. In Section~\ref{section_three}, we define Galois factors and prove their basic properties. In Section~\ref{section_fundamental}, we prove our first main result, the fundamental theorem of Galois theory, and provide examples of determining covering structures using this theorem. In Section~\ref{section_Abs_galois_group}, we construct the absolute Galois group and show that it is a conjugacy invariant. Finally, in Section~\ref{section_gal_cohom}, we define Galois cohomology and the representation $\Upsilon_p$ of $\aut(X)$. Our third main theorem is proved in Subsections~\ref{permu_of_per_orb} and \ref{subsection_application_conj}.

\section{Preliminaries}\label{section_pre}
In this section, we recall some facts, basic concepts, and notation. 

\subsection{Symbolic Dynamics}
In this subsection, we recall basic concepts and facts from symbolic dynamics. Most of the material in this subsection is based on \cite[Chapters 1, 6, and 8]{introduction}. \\
\indent Let $\AA$ be a finite set and let $n=\# \AA$. Let $\Sigma_n=\AA^\Z$ be the full shift over the alphabet $\AA$. We denote the shift map by $\sigma_{\Sigma_n} :\AA^\Z\to\AA^\Z$. The space $\Sigma_n$ is naturally equipped with a metric $d$ defined by
\[
d(x, y) = \begin{cases} 0 & (x = y) \\ 2^{-k} & (x \neq y) \end{cases}
,\quad x,y\in \Sigma_n,
\]
where $k=\min\{|i|: x_i\ne y_i\}$.
A closed and $\sigma_{\Sigma_n}$-invariant subset $X\subseteq \Sigma_n$ is called a shift space. We denote the restriction of $\sigma_{\Sigma_n}$ to $X$ by $\sigma_X:X\to X$, and the restriction of the metric $d$ to $X$ by $d_X$. Thus, $(X,\sigma_X)$ becomes a discrete dynamical system. When it is clear from the context, we simply write $\sigma$ and $d$ for $\sigma_X$ and $d_X$, respectively. Since $\AA^\Z$ is compact and a closed subset of a compact space is compact, any shift space is compact.\\
\indent From now on, unless specified otherwise, we assume that $X$ is a shift space. A finite sequence of elements in $\AA$ is called a block or a word. We denote the length of a block $w$ by $|w|$. Let $\BB_n(X)$ denote the set of all blocks of length $n$ occurring in points of $X$, and define $\BB(X)=\bigcup_{n=0}^\infty \BB_n(X)$. The empty word is denoted by $\varepsilon$. For $x\in X$ and $i,j\in\Z$, we write
\[
x_{[i,j]}=x_ix_{i+1}\cdots x_j\in \BB_{j-i+1}(X).
\]
If $i>j$, we set $x_{[i,j]}=\varepsilon$. In accordance with this notation, we sometimes write $x_{[i]}$ to emphasize $x_i$. For $w\in \BB(X)$ and $i\in \Z$, the cylinder set of $w$ at position $i$ is defined by $[w]_i=\{ x \in X \mid x_{[i, i+|w|-1]}=w\}$. Then, the collection $\{[w]_i \mid w\in \BB(X), i\in\Z\}$ forms a basis for the topology of $X$. \\
\indent Let $\FF\subset \BB(\AA^\N)$. We denote the shift space with the set of forbidden words $\FF$ by $\mathsf{X}_\FF$. A shift space $X$ is called a shift of finite type if there exists a finite set of forbidden words $\FF$ (i.e., $\#\FF<\infty$) such that $X=\mathsf{X}_{\FF}$. In some literature, this is also called a subshift of finite type. We will abbreviate ``shift of finite type'' as SFT. \\
\indent A shift space $X$ is said to be irreducible if for any $u, v \in \BB(X)$, there exists a word $w \in \BB(X)$ with $w \ne \varepsilon$ such that $uwv \in \BB(X)$. The following lemma can be found in \cite[Example 6.3.2]{introduction}.
\begin{lemma}\label{irred_toptrans_equiv}
    A shift space is irreducible if and only if it is topologically transitive.
\end{lemma}
Here, a shift space $X$ is said to be topologically transitive if for any non-empty open sets $U, V \subseteq X$, there exists $n \in \N$ such that $\sigma^n(U) \cap V \ne \emptyset$. This is equivalent to the condition that for any non-empty open sets $U, V \subseteq X$, there exists $n \in \N$ such that $\sigma^{-n}(U) \cap V \ne \emptyset$. Let $O^+(x)$ denote the forward orbit of $x \in X$.
Then, by Lemma~\ref{irred_toptrans_equiv}, we obtain the following lemma.

\begin{lemma}\label{toptransitive}
If $X$ is irreducible, then $X$ has a point with a dense forward orbit.
\end{lemma}

\indent When two shift spaces $X$ and $Y$ satisfy $X\subseteq Y$, $X$ is called a subshift of $Y$. In this context, the following holds.
\begin{lemma}\label{clopen_subshift}
    Let $Y$ be an irreducible shift space, and let $X \subseteq Y$ be a non-empty subshift. If $X$ is open, then $X=Y$.
\end{lemma}

\begin{proof}
    Since $Y$ is an irreducible shift space, by Lemma~\ref{toptransitive}, there exists a point $y \in Y$ with a dense forward orbit. As $X$ is a non-empty open set, the density of the orbit implies that there exists an integer $n \ge 1$ such that $\sigma^n(y) \in X$. Since $X$ is a shift space, we have $\sigma(X) \subseteq X$, which means $\sigma^k(y) \in X$ for any $k \ge n$. This implies that $O^+(\sigma^n(y)) \subseteq X$. Since $y$ has a dense orbit, the orbit of $\sigma^n(y)$ is also dense. Therefore, $\overline{O^+(\sigma^n(y))} = Y$, which yields $\overline{X} = Y$. Furthermore, since $X$ is a subshift, it is closed. Then $X = \overline{X}$ holds. Consequently, we have $X = Y$.
\end{proof}

\indent A map $\varphi: X \to Y$ between shift spaces $X$ and $Y$ is called a homomorphism if it is continuous and shift-commuting (i.e., $\sigma_Y \circ \varphi = \varphi \circ \sigma_X$). By the Curtis-Lyndon-Hedlund theorem, the class of homomorphisms coincides with the class of maps induced by local rules. Because of this, a homomorphism is also referred to as a sliding block map or a sliding block code. Moreover, the map representing the local rule that induces a homomorphism is called a block code, and the homomorphism induced by a block code $\Phi$ is denoted by $\Phi_\infty$. In most of the arguments in this paper, by passing to a higher block presentation, we may assume that the sliding block code is induced by $\Phi:\BB_1(X)\to\BB_1(X)$. For more details on these concepts, see \cite[Chapter 1]{introduction}. A bijective sliding block code is a topological conjugacy. A surjective sliding block code is called a factor map or a factor code. A topological conjugacy from $X$ onto itself is called an automorphism. We denote the set of all automorphisms on $X$ by $\aut(X)$. While it is also common to denote this by $\aut(\sigma_X)$, in this paper we adopt the notation $\aut(X)$ to make the analogy with the theory of covering spaces more transparent. \\
\indent Two points $x, y \in X$ are said to be left-asymptotic if there exists $N \in \Z$ such that $x_i = y_i$ for all $i \le N$. Similarly, $x, y \in X$ are right-asymptotic if there exists $N \in \Z$ such that $x_i = y_i$ for all $i \ge N$. Let $X$ and $Y$ be shift spaces, and let $\varphi: X \to Y$ be a sliding block code. We say that $\varphi$ is right-closing if for any $x, y \in X$, whenever $x$ and $y$ are left-asymptotic and $\varphi(x) = \varphi(y)$, we have $x = y$. We say that $\varphi$ is left-closing if for any $x, y \in X$, whenever $x$ and $y$ are right-asymptotic and $\varphi(x) = \varphi(y)$, we have $x = y$. If $\varphi$ is both right-closing and left-closing, it is called bi-closing. According to \cite[Proposition 8.1.9]{introduction}, $\varphi$ is right-closing if and only if there exists an integer $K$ such that if $\Phi(x_{[-K,K]}) = \Phi(y_{[-K,K]})$ and $x_{[-K,0]} = y_{[-K,0]}$, then $x_1 = y_1$. Similarly, $\varphi$ is left-closing if and only if there exists an integer $K$ such that if $\Phi(x_{[-K,K]}) = \Phi(y_{[-K,K]})$ and $x_{[0,K]} = y_{[0,K]}$, then $x_{-1} = y_{-1}$.

\subsection{Invariant Measures and Non-wandering Points}
In this subsection, we consider a dynamical system $(X, \mathcal{B}(X), \mu, T)$, where $X$ is a compact metric space, $T: X \to X$ is a continuous map, $\mathcal{B}(X)$ is the Borel $\sigma$-algebra on $X$, and $\mu$ is a $T$-invariant Borel probability measure on $X$. We say that $\mu$ has full support if the support of $\mu$ coincides with the entire space $X$, that is, if $\mu(U) > 0$ for any non-empty open set $U$. A point $x \in X$ is called a non-wandering point if for any open neighborhood $U$ of $x$, there exists an integer $n \ge 1$ such that $T^n(U) \cap U \ne \emptyset$. A point that is not non-wandering is called a wandering point. We denote the set of all non-wandering points by $\Omega(T) = \Omega(X, T)$. 

\begin{lemma}\label{non_wandering_full_supp}
    If $\mu$ has full support, every point in $X$ is a non-wandering point. That is, $\Omega(T) = X$. 
\end{lemma}

\begin{proof}
    We prove this by contradiction. Suppose that there exists $x\in X\setminus\Omega(T)$. Then $x$ is a wandering point, which means there is an open neighborhood $U$ of $x$ such that $T^n(U)\cap U=\emptyset$ for all $n\ge1$. On the other hand, since $\mu$ has full support, we have $\mu(U)>0$. However, this contradicts the Poincar\'e recurrence theorem, which states that almost every point in $U$ returns to $U$.
\end{proof}

\subsection{Maps and General Topology}
In this subsection, we recall some basic facts about maps and general topology.
Let $X$ and $Y$ be topological spaces, and $q: X \to Y$ be a surjective map. The map $q$ is called a quotient map if for any subset $U \subseteq Y$, $U$ is open in $Y$ if and only if $q^{-1}(U)$ is open in $X$. Any continuous surjective open map is a quotient map. 

\begin{lemma}[Universal property of quotient maps]\label{universalityofquotientmap}
    Let $q: Y \to X$ be a quotient map. Then, for any topological space $Z$ and any map $g: X \to Z$, $g$ is continuous if and only if $g \circ q$ is continuous.
\end{lemma}

In the following, we state some properties concerning the composition of constant-to-one maps. Here, we assume that the cardinality of each fiber of a constant-to-one map is finite.

\begin{lemma}\label{const_to_one}
    Let $g:A\to B$ and $f:B\to C$ be maps, and set $h=f\circ g:A\to C$. Then the following hold:
\begin{enumerate}
\item If $f$ and $g$ are constant-to-one, then $h$ is also constant-to-one.
\item If $g$ and $h$ are constant-to-one, then $f$ is also constant-to-one.
\end{enumerate}
\end{lemma}

\begin{proof}
    For any $c \in C$, we have
    \[
    h^{-1}(c) = \{a\in A\mid f(g(a))=c\} = \bigsqcup_{b\in f^{-1}(c)} g^{-1}(b).
    \]
    The claims follow from this.
\end{proof}

\subsection{Group Actions}
    In this subsection, we recall basic concepts concerning group actions. Consider the action of a group $G$ on a set $X$. For $x\in X$, we denote its isotropy subgroup (also called the stabilizer) by
    \[
    I_G(x)=\{g\in G\mid gx=x\}.
    \]
    We also denote the fixed-point set by
    \[
    X^G=\{x\in X\mid gx=x \ (g\in G)\}.
    \]
    Furthermore, we write the orbit of $x\in X$ under $G$ as
    \[
    Gx =G\cdot x= \{ gx \mid g \in G \}.
    \]
    A group action $G\curvearrowright X$ is said to be free if $I_G(x)=\{e\}$ for all $x\in X$. A group action $G\curvearrowright X$ is said to be transitive if for any $x,y\in X$, there exists $g\in G$ such that $y=gx$.
    Regarding the order, the following result, widely known as the orbit-stabilizer theorem, holds.
\begin{lemma}\label{orbitstabilizer}
    Let $G$ be a finite group. Then for any $x\in X$, we have $|G|=|Gx||I_G(x)|$.
\end{lemma}

Let $Y$ be an irreducible SFT, and let $H$ be a finite subgroup of $\aut(Y)$. For $y, y' \in Y$, we write $y \sim_H y'$ if and only if there exists $h \in H$ such that $y' = h(y)$. This defines an equivalence relation $\sim_H$ on $Y$. We denote the quotient space (or orbit space) by $Z = Y/\sim_H = Y/H$. Let $[y]_H$ denote the equivalence class of $y \in Y$, and let $q_H: Y \to Z$, $y \mapsto [y]_H$ be the natural projection. We define a map $\sigma^*_Z: Z \to Z$ by $\sigma^*_Z([y]_H) = [\sigma_Y(y)]_H$. By construction, this is well-defined. This naturally induces a dynamical system $(Z, \sigma^*_Z)$. The following lemma is essentially folklore, but we provide its proof in the Appendix to keep the paper self-contained.

\begin{lemma}\label{orbitspaceSFT}
    Let $Y$ be an irreducible SFT, and let $H$ be a finite subgroup of $\aut(Y)$ acting freely on $Y$. Then, the dynamical system $(Z=Y/H, \sigma^*_Z)$ naturally induced by the action is topologically conjugate to an irreducible SFT. 
\end{lemma}

\subsection{Graph Theory and SFTs}
In this subsection, we recall the relationship between graph theory and SFTs. Some of the material in this subsection is based on \cite[Chapters 2]{introduction}. 
Let $G=(V,E)$ be a finite directed graph with vertex set $V$ and edge set $E$. In this paper, directed graphs are allowed to have multiple edges and self-loops. For an edge $e$, we denote its initial and terminal vertices by $i(e)$ and $t(e)$, respectively. For any $v\in V$, let $E_v(G)$ denote the set of all outgoing edges from $v$, and $E^v(G)$ denote the set of all incoming edges to $v$. Let $G'=(V',E')$ be another directed graph. A map $\varphi:V\sqcup E\to V'\sqcup E'$ satisfying $\varphi(V)\subseteq V'$ and $\varphi(E)\subseteq E'$ is called a graph homomorphism if it satisfies $i(\varphi(e))=\varphi(i(e))$ and $t(\varphi(e))=\varphi(t(e))$ for any $e\in E$. We denote such a graph homomorphism by $\varphi:G\to G'$. A bijective homomorphism is called a graph isomorphism, and the set of all automorphisms of a graph is denoted by $\aut(G)$. \\
\indent We denote the edge shift defined by a graph $G$ by $\mathsf{X}_G$. Any edge shift is an SFT, and conversely, every SFT is topologically conjugate to some edge shift. In most of the arguments in this paper, we may assume that an SFT is an edge shift defined by an essential graph $G$. From now on, we implicitly assume $G$ to be essential when necessary.
A graph $G$ is said to be strongly connected if for any vertices $v, w \in V$, there exists a path from $v$ to $w$. A graph is strongly connected if and only if its edge shift is irreducible. In what follows, we recall the strongly connected component (SCC) decomposition.

\begin{definition}\label{SCC}
    Let $G=(V,E)$ be a finite directed graph. For vertices $v, v' \in V$, we define an equivalence relation $\sim_V$ on $V$ by declaring $v \sim_V v'$ if there is a path from $v$ to $v'$ and a path from $v'$ to $v$. Here, we assume the existence of a path of length $0$ from any vertex $v$ to itself. This equivalence relation partitions the vertex set $V$ into a finite number of equivalence classes. That is, for some $r \ge 1$, we have the decomposition $V = C_1 \sqcup \cdots \sqcup C_r$. Along with this vertex partition, the edge set $E$ is also partitioned. Specifically, for each $1 \le j \le r$, let $D_j = \{e \in E \mid i(e), t(e) \in C_j\}$, and let $T = \{e \in E \mid i(e) \in C_k, t(e) \in C_l, k \ne l\}$. Then $E$ decomposes as $E = D_1 \sqcup \cdots \sqcup D_r \sqcup T$. In this case, for each $1 \le j \le r$, the subgraph $(C_j, D_j)$ is strongly connected, and thus this decomposition is called the SCC decomposition of $G$. By taking the quotient set $V/\sim_V$ as the vertex set and $T$ as the edge set, we obtain a finite directed graph $(V/\sim_V, T)$, which has no cycles. When $T = \emptyset$, there are no edges connecting different vertex components, meaning $G$ simply decomposes as $G = G_1 \sqcup \cdots \sqcup G_r$, where $G_j = (C_j, D_j)$. 
    \end{definition}

\subsection{Profinite Groups and Group Cohomology}\label{section_profinite_cohom}
    In this subsection, we recall the concept of profinite groups and their cohomology. We denote the identity element of a group $G$ by $1_G$.
    A topological group $G$ is called a profinite group if it is isomorphic to the projective limit of a family of finite groups. Here, the finite groups are endowed with the discrete topology. In this subsection, let $\{G_i\}_{i\in\Lambda}$ be a family of finite groups indexed by a directed set $\Lambda$, and let $(\{G_i\}_{i\in\Lambda}, \{p_{ij}: G_j \to G_i\}_{i \le j})$ be a projective system. Let
    \[
    G=\varprojlim G_i=\left\{ (g_i)_{i\in \Lambda}\in \prod_{i\in \Lambda}G_i\ \Big| \ g_i=p_{ij}(g_j)\text{ for any }i\le j\right\}.
    \]
    For $i\in \Lambda$, we denote the projection onto the $i$-th coordinate by $\mathrm{pr}_i: G \to G_i$, $g=(g_j)_{j\in \Lambda} \mapsto g_i$.
    The group $G=\varprojlim G_i$ is equipped with the subspace topology induced by the product topology. That is, for each $i\in\Lambda$, if we define an open normal subgroup $N_i$ by
    \[
    N_i=\mathrm{Ker}(\mathrm{pr}_i)=\{g=(g_j)_{j\in \Lambda}\in G\mid g_i=1_{G_i}\},
    \]
    then the collection $\{N_i\}_{i\in\Lambda}$ forms a fundamental system of neighborhoods of $1_G$. Therefore, for any $h=(h_i)_{i\in\Lambda}\in G$, the family
    \[
    \{hN_i\}_{i\in\Lambda}=\{\{g=(g_j)_{j\in \Lambda}\in G\mid g_i=h_i\}\}_{i\in\Lambda}
    \]
    is a fundamental system of neighborhoods of $h$. Consequently, $\{gN_i\}_{g\in G, i\in \Lambda}$ forms a basis for the topology of the topological group $G$. A profinite group is compact, Hausdorff, and totally disconnected.\\
    \indent Let $A$ be a topological $G$-module (i.e., $A$ is equipped with a continuous left $G$-action). By general theory, we can define the cohomology group $H^1(G, A)$ with coefficients in $A$. In this paper, we restrict our attention to the case where $A$ is equipped with the discrete topology and the action of $G$ on $A$ is trivial. Under these assumptions, the following holds:
    \begin{lemma}\label{cohom_homconti}
        Assume that the action of $G$ on $A$ is trivial. Then
        \[
        H^1(G,A)= \mathrm{Hom}_{\mathrm{cont}}(G,A),
        \]
        where $\mathrm{Hom}_{\mathrm{cont}}(G,A)$ denotes the group of all continuous homomorphisms from $G$ to $A$.
    \end{lemma}

\section{Definition and Basic Properties of Galois Groups}\label{section_three}

\subsection{Unramified Factor Maps}\label{subsection}
Based on previous results on constant-to-one maps in SFTs, this subsection shows that a $d$-to-1 map acts as a local homeomorphism. As a result, the unramified factor maps defined below can be regarded as analogues of covering spaces in topology and field extensions in field theory.

\begin{definition}
    Let $X$ and $Y$ be irreducible SFTs. A factor map $\vp: Y \to X$ is said to be unramified if there exists an integer $d \ge 1$ such that $|\vp^{-1}(x)| = d$ for all $x \in X$. The integer $d$ is called the degree of $\vp$, and is denoted by $\deg(\vp) = d$. 
\end{definition}

The following two facts are obtained from the results of Nasu~\cite{Nasu} and Jung~\cite{Jung}.

\begin{lemma}\label{openmap}
    Let $X$ and $Y$ be irreducible SFTs, and let $\vp: Y \to X$ be an unramified factor map. Then $\vp$ is an open map.
\end{lemma}

\begin{proof}
    This follows immediately from Nasu~\cite[Proposition 1.1, Proposition 1.3, and Corollary 6.6]{Nasu}.
\end{proof}

\begin{lemma}\label{biclosing}
    Let $X$ and $Y$ be irreducible SFTs, and let $\vp: Y \to X$ be an unramified factor map. Then $\vp$ is bi-closing.
\end{lemma}

\begin{proof}
    This follows immediately from Jung~\cite[Theorem 1.1]{Jung} and Lemma~\ref{openmap}.
\end{proof}

We prove the local injectivity of unramified factor maps.

\begin{lemma}\label{separationlemma}
    Let $X$ and $Y$ be irreducible SFTs, and let $\vp: Y \to X$ be an unramified factor map. Then there exists a $\delta > 0$ such that for any $x, y \in Y$, if $\vp(x)=\vp(y)$ and $d(x,y)<\delta$, then $x=y$.
\end{lemma}

\begin{proof}
    By Lemma~\ref{biclosing}, $\vp$ is bi-closing. 
    Therefore, there exists an integer $K_1$ such that for any $x, y \in X$, if $\Phi(x_{[-K_1,K_1]})=\Phi(y_{[-K_1,K_1]})$ and $x_{[-K_1,0]}=y_{[-K_1,0]}$, then $x_1 = y_1$, where $\Phi$ is the block code inducing $\vp$.
    Similarly, there exists an integer $K_2$ such that for any $x, y \in X$, if $\Phi(x_{[-K_2,K_2]})=\Phi(y_{[-K_2,K_2]})$ and $x_{[0,K_2]}=y_{[0,K_2]}$, then $x_{-1} = y_{-1}$. Let $K = \max\{K_1, K_2\}$ and $\delta = 2^{-K}$. 
    Fix $x, y \in Y$ satisfying $d(x, y)<\delta$ and $\vp(x)=\vp(y)$. What we want to show is $x_m = y_m$ and $x_{-m} = y_{-m}$ for all $m = 0, 1, 2, \dots$. Therefore, we will prove this by induction on $m$.
    First, since $d(x, y)<\delta=2^{-K}$, the definition of the metric on $X$ implies that $x_m=y_m$ and $x_{-m}=y_{-m}$ for all $m\le K$. Next, suppose $m > K$. By the induction hypothesis, we have $x_{[-m+1, m-1]}=y_{[-m+1, m-1]}$.
    Thus, $(\sigma^{m-1}(x))_{[-2m+2, 0]}=(\sigma^{m-1}(y))_{[-2m+2, 0]}$, and in particular, 
    \[
    (\sigma^{m-1}(x))_{[-K_1,0]}=(\sigma^{m-1}(y))_{[-K_1,0]}
    \]
    holds. Moreover, since $\vp$ commutes with the shift and $\vp(x)=\vp(y)$, we obtain $\vp(\sigma^{m-1}(x))=\vp(\sigma^{m-1}(y))$, which yields
    \[
    \Phi((\sigma^{m-1}(x))_{[-K_1, K_1]}) = \Phi((\sigma^{m-1}(y))_{[-K_1, K_1]}).
    \]
    By the definition of $K_1$, it follows that $(\sigma^{m-1}(x))_{[1]}=(\sigma^{m-1}(y))_{[1]}$, which implies $x_m = y_m$. By a similar argument, we also obtain $x_{-m} = y_{-m}$. This completes the induction.
\end{proof}

By restating Lemma \ref{separationlemma} using the definition of the metric on the shift space, we obtain the following.

\begin{lemma}\label{separationlemmavar2}
     Let $X$ and $Y$ be irreducible SFTs, and let $\vp: Y \to X$ be an unramified factor map. Then there exists an integer $N > 0$ such that for any $x, y \in Y$, if $\vp(x)=\vp(y)$ and $x_{[-N,N]} = y_{[-N,N]}$, then $x = y$.
\end{lemma}

\subsection{Lifts}

\begin{definition}
    Let $X, Y,$ and $Z$ be irreducible SFTs. Let $\vp: Y \to X$ be an unramified factor map, and let $\varphi: Z \to X$ be a sliding block code. A sliding block code $\tilde{\varphi}:Z\to Y$ that satisfies $\vp\circ\tilde{\varphi}=\varphi$ is said to be a lift of $\varphi$ (with respect to $\vp$).
\[
\begin{tikzcd}[row sep=large, column sep=large]
     & Y \arrow[d, "\vp"] \\
    Z \arrow[r, "\varphi"'] \arrow[ru, "\tilde{\varphi}"] & X
\end{tikzcd}
\]
   
\end{definition}

As the uniqueness of lifts plays a key role in covering space theory, a similar uniqueness property holds for SFTs.

\begin{theorem}[Uniqueness of lifts]\label{liftunique}
Let $X, Y,$ and $Z$ be irreducible SFTs, $\vp: Y \to X$ be an unramified factor map, and $\varphi:Z \to X$ be a sliding block code. Let $\tilde{\varphi}_1, \tilde{\varphi}_2: Z \to Y$ be two lifts of $\varphi$. If there exists a $z_0 \in Z$ such that $\tilde{\varphi}_1(z_0) = \tilde{\varphi}_2(z_0)$, then $\tilde{\varphi}_1=\tilde{\varphi}_2$.
\end{theorem}

\begin{proof}
    Set $E=\{z \in Z \mid \tilde{\varphi}_1(z)=\tilde{\varphi}_2(z)\} \subseteq Z$. Since $z_0 \in E$, the set $E$ is non-empty. It suffices to show that $E = Z$.
    Since $Y$ is a Hausdorff space and $\tilde{\varphi}_1, \tilde{\varphi}_2$ are continuous, $E$ is a closed subset of $Z$. Next, we show that $E$ is open in $Z$. By Lemma \ref{separationlemma}, there exists a $\delta > 0$ such that for any $x, y \in Y$, if $\vp(x) = \vp(y)$ and $d(x, y) < \delta$ then $x = y$. Fix $z \in E$, and set $x = \tilde{\varphi}_1(z) = \tilde{\varphi}_2(z)$. By the continuity of $\tilde{\varphi}_1$ and $\tilde{\varphi}_2$, there exists a $\delta_0 > 0$ such that for any $w \in Z$ satisfying $d(w, z) < \delta_0$, we have $d(\tilde{\varphi}_1(w), x) < \delta/2$ and $d(\tilde{\varphi}_2(w), x) < \delta/2$.

    First, we show that $B(z,\delta_0) \subseteq E$. Fix $w \in B(z,\delta_0)$, where $B(z,\delta_0)$ denotes the open ball. By the triangle inequality, we have
    \[
    d(\tilde{\varphi}_1(w), \tilde{\varphi}_2(w)) \le d(\tilde{\varphi}_1(w), x) + d(x, \tilde{\varphi}_2(w)) < \frac{\delta}{2} + \frac{\delta}{2}=\delta.
    \]
    Moreover, since both $\tilde{\varphi}_1$ and $\tilde{\varphi}_2$ are lifts of $\varphi$, we have $\vp(\tilde{\varphi}_1(w)) = \vp(\tilde{\varphi}_2(w))$. By the choice of $\delta$, this implies $\tilde{\varphi}_1(w) = \tilde{\varphi}_2(w)$, and thus $w \in E$. This shows that $B(z,\delta_0) \subseteq E$, and hence $E$ is an open set. It is clear from the definition that the set $E$ is shift-invariant. Thus far, we have shown that $E$ is a clopen, shift-invariant set. From Lemma~\ref{toptransitive},  there exists a point $a \in Z$ which has a dense forward orbit. Because $E$ is an open set, there exists an integer $N$ such that $\sigma^N(a) \in E$. By the shift-invariance of $E$, we have
    $O^+(\sigma^N(a)) \subseteq E$. Since $\overline{O^+(\sigma^N(a))} = Z$ holds, we have $\overline{E} = Z$. Finally, since $E$ is a closed set, we conclude that $E=\overline{E}=Z$.
\end{proof}

\begin{remark}
In the proof of the uniqueness of lifts, it is observed that topological transitivity in dynamical systems plays the same role as connectedness in topological spaces.
\end{remark}

\subsection{Galois factor and Galois group}
In this subsection, we define Galois factors and their Galois groups. A Galois factor corresponds to a Galois extension in field theory and a Galois covering in covering space theory.

\begin{definition}
    Let $X$ and $Y$ be irreducible SFTs, and let $\vp: Y \to X$ be an unramified factor map. We define the relative automorphism group $\aut(Y/X,\vp)$ with respect to $\vp$ by
    \[
\aut(Y/X,\vp)=\{\varphi\in\aut(Y) \mid \vp\circ \varphi=\vp\}.
    \]
    If no confusion arises, we simply write $\aut(Y/X)$.
\end{definition}

    The relative automorphism group $\aut(Y/X,\vp)$ acts naturally on the fiber $\vp^{-1}(x)$ over each point $x \in X$. Specifically, for $a \in \vp^{-1}(x)$ and $\varphi \in \aut(Y/X,\vp)$, we set $\varphi \cdot a = \varphi(a)$. Since $x = \vp(a) = \vp(\varphi(a))$, we have $\varphi(a) \in \vp^{-1}(x)$, which implies that this action is well-defined. Furthermore, it naturally acts on $Y$ as well. \\
\indent In the following, we prove several propositions necessary for the definition of Galois factors.

\begin{lemma}\label{rigidity}
    Let $X$ and $Y$ be irreducible SFTs, $\vp: Y \to X$ be an unramified factor map, and $\varphi \in \aut(Y/X, \vp)$. If $\varphi(y) = y$ for some $y \in Y$, then $\varphi=\id_Y$.
\end{lemma}

\begin{proof}
    Since both $\varphi$ and $\id_Y$ are lifts of $\vp$ (with respect to $\vp$), the result follows by applying Theorem~\ref{liftunique} with $Z = Y$, $\tilde{\varphi}_1 = \varphi$, and $\tilde{\varphi}_2 = \id_Y$.
\end{proof}

\begin{corollary}\label{freeaction}
    Let $X$ and $Y$ be irreducible SFTs, and let $\vp: Y \to X$ be an unramified factor map. Then the natural action of $\aut(Y/X,\vp)$ on the fiber $\vp^{-1}(x)$ over each point $x \in X$ is free. Furthermore, its action on $Y$ is also free.
\end{corollary}

\begin{corollary}\label{howmanyorbit}
    Let $X$ and $Y$ be irreducible SFTs, $\vp: Y \to X$ be an unramified factor map, and $x \in X$. Then, for any $y \in \vp^{-1}(x)$, we have $|\aut(Y/X,\vp)| = |\aut(Y/X,\vp)\cdot y|$, where $\aut(Y/X,\vp)\cdot y$ denotes the orbit of $y$ under the action.
\end{corollary}

\begin{proof}
    Fix $y \in \vp^{-1}(x)$. Since the action is free by Corollary~\ref{freeaction}, the stabilizer subgroup is $I_{\aut(Y/X,\vp)}(y) = \{\id_Y\}$. Thus, $|I_{\aut(Y/X,\vp)}(y)|=1$. By Lemma~\ref{orbitstabilizer}, we have
    \[
    |\aut(Y/X,\vp)|=|\aut(Y/X,\vp)\cdot y||I_{\aut(Y/X,\vp)}(y)|=|\aut(Y/X,\vp)\cdot y|,
    \]
    which completes the proof.
\end{proof}

\begin{proposition}\label{cj_le_deg}
    Let $X$ and $Y$ be irreducible SFTs, and let $\vp: Y \to X$ be an unramified factor map. Then $\aut(Y/X,\vp)$ is a finite group, and $|\aut(Y/X,\vp)| \le \deg(\vp)$.
\end{proposition}

\begin{proof}
    Let $y \in Y$ and let $x = \vp(y)$. Define a map $f_y: \aut(Y/X,\vp) \to \vp^{-1}(x)$ by $\varphi \mapsto \varphi(y)$. Since $x = \vp(y) = \vp(\varphi(y))$, we have $\varphi(y) \in \vp^{-1}(x)$, which implies that $f_y$ is well-defined. First, we show that $f_y$ is injective. Suppose that $\varphi$ and $\psi$ satisfy $f_y(\varphi) = f_y(\psi)$. That is, $\varphi(y) = \psi(y)$. It follows that $(\psi^{-1} \circ \varphi)(y) = y$. By Lemma~\ref{rigidity}, we have $\psi^{-1} \circ \varphi = \id_Y$, and hence $\varphi = \psi$. This means that $f_y$ is injective. Therefore, $|\aut(Y/X,\vp)| \le |\vp^{-1}(x)| = \deg(\vp)$.
\end{proof}

\begin{definition}
    Let $X$ and $Y$ be irreducible SFTs, and let $\vp: Y \to X$ be a factor map. We say that $(Y, \vp)$ is a Galois factor of $X$, or simply that $\vp$ is a Galois factor, if $\vp$ is unramified and $|\aut(Y/X, \vp)|=\deg(\vp)$. When $\vp$ is a Galois factor, we denote $\aut(Y/X, \vp)$ by $\gal(Y/X, \vp)$. If no confusion arises, we simply write $\gal(Y/X)$.
\end{definition}

We state the characterization theorem for Galois factors.

\begin{proposition}\label{transitiveaction}
    Let $X$ and $Y$ be irreducible SFTs, and let $\vp: Y \to X$ be an unramified factor map. Then the following are equivalent:
    \begin{enumerate}
        \item $\vp$ is a Galois factor.
        \item For any $x \in X$, $\aut(Y/X)$ acts transitively on the fiber $\vp^{-1}(x)$.
    \end{enumerate}
\end{proposition}

\begin{proof}
    Set $G = \aut(Y/X)$. Suppose that $\vp$ is a Galois factor. Fix $x \in X$, and  $y, y' \in \vp^{-1}(x)$. By Corollary~\ref{howmanyorbit}, we have $|G \cdot y| = |G| = \deg(\vp) = |\vp^{-1}(x)|$. Since $G \cdot y \subseteq \vp^{-1}(x)$ and both sets are finite with the same cardinality, we conclude that $G \cdot y = \vp^{-1}(x)$. Thus, there exists a $\varphi \in G$ such that $\varphi(y) = y'$. This means that the action is transitive.
    Conversely, suppose that for any $x \in X$, the group $\aut(Y/X)$ acts transitively on the fiber $\vp^{-1}(x)$. Let $x \in X$ and $y \in \vp^{-1}(x)$. Since the action is transitive, we have $G \cdot y = \vp^{-1}(x)$, and hence $|G \cdot y| = |\vp^{-1}(x)|$. By Corollary~\ref{howmanyorbit}, we obtain $|G| = |G \cdot y| = |\vp^{-1}(x)|$, which completes the proof.
\end{proof}

Next, we define unramified intermediate factors. This concept is the analogue of intermediate extensions in field theory and intermediate coverings in covering space theory.

\begin{definition}
    Let $X$ and $Y$ be irreducible SFTs, and let $\vp: Y \to X$ be an unramified factor map. An unramified intermediate factor of $\vp$ is a triple $(Z, \alpha, \beta)$ consisting of an irreducible SFT $Z$, an unramified factor map $\alpha: Y \to Z$, and a sliding block code $\beta: Z \to X$ such that $\beta \circ \alpha = \vp$. We sometimes simply refer to $Z$ itself as an intermediate factor. 
    Let $(Z_1, \alpha_1, \beta_1)$ and $(Z_2, \alpha_2, \beta_2)$ be unramified intermediate factors. We say that $(Z_1, \alpha_1, \beta_1)$ and $(Z_2, \alpha_2, \beta_2)$ are topologically conjugate as intermediate factors if there exists a topological conjugacy $\varphi: Z_1 \to Z_2$ of shift spaces such that $\varphi \circ \alpha_1 = \alpha_2$ and $\beta_1 = \beta_2 \circ \varphi$. That is, the following diagram commutes.
    \[
    \begin{tikzcd}
        & Y \arrow[ld, "\alpha_1"'] \arrow[rd, "\alpha_2"] & \\
        Z_1 \arrow[rr, "\varphi"] \arrow[rd, "\beta_1"'] & & Z_2 \arrow[ld, "\beta_2"] \\
        & X &
    \end{tikzcd}
    \]
    A map $\varphi$ is called a conjugacy of intermediate factors.
\end{definition}

\begin{remark}
    In the above setting, we have $\vp = \beta \circ \alpha$, and by Lemma~\ref{const_to_one}, the map $\beta$ is automatically unramified.
\end{remark}

The following is an analogue of the multiplicativity of extension degrees in field theory. This follows trivially from the definition.

\begin{lemma}\label{degreemultiplicativity}
    Let $X$ and $Y$ be irreducible SFTs, $\vp: Y \to X$ be an unramified factor map, and $(Z, \alpha, \beta)$ be an unramified intermediate factor. Then we have $\deg(\alpha)\deg(\beta) = \deg(\vp) = \deg(\beta \circ \alpha)$.
\end{lemma}

\section{Fundamental Theorem of Galois Theory for SFTs}\label{section_fundamental}
\subsection{Proof of the Fundamental Theorem}\label{subsection_pf_of_fundamental_thm}

In this subsection, we provide a proof of the fundamental theorem of Galois theory. Analogous to field theory, where there is a one-to-one correspondence between subgroups of the Galois group and isomorphism classes of intermediate extensions, it will be shown that there is a one-to-one correspondence between subgroups of the Galois group and topological conjugacy classes of intermediate factors.

\begin{definition}\label{def_midfactor_to_gal}
    Let $X$ and $Y$ be irreducible SFTs, $\vp: Y \to X$ be a Galois factor, $G = \gal(Y/X)$, and $(Z, \alpha, \beta)$ be an unramified intermediate factor of $\vp$. We define $H(Z, \alpha, \beta)$ by
    \[
    H(Z,\alpha,\beta)=\{\varphi\in \aut(Y) \mid \alpha\circ \varphi=\alpha\}=\aut(Y/Z,\alpha).
    \]
    When the context is clear, we sometimes simply write $H(Z)$. 
    It is clear that $H(Z)$ is a subgroup of $G$.
     
\end{definition}

\begin{theorem}\label{conjugatesamegroupequiv}
    Let $X$ and $Y$ be irreducible SFTs, $\vp: Y \to X$ be a Galois factor, and $(Z_1, \alpha_1, \beta_1)$ and $(Z_2, \alpha_2, \beta_2)$ be unramified intermediate factors of $\vp$. Then $(Z_1, \alpha_1, \beta_1)$ and $(Z_2, \alpha_2, \beta_2)$ are topologically conjugate as intermediate factors if and only if $H(Z_1, \alpha_1, \beta_1) = H(Z_2, \alpha_2, \beta_2)$.
\end{theorem}

\begin{proof}
    Suppose that $(Z_1, \alpha_1, \beta_1)$ and $(Z_2, \alpha_2, \beta_2)$ are topologically conjugate as intermediate factors. Then, by definition, there exists a topological conjugacy $f: Z_1 \to Z_2$ of shift spaces such that $\alpha_2 = f \circ \alpha_1$. 
    First, we show that $H(Z_1, \alpha_1, \beta_1) \subseteq H(Z_2, \alpha_2, \beta_2)$. 
    Fix $\varphi \in H(Z_1, \alpha_1, \beta_1) = \{\varphi \in \aut(Y) \mid \alpha_1 \circ \varphi = \alpha_1\}$. Then we have
    \[
    \alpha_2\circ\varphi=f\circ\alpha_1\circ\varphi=
    f\circ\alpha_1=\alpha_2, 
    \]
    which implies $\varphi \in H(Z_2, \alpha_2, \beta_2)=\{\varphi \in \aut(Y) \mid \alpha_2 \circ \varphi = \alpha_2\}$. The reverse inclusion $H(Z_1, \alpha_1, \beta_1) \supseteq H(Z_2, \alpha_2, \beta_2)$ can be shown similarly. Therefore, we have $H(Z_1, \alpha_1, \beta_1) = H(Z_2, \alpha_2, \beta_2)$. \\
    \indent Conversely, suppose that $H(Z_1, \alpha_1, \beta_1) = H(Z_2, \alpha_2, \beta_2) =: H$. For any $z \in Z_1$, since $\alpha_1: Y \to Z_1$ is surjective, there exists a $y \in Y$ such that $\alpha_1(y) = z$. We define a map $f: Z_1 \to Z_2$ by setting $f(z) = \alpha_2(y)$. 
    First, we show that $f$ is well-defined. Let $z \in Z_1$, and fix $y, y' \in Y$ satisfying $\alpha_1(y) = \alpha_1(y') = z$. 
    It suffices to show that $\alpha_2(y) = \alpha_2(y')$. 
    By the definition of an intermediate factor, we have
    \[
    \vp(y)=\beta_1(\alpha_1(y))=\beta_1(\alpha_1(y'))=\vp(y'). 
    \]
    This means that $y$ and $y'$ are in the same fiber. Since $\vp$ is a Galois factor, Proposition~\ref{transitiveaction} implies that the Galois group $G = \gal(Y/X, \vp)$ acts transitively on the fiber. That is, there exists a $\varphi \in G$ such that $\varphi(y) = y'$. 
    Here, since $\varphi \in G$, we have $\vp \circ \varphi = \vp$, and thus
    \[
    \beta_1 \circ (\alpha_1 \circ \varphi) = (\beta_1 \circ \alpha_1) \circ \varphi = \vp \circ \varphi = \vp.
    \]
    This implies that $\alpha_1 \circ \varphi$ is a lift of $\vp$ with respect to $\beta_1$. Furthermore, the relation $\beta_1 \circ \alpha_1 = \vp$ implies that $\alpha_1$ is also a lift of $\vp$ with respect to $\beta_1$. Moreover, we have
    \[
    (\alpha_1\circ\varphi)(y)=\alpha_1(\varphi(y))=\alpha_1(y')=\alpha_1(y). 
    \]  
    By the uniqueness of lifts (Theorem \ref{liftunique}), we have $\alpha_1 \circ \varphi = \alpha_1$. 
    This means that $\varphi \in H(Z_1, \alpha_1, \beta_1)$. By the assumption $H(Z_1, \alpha_1, \beta_1) = H(Z_2, \alpha_2, \beta_2)$, we have $\varphi \in H(Z_2, \alpha_2, \beta_2)$, and therefore $\alpha_2 \circ \varphi = \alpha_2$. From the above, we obtain
    \[\alpha_2(y') = \alpha_2(\varphi(y)) = (\alpha_2 \circ \varphi)(y) = \alpha_2(y),
    \]
    which confirms that $f$ is well-defined.
    Next, we show that $f: Z_1 \to Z_2$ is a topological conjugacy as intermediate factors. We define $g: Z_2 \to Z_1$ similarly; for $z' \in Z_2$, take $y' \in Y$ such that $\alpha_2(y') = z'$ and set $g(z') = \alpha_1(y')$. The map $g$ is the inverse of $f$. Thus, $f$ is bijective. From the definition of $f$, it is clear that $f \circ \alpha_1 = \alpha_2$. By a simple calculation, we also have $\beta_1 = \beta_2 \circ f$ and $f\circ \sigma_{Z_1}=\sigma_{Z_2}\circ f$. 
    Finally, we show that $f$ is continuous. Since $\alpha_1: Y \to Z_1$ is an unramified factor map, it is an open map by Lemma~\ref{openmap}. Furthermore, as a factor map, $\alpha_1$ is a continuous surjection. Therefore, $\alpha_1$ is a quotient map. Since $f \circ \alpha_1 = \alpha_2$ holds and $\alpha_2$ is continuous, the universal property of quotient maps (Lemma~\ref{universalityofquotientmap}) implies that $f$ is continuous. From the above, it follows that $f$ is a topological conjugacy of intermediate factors.
\end{proof}

\begin{theorem}\label{subgrouptointermediatefactor}
    Let $X$ and $Y$ be irreducible SFTs, and $\vp: Y \to X$ be a Galois factor. Set $G = \gal(Y/X)$, and let $H$ be a subgroup of $G$. Then there exists an unramified intermediate factor $(Z, \alpha, \beta)$ of $\vp$ such that $\alpha: Y \to Z$ is a Galois factor and $H = \gal(Y/Z)$. 
    
\end{theorem}

\begin{proof}
    Recall Lemma~\ref{orbitspaceSFT}. For $y, y' \in Y$, we write $y \sim_H y'$ if and only if there exists an $h \in H$ such that $y' = h(y)$. This defines an equivalence relation $\sim_H$ on $Y$. Let $Z = Y/{\sim_H} = Y/H$ be the quotient space (orbit space). We denote the equivalence class of $y \in Y$ by $[y]_H$, and let $q_H: Y \to Z$, $y \mapsto [y]_H$ be the natural projection. We define a map $\sigma^*_Z: Z \to Z$ by $\sigma^*_Z([y]_H) = [\sigma_Y(y)]_H$. By Lemma~\ref{orbitspaceSFT}, the dynamical system $(Z, \sigma^*_Z)$ is topologically conjugate to an irreducible SFT. Thus, we may identify $Z$ with an irreducible SFT.

    Since $q_H$ is a quotient map, it is a continuous surjection, and by definition, it commutes with the shifts. That is, $q_H: Y \to Z$ is a factor map. Furthermore, by the universal property of quotient spaces, there exists a unique continuous map $\beta: Z \to X$ such that $\vp = \beta \circ q_H$. In what follows, we will show that $(Z, q_H, \beta)$ is the desired intermediate factor. Then we have
    \begin{align*}
      \beta(\sigma^*_Z(z))&=\beta(\sigma^*_Z(q_H(y)))\\
      &=\beta(q_H(\sigma_Y(y)))\\
      &=\vp(\sigma_Y(y))\\
      &=\sigma_X(\vp(y))\\
      &=\sigma_X(\beta(q_H(y))\\
      &=\sigma_X(\beta(z)),
    \end{align*}
    which implies that $\beta$ commutes with the shifts.
    Next, we show that $q_H$ is unramified. Fix $z \in Z$. Since $q_H$ is a surjection, there exists a $y \in Y$ such that $q_H(y) = z$. Then we have
    \begin{align*}
    q_H^{-1}(z)=\{y'\in Y\mid [y']_H=[y]_H\}=H y, 
    \end{align*}
    where $H y$ is the orbit of $y$ under the action of $H$. By applying Corollary~\ref{freeaction} to $\vp$, we see that the action of $G$ on $Y$ is free. Thus, the action of $H$ on $q_H^{-1}(z)$ is also free, which implies $|I_H(y)| = 1$. By Lemma~\ref{orbitstabilizer}, we obtain 
    \[
    |q_H^{-1}(z)|=|Hy||I_H(y)|=|H|. 
    \]
    This shows that every fiber of $q_H$ has cardinality $|H|$, and hence $q_H$ is unramified. From the above, it follows that $(Z, q_H, \beta)$ is an unramified intermediate factor.\\
    \indent 
    In the following, we show that $H = \aut(Y/Z, q_H) = \{\varphi \in \aut(Y) \mid q_H \circ \varphi = q_H\}$. Fix $h \in H$ and $y \in Y$. From the definition of $\sim_H$, we have $y \sim_H h(y)$, which yields
    \[
    (q_H\circ h)(y)=[h(y)]_H=[y]_H=q_H(y), 
    \]
    and hence $h \in \aut(Y/Z, q_H)$. Therefore, we obtain the inclusion $H \subseteq \aut(Y/Z, q_H)$. 
    Conversely, fix $\varphi \in \aut(Y/Z, q_H)$ and $y_0 \in Y$. Then 
    \[
    [\varphi(y_0)]_H = (q_H \circ \varphi)(y_0) = q_H(y_0) = [y_0]_H
    \] 
    holds. This means that there exists an $h \in H$ such that $h(y_0) = \varphi(y_0)$. Set $\psi = h^{-1} \circ \varphi$. 
    As $h \in H \subseteq G = \gal(Y/X, \vp)$ and $\varphi \in \aut(Y/Z, q_H)$, we have 
    \begin{align*}
    \vp\circ\psi&=(\vp\circ h^{-1})\circ \varphi\\
    &=\vp \circ\varphi\\
    &=\beta\circ (q_H\circ\varphi)\\
    &=\beta\circ q_H\\
    &=\vp,
    \end{align*}
    which implies $\psi\in \gal(Y/X)$. Furthermore, since $h(y_0)=\varphi(y_0)$, we have $\psi(y_0) = y_0$. By Lemma~\ref{rigidity}, we obtain $\psi = \id_Y$. Therefore, $\varphi = h \in H$. This proves the reverse inclusion $H \supseteq \aut(Y/Z, q_H)$. \\
    \indent Finally, we show that $q_H: Y \to Z$ is a Galois factor. By Proposition~\ref{transitiveaction}, it suffices to show that $\aut(Y/Z, q_H)$ acts transitively on every fiber. Let $z \in Z$, and let $a, a' \in q_H^{-1}(z)$. By the definition of $q_H$, there exists a $y \in Y$ such that the fiber is the orbit of $y$. That is, $q_H^{-1}(z) = H y = \{h(y) \in Y \mid h \in H\}$. Therefore, there exist $h_1, h_2 \in H$ such that $a = h_1(y)$ and $a' = h_2(y)$. Then we have $a = (h_1 \circ h_2^{-1})(a')$, which means that the action is transitive.
\end{proof}

\begin{definition}
    Let $X$ and $Y$ be irreducible SFTs, $\vp: Y \to X$ be a Galois factor, $G = \gal(Y/X)$, and $H$ be a subgroup of $G$. Then, by Theorem~\ref{conjugatesamegroupequiv}, the intermediate factors obtained by Theorem~\ref{subgrouptointermediatefactor} are topologically conjugate to each other. We denote this intermediate factor by $(Z(H), \alpha(H), \beta(H))$. Note that this factor is conjugate to $(Y/H, q_H, \beta)$, which was constructed in the proof of Theorem~\ref{subgrouptointermediatefactor}.
\end{definition}

As a corollary of Theorem \ref{subgrouptointermediatefactor}, we obtain the following.

\begin{corollary}
    Let $X$ and $Y$ be irreducible SFTs, $\vp: Y \to X$ be a Galois factor. Set $G = \gal(Y/X)$, let $H$ be a subgroup of $G$, and let $(Z, \alpha, \beta) = (Z(H), \alpha(H), \beta(H))$. Then the following hold:
    \begin{enumerate}
        \item $\deg(\alpha)=|H|$.
        \item $\deg(\beta)=[G:H]=|G|/|H|$.
    \end{enumerate}
\end{corollary}

\begin{definition}
    Let $X$ and $Y$ be irreducible SFTs, $\vp: Y \to X$ be a Galois factor. Set $G = \gal(Y/X)$. We denote the set of all subgroups of $G$ by $\mathcal{S}(G)$, and denote the set of all topological conjugacy classes of unramified intermediate factors of $\vp$ by $\mathcal{I}(Y/X, \vp) = \mathcal{I}(Y/X)$. 
    Furthermore, we define the following two maps: 
    \begin{align*}
        &\Phi: \mathcal{I}(Y/X) \to \mathcal{S}(G), \quad [(Z, \alpha, \beta)] \mapsto H(Z, \alpha, \beta),\\
        &\Psi: \mathcal{S}(G) \to \mathcal{I}(Y/X), \quad H \mapsto [(Z(H), \alpha(H), \beta(H))].
    \end{align*}
    Note that the well-definedness of $\Phi$ is guaranteed by Theorem~\ref{conjugatesamegroupequiv}. 
\end{definition}

By Theorems~\ref{subgrouptointermediatefactor} and \ref{conjugatesamegroupequiv}, the following holds.

\begin{theorem}[The fundamental theorem of Galois theory for shifts of finite type]
    With the notation above, the maps $\Phi$ and $\Psi$ are mutually inverse maps. Therefore, there is a one-to-one correspondence between $\mathcal{S}(G)$ and $\mathcal{I}(Y/X, \vp)$.
\end{theorem}

Next, we state a theorem concerning the inclusion of subgroups.

\begin{theorem}
    Let $X$ and $Y$ be irreducible SFTs, $\vp: Y \to X$ be a Galois factor, and $G = \gal(Y/X)$. Let $(Z_1, \alpha_1, \beta_1)$ and $(Z_2, \alpha_2, \beta_2)$ be intermediate factors, and set $H_1 = H(Z_1, \alpha_1, \beta_1)$ and $H_2 = H(Z_2, \alpha_2, \beta_2)$. Then the following are equivalent: 
    \begin{enumerate}
        \item $H_1\subseteq H_2$
        \item There exists an unramified factor map $\gamma: Z_1 \to Z_2$ such that $(Z_1, \alpha_1, \gamma)$ is an unramified intermediate factor of $\alpha_2: Y \to Z_2$.
    \end{enumerate}
\end{theorem}

\begin{proof}
    Suppose that $H_1 \subseteq H_2$. For any $z \in Z_1$, since $\alpha_1$ is surjective, we can take a $y \in Y$ such that $\alpha_1(y) = z$. We define $\gamma: Z_1 \to Z_2$ by setting $\gamma(z) = \alpha_2(y)$. 
    First, we show that $\gamma$ is well-defined. Fix $y, y' \in Y$ with $\alpha_1(y) = \alpha_1(y') = z$. By Theorem~\ref{subgrouptointermediatefactor}, $\alpha_1: Y \to Z_1$ is a Galois factor. By Proposition~\ref{transitiveaction}, the group $H_1$ acts transitively on every fiber. Therefore, since $y, y' \in \alpha_1^{-1}(z)$, there exists a $\varphi \in H_1$ such that $\varphi(y) = y'$. By the assumption, we have $\varphi \in H_1 \subseteq H_2 = \aut(Y/Z_2, \alpha_2)$, and hence $\alpha_2 \circ \varphi = \alpha_2$. Thus, we obtain
    \[
    \alpha_2(y')=\alpha_2(\varphi(y))=\alpha_2(y),
    \]
    which means that $\gamma$ is well-defined.
    It is clear from the definition of $\gamma$ that $\gamma \circ \alpha_1 = \alpha_2$. Therefore, it suffices to show that $\gamma$ is an unramified factor map. First, the fact that $\gamma$ is surjective and unramified follows immediately from $\gamma \circ \alpha_1 = \alpha_2$ and Lemma \ref{const_to_one}. Furthermore, a simple calculation shows that $\gamma$ commutes with the shifts.
    Next, we show that $\gamma$ is continuous. By Lemma~\ref{openmap}, $\alpha_1$ is an open map. Since $\alpha_1$ is a factor map, it is a continuous surjection. Therefore, $\alpha_1$ is a quotient map. Furthermore, since $\alpha_2$ is continuous and $\gamma \circ \alpha_1 = \alpha_2$, the universal property of quotient maps (Lemma~\ref{universalityofquotientmap}) implies that $\gamma$ is continuous. This shows that $\gamma$ is an unramified factor map.\\
    \indent Conversely, assume that there exists an unramified factor map $\gamma: Z_1 \to Z_2$ such that $(Z_1, \alpha_1, \gamma)$ is an unramified intermediate factor of $\alpha_2: Y \to Z_2$. 
    Fix $\varphi \in H_1 = \aut(Y/Z_1, \alpha_1)$. Then we have $\alpha_1 \circ \varphi = \alpha_1$. Thus, $\gamma \circ \alpha_1 \circ \varphi = \gamma \circ \alpha_1$. Since $\gamma \circ \alpha_1 = \alpha_2$, we obtain $\alpha_2 \circ \varphi = \alpha_2$. This implies that $\varphi \in H_2 = \aut(Y/Z_2, \alpha_2)$. Therefore, we obtain the inclusion $H_1 \subseteq H_2$.
\end{proof}

\begin{remark}
    When there exists an unramified factor map $\gamma: Z_1 \to Z_2$ such that $(Z_1, \alpha_1, \gamma)$ is an unramified intermediate factor of $\alpha_2: Y \to Z_2$, the relation $\beta_2 \circ \gamma = \beta_1$ follows automatically. Indeed, we have 
    \[
    (\beta_2\circ\gamma)(z)=\beta_2(\gamma(\alpha_1(y)))=\beta_2(\alpha_2(y))=\vp(y)=\beta_1(\alpha_1(y))=\beta_1(z).
    \]
    Therefore, the following diagram commutes. 
\[
\begin{tikzcd}[column sep=large]
    Y \arrow[r, "\alpha_1"] 
      \arrow[rr, "\alpha_2"', bend right=30] 
      \arrow[rrr, "\vp", bend left=40] 
    & Z_1 \arrow[r, "\gamma"] 
          \arrow[rr, "\beta_1", bend left=30] 
    & Z_2 \arrow[r, "\beta_2"] 
    & X
\end{tikzcd}
\]
\end{remark}

Next, we state a theorem about normal subgroups of the Galois group.

\begin{theorem}\label{normal_subgroup}
    Let $X$ and $Y$ be irreducible SFTs, $\vp: Y \to X$ be a Galois factor, and set $G = \gal(Y/X)$. Let $H$ be a subgroup of $G$, and set $(Z, \alpha, \beta) = (Z(H), \alpha(H), \beta(H))$ be the intermediate factor corresponding to $H$. Then the following hold:
    \begin{enumerate}
        \item $H$ is a normal subgroup of $G$ if and only if $\beta: Z \to X$ is a Galois factor. 
        \item In this case, the Galois group $\gal(Z/X, \beta)$ is isomorphic to the quotient group $G/H$.
    \end{enumerate}
\end{theorem}

\begin{proof}
    Suppose that $H$ is a normal subgroup of $G$. We now simultaneously show that $\beta: Z \to X$ is a Galois factor and  $\gal(Z/X, \beta)\cong G/H$. 
    Let $\varphi \in G$. We define a map $\tilde{\varphi}: Z \to Z$ as follows. For any $z \in Z$, since $\alpha$ is surjective, there exists a $y \in Y$ such that $\alpha(y) = z$. Using this, we set $\tilde{\varphi}(z) = \alpha(\varphi(y))$. 
    First, we show that $\tilde{\varphi}$ is well-defined. Fix $y, y' \in Y$ with $\alpha(y) = \alpha(y') = z$. By Theorem~\ref{subgrouptointermediatefactor}, $\alpha: Y \to Z$ is a Galois factor and $\gal(Y/Z, \alpha) = H$ holds. Thus, by Proposition~\ref{transitiveaction}, the group $H$ acts transitively on every fiber. Since $y, y' \in \alpha^{-1}(z)$, there exists a $\psi \in H$ such that $y' = \psi(y)$. By the assumption that $H$ is a normal subgroup, there exists a $\psi' \in H$ such that $\varphi \circ \psi = \psi' \circ \varphi$. Then we have
    \[
    \alpha(\varphi(y'))=\alpha(\varphi(\psi(y)))=
    \alpha(\psi'(\varphi(y)))=\alpha(\varphi(y)), 
    \]
    which implies that $\tilde{\varphi}$ is well-defined. Next, we show that $\tilde{\varphi} \in \aut(Z/X, \beta)$.
    From the definition, it is clear that $\tilde{\varphi} \circ \alpha = \alpha \circ \varphi$. Since $\widetilde{\varphi^{-1}}$ is the inverse of $\tilde{\varphi}$, it is also immediately clear that $\tilde{\varphi}$ is bijective.
    By Lemma~\ref{openmap}, $\alpha$ is an open map. Because $\alpha$ is continuous, it is a quotient map. Since $\alpha \circ \varphi$ is also continuous, the universality of quotient maps (Lemma~\ref{universalityofquotientmap}) implies that $\tilde{\varphi}$ is continuous. Next, we show that $\tilde{\varphi}$ commutes with the shifts. Fix $z \in Z$, and $y \in Y$ with $\alpha(y) = z$. Then we have 
    \begin{align*}
         (\sigma_Z\circ\tilde{\varphi})(z)&=(\sigma_Z\circ\tilde{\varphi}\circ\alpha)(y)\\
         &=(\sigma_Z\circ\alpha\circ\varphi)(y)\\
         &=(\alpha\circ\varphi\circ\sigma_Y)(y)\\
         &=(\tilde{\varphi}\circ\alpha\circ\sigma_Y)(y)\\
         &=(\tilde{\varphi}\circ\sigma_Z\circ\alpha)(y)\\
         &=(\tilde{\varphi}\circ\sigma_Z)(z). 
    \end{align*}
    Finally, we show that $\beta \circ \tilde{\varphi} = \beta$. Since $\varphi \in G = \aut(Y/X, \vp)$, we have $\vp \circ \varphi = \vp$. Furthermore, since $(Z, \alpha, \beta)$ is an intermediate factor, we have $\vp = \beta \circ \alpha$. Thus, we obtain
    \[
    \beta\circ\tilde{\varphi}\circ\alpha=
    \beta\circ\alpha\circ\varphi=\vp\circ\varphi=\vp=\beta\circ\alpha.
    \]
    Since $\alpha$ is surjective, we have $\beta \circ \tilde{\varphi} = \beta$. Therefore, we have shown that $\tilde{\varphi} \in \aut(Z/X, \beta)$. \\
    \indent We define a group homomorphism $F: G \to \aut(Z/X, \beta)$ by $F(\varphi) = \tilde{\varphi}$. First, we show that $\mathrm{Ker}(F) = H$. Fix $\varphi \in \mathrm{Ker}(F)$. Then we have $\tilde{\varphi} = F(\varphi) = \id_Z$. Let $y \in Y$, and $z = \alpha(y)$. Then we obtain
    \[
    \alpha(\varphi(y))=\tilde{\varphi}(z)=\id_Z(z)=z=\alpha(y). 
    \]
    This means that $y, \varphi(y) \in \alpha^{-1}(z)$. Since $\alpha$ is a Galois factor, Proposition~\ref{transitiveaction} implies that the Galois group $\gal(Y/Z, \alpha) = H$ acts transitively on the fibers. Therefore, there exists a $\psi \in H$ such that $\varphi(y) = \psi(y)$. That is, $(\psi^{-1} \circ \varphi)(y) = y$. Thus, Lemma~\ref{rigidity} implies that $\psi^{-1} \circ \varphi = \id_Y$, which yields $\varphi = \psi \in H$. This shows that $\mathrm{Ker}(F) \subseteq H$. Conversely, let $\varphi \in H = \gal(Y/Z, \alpha)$. Then we have 
    \[
    \tilde{\varphi}\circ\alpha=\alpha\circ\varphi=
    \alpha=\id_{Z}\circ\alpha.
    \]
    Since $\alpha$ is surjective, we have $\tilde{\varphi} = \id_Z$. This implies that $\varphi \in \mathrm{Ker}(F)$, and hence $\mathrm{Ker}(F) \supseteq H$. From the above, we conclude that $\mathrm{Ker}(F) = H$. \\
    \indent By applying the first isomorphism theorem to $F: G \to \aut(Z/X, \beta)$, we obtain
    \[
    G/H\cong \mathrm{Im} (F)\subseteq \aut(Z/X, \beta). 
    \]
    Comparing their orders, we have
    \[
    |\aut(Z/X, \beta)|\ge |G/H|=|G|/|H|.
    \]
    Here, since $\vp$ and $\alpha$ are Galois factors, we have $\deg(\vp) = |G|$ and $\deg(\alpha) = |H|$. From this and Lemma~\ref{degreemultiplicativity}, it follows that
    \[
    \deg(\beta)=\deg(\vp)/\deg(\alpha)=|G|/|H|\le|\aut(Z/X, \beta)|.
    \]
    On the other hand, Proposition~\ref{cj_le_deg} states that $\deg(\beta) \ge |\aut(Z/X, \beta)|$. Therefore $\deg(\beta) = |\aut(Z/X, \beta)|$. This implies that $\beta$ is a Galois factor. Furthermore, since
    \[
    |\mathrm{Im}(F)|=|G/H|=\deg(\beta)=|\aut(Z/X, \beta)|,
    \]
    and $\mathrm{Im}(F) \subseteq \aut(Z/X, \beta)$, we conclude that $G/H \cong \mathrm{Im}(F) = \gal(Z/X, \beta)$.
    Conversely, assume that $\beta: Z \to X$ is a Galois factor. Let $\varphi \in G$ and $\chi \in H = \gal(Y/Z, \alpha)$. Consider the sliding block code $\alpha \circ \varphi^{-1}: Y \to Z$. Since we have
    \[
    \beta\circ(\alpha\circ\varphi^{-1})=\vp\circ\varphi^{-1}=\vp, 
    \]
    the map $\alpha \circ \varphi^{-1}$ is a lift of $\vp$ along $\beta$. Let $y_0 \in Y$. Furthermore, we have
    \[
    \beta((\alpha\circ\varphi^{-1})(y_0))=\vp(y_0)=\beta(\alpha(y_0)).
    \]
    That is, $(\alpha \circ \varphi^{-1})(y_0)$ and $\alpha(y_0)$ are elements in the same fiber of $\beta$. 
    Furthermore, since $\beta: Z \to X$ is a Galois factor, Proposition~\ref{transitiveaction} implies that $\gal(Z/X, \beta)$ acts transitively on every fiber. Therefore, there exists a $\varphi' \in \gal(Z/X, \beta)$ such that $(\alpha \circ \varphi^{-1})(y_0) = (\varphi' \circ \alpha)(y_0)$. 
    Considering the map $\varphi' \circ \alpha$, we have
    \[
    \beta\circ(\varphi'\circ\alpha)=\beta\circ\alpha=\vp. 
    \]
    Thus, $\varphi' \circ \alpha$ is also a lift of $\vp$ along $\beta$.
    Therefore, by the uniqueness of lifts (Theorem~\ref{liftunique}), we have $\alpha \circ \varphi^{-1} = \varphi' \circ \alpha$, which implies $\varphi'^{-1}\circ\alpha=\alpha\circ\varphi$.
    Hence, we have 
    \begin{align*}
        \alpha\circ(\varphi\circ\chi\circ\varphi^{-1})&= (\alpha\circ\varphi)\circ\chi\circ\varphi^{-1}\\
        &=(\varphi'^{-1}\circ\alpha)\circ\chi\circ\varphi^{-1}\\
        &=\varphi'^{-1}\circ(\alpha\circ\chi)\circ\varphi^{-1}\\
        &=\varphi'^{-1}\circ\alpha\circ\varphi^{-1}\\
        &=\alpha\circ\varphi\circ\varphi^{-1}\\
        &=\alpha.
    \end{align*}
    This implies that $\varphi \circ \chi \circ \varphi^{-1} \in \gal(Y/Z, \alpha) = H$. This means that $H$ is a normal subgroup.
    \end{proof}

\subsection{Examples of Determining the Structure of Galois Factors}\label{subsection_ex_of_gal}
    In this subsection, we provide some examples of structure determination using the fundamental theorem.

    \begin{example}\label{ex_two_shift}
    Consider the full shift $X = \AA^\Z$ on the alphabet $\AA = \{0, 1\}$. Let $Y = \AA^\Z$. We define $\vp: Y \to X$ by
    \[
    \vp(y)_i= y_{i+1}-y_i\pmod{2}
    \]
    for each $i \in \Z$. It is easy to see that $\vp$ is a factor map. Let $x \in X$, and consider $y \in \vp^{-1}(x)$. If we set $y_0 = 0$, all other coordinates are uniquely determined by the relation $\vp(y)_i = y_{i+1} - y_i \pmod{2}$. The same holds when $y_0 = 1$. That is, we have $|\vp^{-1}(x)| = 2$, which implies that $\vp$ is an unramified factor map of degree $\deg(\vp) = 2$. 
    Define $\tau: Y \to Y$ by $\tau(y)_i=y_i+1$ for each $i \in \Z$. 
    Clearly, $\tau \in \aut(Y)$. Furthermore, for each $i \in \Z$, we have
    \begin{align*}
        \vp(\tau(y))_i&=\tau(y)_{i+1}-\tau(y)_i\pmod 2\\
        &=(y_{i+1}+1)-(y_i+1) \pmod 2\\
        &=y_{i+1}-y_i \pmod 2\\
        &=\vp(y)_i.
    \end{align*}
    Therefore, $\tau \in \aut(Y/X, \vp)$. By definition, we also have $\id_Y \in \aut(Y/X, \vp)$. This implies that $|\aut(Y/X, \vp)| \ge 2$. On the other hand, by Proposition~\ref{cj_le_deg}, we have
    \[
    |\aut(Y/X,\vp)|\le\deg(\vp)=2.
    \]
Thus, $|\aut(Y/X, \vp)| = 2$, and we obtain
    \[
    \aut(Y/X, \vp) = \{\id_Y, \tau\}.
    \]
Therefore, $\vp$ is a Galois factor, and
    \[
    \gal(Y/X, \vp) \cong \Z/2\Z.
    \]
Since $\Z/2\Z$ has no non-trivial subgroups, $\vp$ has no non-trivial unramified intermediate factors. 
    \end{example}

The following example is a generalization of the previous one. Since the calculations are similar to those in the previous example, we omit the details.

    \begin{example}\label{ex_z_over_nz}
    Let $n \ge 2$ be an integer, and let $X = Y = \AA^\Z$ be the full shifts on the alphabet $\AA = \{0, 1, 2, \dots, n-1\}$. We define $\vp: Y \to X$ by
    \[
    \vp(y)_i=y_{i+1}-y_i \pmod{n}
    \]
    for each $i \in \Z$. For the same reason as in Example~\ref{ex_two_shift}, we have $|\vp^{-1}(x)| = n$ for each $x \in X$. That is, $\vp$ is an unramified factor map with $\deg(\vp) = n$. 
    Let $c \in \{0, 1, 2, \dots, n-1\}$, and define $\tau_c: Y \to Y$ by
    \[
    \tau_c(y)_i=y_i+c \pmod{n}
    \]
    for each $i \in \Z$. Then we have
    \[
    \gal(Y/X,\vp)=\{\id_Y=\tau_0,\tau_1,\tau_2,\dots,\tau_{n-1}\}=\langle\tau_1\rangle\cong\Z/n\Z. 
    \]
    In the following, we consider the case $n = 6$. In this case, there are two non-trivial subgroups of $G = \gal(Y/X, \vp)$:
    \begin{align*}
        H_1&=\{\tau_0,\tau_3\}\cong\Z/2\Z\\
        H_2&=\{\tau_0,\tau_2,\tau_4\}\cong\Z/3\Z.
    \end{align*}
    Let us construct the unramified intermediate factors corresponding to these subgroups. 
    By the proof of Theorem~\ref{subgrouptointermediatefactor}, it suffices to consider the orbit space $Y/H_1$. In $Y/H_1$, an element $y$ is identified with the sequence obtained by adding $3$ to all of its coordinates. Thus, let us consider the following construction.
    Let $\AA_1 = \Z/3\Z \oplus \Z/6\Z$, and define
    \[
    Z_1=\{(a,b)\in\AA_1^\Z\mid a_{i+1}-a_i=b_i \pmod{3}\text{ for any }i\in \Z\}.
    \]
    We define the maps $\alpha_1: Y \to Z_1$ and $\beta_1: Z_1 \to X$ by 
    \begin{align*}
        &\alpha_1(y)_i=(y_i\bmod 3, \ y_{i+1}-y_i\bmod 6)\\
        &\beta_1(a, b)_i=b_i
    \end{align*}
    for each $i \in \Z$. A direct computation shows that $\vp=\beta_1\circ\alpha_1$ and $H_1=\gal(Y/Z_1,\alpha_1)$. 
    Since $H_1$ is a normal subgroup, Theorem~\ref{normal_subgroup} implies that $\beta_1$ should be a Galois factor and its Galois group should be $G/H_1$. Let us verify this. For $c \in \{0, 1, 2\}$, define $\tau'_c: Z_1 \to Z_1$ by
    \[
    \tau'_c(a, b)_i=(a_i+c\pmod 3, \ b_i)
    \]
    for each $i \in \Z$. Then we have $\tau'_c \in \aut(Z_1/X, \beta_1)$, and thus
    \[
    \gal(Z_1/X,\beta)=\{\tau'_0, \tau'_1, \tau'_2\}=
    \langle\tau'_1\rangle\cong\Z/3\Z\cong (\Z/6\Z)/(\Z/2\Z)\cong G/H_1
    \]
    holds. 
    \end{example}

    In the following, we provide two examples in which the Galois group is not a cyclic group.
    
    \begin{example}\label{ex_klein}
    Let $\AA = \{(0,0), (0,1), (1,0), (1,1)\} = \Z/2\Z \oplus \Z/2\Z$, and let $X = Y = \AA^\Z$. We write the $i$-th coordinate of $y \in Y$ as $y_i = (y_i^{(1)}, y_i^{(2)})$. We define $\vp: Y \to X$ by
    \[
    \vp(y)_i=(y_{i+1}^{(1)}-y_i^{(1)}, \ y_{i+1}^{(2)}-y_i^{(2)} ). 
    \]
    Then, $\vp$ is an unramified factor map with $\deg(\vp) = 4$. 
    For $c \in \AA = \Z/2\Z \oplus \Z/2\Z$, we define $\tau_c: Y \to Y$ by
\[
    \tau_c(y)_i = y_i + c \pmod{(2,2)}
\]
for each $i \in \Z$. 
Then we can see that the Galois group $G = \gal(Y/X, \vp)$ is the Klein four-group:
\[
    G = \{\id_Y = \tau_{(0,0)}, \tau_{(1,0)}, \tau_{(0,1)}, \tau_{(1,1)}\} \cong \Z/2\Z \oplus \Z/2\Z.
\]
The group $G$ has three non-trivial proper subgroups:
\begin{align*}
    H_1 &= \{\tau_{(0,0)}, \tau_{(1,0)}\} \cong \Z/2\Z, \\
    H_2 &= \{\tau_{(0,0)}, \tau_{(0,1)}\} \cong \Z/2\Z, \\
    H_3 &= \{\tau_{(0,0)}, \tau_{(1,1)}\} \cong \Z/2\Z.
\end{align*}
 Let $Z_1 = Z_2 = Z_3 = \AA^\Z$. The intermediate factors corresponding to these subgroups are given as follows:
\begin{align*}
    \alpha_1: Y \to Z_1, &\quad \alpha_1(y)_i = (y_{i+1}^{(1)} - y_i^{(1)}, \ y_i^{(2)}), \\
    \beta_1: Z_1 \to X, &\quad \beta_1(z)_i = (z_i^{(1)}, \ z_{i+1}^{(2)} - z_i^{(2)}), \\
    \alpha_2: Y \to Z_2, &\quad \alpha_2(y)_i = (y_i^{(1)}, \ y_{i+1}^{(2)} - y_i^{(2)}), \\
    \beta_2: Z_2 \to X, &\quad \beta_2(z)_i = (z_{i+1}^{(1)} - z_i^{(1)}, \ z_i^{(2)}), \\
    \alpha_3: Y \to Z_3, &\quad \alpha_3(y)_i = (y_{i+1}^{(1)} - y_i^{(1)}, \ y_i^{(2)} - y_i^{(1)}), \\
    \beta_3: Z_3 \to X, &\quad \beta_3(z)_i = (z_i^{(1)}, \ z_{i+1}^{(2)} + z_i^{(1)} - z_i^{(2)}),
\end{align*}
    for each $i \in \Z$. 
    This structure is identical to that of the Galois extension $\Q(\sqrt{2}, \sqrt{3}) / \Q$ in field theory. The Galois group of $\Q(\sqrt{2}, \sqrt{3}) / \Q$ is also $\Z/2\Z \oplus \Z/2\Z$, and its intermediate fields are $\Q(\sqrt{2})$, $\Q(\sqrt{3})$, and $\Q(\sqrt{6})$. The following two diagrams share the same structure:
    
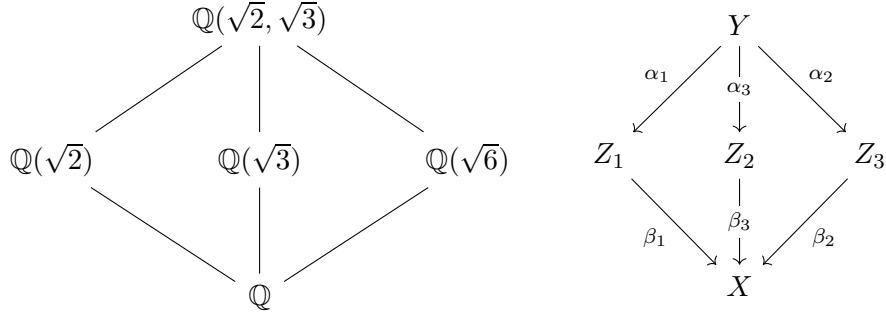
\[
    \begin{tikzcd}[row sep=3em, column sep=2.5em]
    & \mathbb{Q}(\sqrt{2}, \sqrt{3}) \arrow[dl, dash] \arrow[d, dash] \arrow[dr, dash] & \\
    \mathbb{Q}(\sqrt{2}) \arrow[dr, dash] & \mathbb{Q}(\sqrt{3}) \arrow[d, dash] & \mathbb{Q}(\sqrt{6}) \arrow[dl, dash] \\
    & \mathbb{Q} & 
\end{tikzcd}
\hspace{2em}
\begin{tikzcd}[row sep=3em, column sep=2.5em]
    & Y \arrow[dl, "\alpha_1"'] \arrow[d, "\alpha_3" description] \arrow[dr, "\alpha_2"] & \\
    Z_1 \arrow[dr, "\beta_1"'] & Z_2 \arrow[d, "\beta_3" description] & Z_3 \arrow[dl, "\beta_2"] \\
    & X &
\end{tikzcd}
    \]
    \end{example}

    Next, we present an example where the Galois group is non-abelian.

    \begin{example}\label{ex_sym_shift}
    Let $\mathfrak{S}_3$ be the symmetric group of degree 3, $\AA = \mathfrak{S}_3$, and $X = Y = \AA^\Z$. We define $\vp: Y \to X$ for each $i \in \Z$ by
     \[
     \vp(y)_i=y^{-1}_iy_{i+1}.
     \]
    Let $x \in X$. We consider $y \in \vp^{-1}(x)$. Once $y_0$ is fixed, all other coordinates of $y$ are uniquely determined by the relation $x_i = y_i^{-1} y_{i+1}$. Since there are $|\mathfrak{S}_3| = 6$ choices for $y_0$, we have $|\vp^{-1}(x)| = 6$, which implies that $\vp$ is an unramified factor map. 
    For $g \in \mathfrak{S}_3$, we define $\tau_g: Y \to Y$ by
    \[
    \tau_g(y)_i = g y_i
    \]
    for each $i \in \Z$. Then we have
    \begin{align*}
            \vp(\tau_g(y))_i
            &=(\tau_g(y)_i)^{-1} (\tau_g(y)_{i+1})\\
            &=(g y_i)^{-1} (g y_{i+1})\\
            &=y_i^{-1} g^{-1} g y_{i+1}\\
            &=y_i^{-1} y_{i+1}\\
            &=\vp(y)_i. 
        \end{align*}
    Thus, we obtain $\vp \circ \tau_g = \vp$. This means that $\tau_g \in \gal(Y/X, \vp)$. Set $G = \gal(Y/X, \vp)$. Furthermore, for each $i \in \Z$ and $g, h \in \mathfrak{S}_3$, we have
    \[
    \tau_h(\tau_g(y))_i = h(g y_i) = (hg)y_i = \tau_{hg}(y)_i.
    \]
    Therefore, the map $f: \mathfrak{S}_3 \to G$ defined by $g \mapsto \tau_g$ is a homomorphism. Since its inverse map is given by $\tau_g \mapsto g$, $f$ is an isomorphism. From the above, we conclude that
    \[
    \gal(Y/X, \vp) \cong \mathfrak{S}_3.
    \]
\end{example}

\subsection{Inverse Galois Problem for SFT}\label{subsection_inverse galois}
Let $G$ be a finite group. A natural question arises: does there exist a Galois factor with Galois group $G$? This is analogous to the corresponding problem in field theory. In fact, if we place no restrictions on the base space, the answer is relatively simple. We can prove the following theorem by replacing $\mathfrak{S}_3$ with $G$ in Example~\ref{ex_sym_shift}. The key idea is to use group shifts. This allows us to embed the group structure directly into the shift space.

\begin{theorem}\label{inverse_gal}
    Let $G$ be a finite group. Then, there exist irreducible SFTs $X$ and $Y$, and a Galois factor $\vp: Y \to X$ such that $\gal(Y/X, \vp) \cong G$.
\end{theorem}

\begin{proof}
    Set $\AA = G$, and $X = Y = \AA^\Z$. We define a factor map $\vp: Y \to X$ by
    \[
    \vp(y)_i=y^{-1}_iy_{i+1}
    \]
    for each $i \in \Z$. 
    By an argument similar to that in Example~\ref{ex_sym_shift}, $\vp$ is an unramified factor map with $\deg(\vp) = |G|$. Furthermore, for $g \in G$, we define $\tau_g: Y \to Y$ by $\tau_g(y)_i = g y_i$ for each $i \in \Z$. Then we have $\tau_g \in \aut(Y/X, \vp)$. Thus, we obtain
    \[
    \{\tau_g\mid g\in G\}\subseteq \aut(Y/X,\vp). 
    \]
    That is, we have 
    \[
    |G|=|\{\tau_g\mid g\in G\}|\le |\aut(Y/X,\vp)|.
    \]
    On the other hand, By Proposition~\ref{cj_le_deg}, we have 
    \[
    |\aut(Y/X,\vp)|\le\deg(\vp)=|G|.
    \]
    From the above, we have $|\aut(Y/X, \vp)| = \deg(\vp)$, which implies that $\vp$ is a Galois factor and $ \gal(Y/X,\vp)=\{\tau_g\mid g\in G\}$.
    For the same reason as in Example~\ref{ex_sym_shift}, the map $f: G \to \gal(Y/X, \vp)$ given by $g \mapsto \tau_g$ is an isomorphism. Therefore, we conclude that $\gal(Y/X,\vp)\cong G$. This completes the proof.
\end{proof}

    If a base space $X$ is fixed, the inverse Galois problem appears to be non-trivial. For instance, for a specific irreducible SFT such as the full shift on $n$ symbols $X = \Sigma_n$ or the golden mean shift, and for an arbitrary finite group $G$, does there exist an irreducible SFT $Y$ and a Galois factor $\vp: Y \to X$ such that $\gal(Y/X, \vp) \cong G$? Furthermore, we can ask an even stronger question: for an arbitrary finite group $G$, does there exist a Galois factor $\vp: X \to X$ such that $\gal(X/X, \vp) \cong G$? This version of the inverse Galois problem for irreducible SFTs has the potential to provide an analogy to the classical inverse Galois problem over the field of rational numbers. We will revisit these questions in the final chapter.

\section{Absolute Galois Group}\label{section_Abs_galois_group}
\subsection{Definition of the Absolute Galois Group}
In this subsection, we define the absolute Galois group. A pair $(X, x_0)$ consisting of an SFT $X$ and an element $x_0 \in X$ is called a pointed SFT. Note that the subscript of $x_0$ does not indicate a coordinate. A map $\varphi: (X, x_0) \to (Y, y_0)$ is said to be a pointed map if $\varphi(x_0) = y_0$. Similarly, a pointed unramified intermediate factor of a pointed unramified factor map $\vp: (Y, y_0) \to (X, x_0)$ is a tuple $(Z, z_0, \alpha, \beta)$ such that $(Z, \alpha, \beta)$ is an unramified intermediate factor of $\vp: Y \to X$, and furthermore, the maps $\alpha: (Y, y_0) \to (Z, z_0)$ and $\beta: (Z, z_0) \to (X, x_0)$ are pointed maps.
Any other ``pointed'' notions are defined analogously. \\
\indent In the following, we prove the lemmas necessary for the definition of the absolute Galois group.

\begin{lemma}\label{medium_unique}
    Let $(X, x_0)$, $(Y_1, y_{1,0})$, and $(Y_2, y_{2,0})$ be pointed irreducible SFTs. Let $\vp_1: (Y_1, y_{1,0}) \to (X, x_0)$ and $\vp_2: (Y_2, y_{2,0}) \to (X, x_0)$ be pointed Galois factors. Then, there is at most one pointed unramified factor map $\alpha: (Y_2, y_{2,0}) \to (Y_1, y_{1,0})$ such that $\vp_2 = \vp_1 \circ \alpha$.
\end{lemma}

\begin{proof}
    Suppose that $\alpha, \alpha': (Y_2, y_{2,0}) \to (Y_1, y_{1,0})$ are two pointed unramified factor maps such that $\vp_2 = \vp_1 \circ \alpha = \vp_1 \circ \alpha'$. Then, both $\alpha$ and $\alpha'$ are lifts of $\vp_2$ along $\vp_1$. Since $\alpha(y_{2,0}) = \alpha'(y_{2,0}) = y_{1,0}$, the uniqueness of lifts (Theorem~\ref{liftunique}) implies that $\alpha = \alpha'$.
\end{proof}

\begin{lemma}\label{directed_set}
    Let $(X, x_0)$, $(Y_1, y_{1,0})$, and $(Y_2, y_{2,0})$ be pointed irreducible SFTs. Let $\vp_1: (Y_1, y_{1,0}) \to (X, x_0)$ and $\vp_2: (Y_2, y_{2,0}) \to (X, x_0)$ be pointed Galois factors. Then, there exist a pointed irreducible SFT $(Y, y_0)$, a pointed Galois factor $\vp: (Y, y_0) \to (X, x_0)$, and pointed unramified factor maps $\alpha_1: (Y, y_0) \to (Y_1, y_{1,0})$ and $\alpha_2: (Y, y_0) \to (Y_2, y_{2,0})$ such that $(Y_1, y_{1,0}, \alpha_1, \vp_1)$ and $(Y_2, y_{2,0}, \alpha_2, \vp_2)$ are pointed unramified intermediate factors of $\vp: (Y, y_0) \to (X, x_0)$. That is, the following conditions hold:
    
    \begin{enumerate}
        \item $\vp: Y \to X$ is a Galois factor satisfying $\vp(y_0) = x_0$.
        \item $\alpha_1: Y \to Y_1$ and $\alpha_2: Y \to Y_2$ are unramified factor maps satisfying $\alpha_1(y_0) = y_{1,0}$ and $\alpha_2(y_0) = y_{2,0}$.
        \item $\vp=\vp_1\circ\alpha_1=\vp_2\circ\alpha_2$ holds. 
    \end{enumerate}
\end{lemma}

\begin{proof}
    Set $G_1 = \gal(Y_1/X, \vp_1)$ and $G_2 = \gal(Y_2/X, \vp_2)$. Let $W$ be the fiber product of $Y_1$ and $Y_2$ over $X$ with respect to $\vp_1$ and $\vp_2$. That is,
    \[
    W = Y_1\times_X Y_2=\{ (y_1, y_2)\in Y_1\times Y_2\mid \vp_1(y_1)=\vp_2(y_2) \}.
    \]
    Since $Y_1$ and $Y_2$ are SFTs, $W$ is also an SFT. Let $\pi: W \to X$ be the natural projection. That is, we define $\pi(y_1, y_2) = \vp_1(y_1) = \vp_2(y_2)$ for each $(y_1, y_2) \in W$. By construction, $\pi$ is a factor map. Furthermore, since $\vp_1$ and $\vp_2$ are open maps by Lemma \ref{openmap}, $\pi$ is also an open map. Moreover, for each $x \in X$, we have
    \[
    \pi^{-1}(x)=\{(y_1,y_2)\in W\mid \vp_1(y_1)=x \text{ and }\vp_2(y_2)=x\}=\vp_1^{-1}(x)\times \vp_2^{-1}(x). 
    \]
    Therefore, we have
    \[
    |\pi^{-1}(x)|=|\vp_1^{-1}(x)\times \vp_2^{-1}(x)|=\deg(\vp_1)\deg(\vp_2).
    \]
    Letting $d = \deg(\vp_1)\deg(\vp_2)$, we see that $\pi$ is a $d$-to-1 map.
    Next, we consider the action of the direct product group. The direct product group $G_1 \times G_2$ acts naturally on $W$. That is, for $(y_1, y_2) \in W$ and $(\varphi_1, \varphi_2) \in G_1 \times G_2$, the action is given by
    \[
    (\varphi_1, \varphi_2) \cdot (y_1, y_2) = (\varphi_1, \varphi_2) (y_1, y_2)= (\varphi_1(y_1), \varphi_2(y_2)). 
    \]
    Since $G_1$ and $G_2$ act freely and transitively on their respective fibers by Corollary~\ref{freeaction} and Proposition~\ref{transitiveaction}, the direct product group $G_1 \times G_2$ acts freely and transitively on each fiber $\pi^{-1}(x)$ of $W$. To ensure that $W$ contains an irreducible component, we employ an ergodic theory argument.
    Since $X$ is an irreducible SFT, there exists a fully supported, shift-invariant Borel probability measure $\mu_X$ on $X$ (for instance, the measure of maximal entropy). Here, we define a Borel probability measure $\mu_W$ on $W$ by 
    \[
    \mu_W(A) = \frac{1}{d}\int_X |A \cap \pi^{-1}(x)| d\mu_X(x)
    \]
    for any Borel set $A \subseteq W$. 
    Since $\mu_X$ is a fully supported shift-invariant probability measure, so is $\mu_W$. By Lemma~\ref{non_wandering_full_supp}, we have $W = \Omega(W, \sigma_W)$. Let $G = (V, E)$ be an essential graph that presents $W$. That is, we may assume $W = \mathsf{X}_G$. From Definition~\ref{SCC}, we decompose $V$ and $E$ as $V = C_1 \sqcup \cdots \sqcup C_r$ and $E = D_1 \sqcup \cdots \sqcup D_r \sqcup T$, where we use the same notation as in Definition~\ref{SCC}.
    First, we prove that $T = \emptyset$ by contradiction. Suppose that there exists $e \in T$. Then, the cylinder set $[e]_0 = \{x \in W \mid x_{[0]} = e\}$ is an open set. Since the condensed graph $(V/{\sim_V}, T)$ has no cycles, we have $\sigma_W^n([e]_0) \cap [e]_0 = \emptyset$ for any $n \ge 1$. This means that every element of $[e]_0$ is a wandering point. That is, $[e]_0 \subseteq W \setminus \Omega(W, \sigma_W)$, which contradicts $W = \Omega(W, \sigma_W)$. Therefore, $T = \emptyset$.
    Therefore, the graph $G$ decomposes into strongly connected graphs $G_1, \dots, G_r$ as $G = G_1 \sqcup \cdots \sqcup G_r$. Thus, by setting $Z_j = \mathsf{X}_{G_j}$ for $1 \le j \le r$, the SFT $W$ decomposes into irreducible components as $W = Z_1 \sqcup \cdots \sqcup Z_r$. Since each $Z_j$ is a shift space, it is a closed subset of $W$. Furthermore, its complement $Z_j^c$ is a finite union of closed sets, which implies it is also closed. Hence, each $Z_j$ is a clopen set. Without loss of generality, we may assume that $(y_{1,0}, y_{2,0}) \in Z_1 \subseteq W$. \\
    \indent Set $Y=\mathsf{X}_{G_1}$, $y_0=(y_{1,0},y_{2,0})$, and $\vp=\pi|_{Y}:Y\to X$. Since $Y$ is the edge shift of a strongly connected finite directed graph, it is an irreducible SFT.
    It is clear that $\vp: (Y, y_0) \to (X, x_0)$ is a pointed sliding block code. We show that $\vp$ is an unramified Galois factor.
    Set $G = G_1 \times G_2$. Note that since each $\varphi \in G$ is an automorphism on $W$, it maps an irreducible component to an irreducible component. Then we define a subgroup $H \subseteq G$ by
    \[
    H=\{ \varphi:W\to W\in G\mid \varphi(Y)=Y\}. 
    \]
    The condition $\varphi(Y) = Y$ allows us to consider its restriction $\varphi|_Y: Y \to Y$. Since $\varphi|_Y \in \aut(Y)$, we can naturally regard $H$ as a subgroup of $\aut(Y)$, i.e., $H \subseteq \aut(Y)$. Since there is no danger of confusion, we simply write $\varphi$ for $\varphi|_Y$ in what follows.
    Fix $x \in X$. First, we show that $H$ acts transitively on $Y \cap \pi^{-1}(x)$. 
    Let $a, b \in Y \cap \pi^{-1}(x)$. Since $G$ acts transitively on $\pi^{-1}(x)$, there exists $\varphi \in G$ such that $\varphi(a) = b$. Since $a, b \in Y$, we have $\varphi(Y) \cap Y \neq \emptyset$. Now, since the elements of $G$ map irreducible components of $W$ to irreducible components, we have $\varphi(Y) = Y$. This means that $\varphi \in H$, which shows that $H$ acts transitively on $Y \cap \pi^{-1}(x)$. 
    Because the action is transitive, for any $y \in Y \cap \pi^{-1}(x)$, we have $Hy=Y \cap \pi^{-1}(x)$, where $Hy$ is the orbit of the group action. Furthermore, since the action of $G$ is free, the action of $H$ is also free. Hence, the stabilizer subgroup $I_H(y)$ is trivial. Therefore, by Lemma~\ref{orbitstabilizer}, we obtain
    \[
    |\vp^{-1}(x)|=|Y \cap \pi^{-1}(x)|=|H x|=|H|.
    \]
    This means that the cardinality of the fibers is constant and equal to $|H|$, which implies that $\vp$ is a constant-to-one map. \\
    \indent Next, we will show that $\vp$ is a Galois factor. We first prove the surjectivity of $\vp$. Since $\pi$ is an open map, $\vp(Y)$ is an open subset of $X$, and since $\vp$ is a sliding block code, $\vp(Y)$ is a shift space. That is, $\vp(Y)$ is a non-empty open subshift of $X$. Thus, by Lemma~\ref{clopen_subshift}, we have $\vp(Y) = X$. Therefore, $\vp$ is surjective. 
    Next, we will show that $H = \aut(Y/X, \vp)$. For any $\varphi = (\varphi_1, \varphi_2) \in H$ and any $y = (y_1, y_2) \in Y$, a straightforward calculation shows that $\vp \circ \varphi = \vp$. This immediately implies that $H \subseteq \aut(Y/X, \vp)$. 
    Now, by Proposition~\ref{cj_le_deg}, we have
    \[
    |\aut(Y/X,\vp)|\le \deg(\vp)=|H|.
    \]
    Since $H \subseteq \aut(Y/X, \vp)$, this forces the equality $H = \aut(Y/X, \vp)$. So far, we have seen that $\vp$ is an unramified factor map and that the action of $H = \aut(Y/X, \vp)$ on the fibers is transitive. Therefore, by Proposition~\ref{transitiveaction}, we conclude that $\vp$ is a Galois factor. \\
    \indent Let $\alpha_1: Y \to Y_1$ and $\alpha_2: Y \to Y_2$ be the projections. That is, we define $\alpha_1(y_1, y_2) = y_1$ and $\alpha_2(y_1, y_2) = y_2$ for each $(y_1, y_2) \in Y$. To complete the proof, it suffices to show that $(Y_1, y_{1,0}, \alpha_1, \vp_1)$ and $(Y_2, y_{2,0}, \alpha_2, \vp_2)$ are pointed unramified intermediate factors of the Galois factor $\vp: (Y, y_0) \to (X, x_0)$. We only show this for $(Y_1, y_{1,0}, \alpha_1, \vp_1)$ since the proof for the other case is analogous. It is clear that $\alpha_1: (Y, y_0) \to (Y_1, y_{1,0})$ is pointed, that $\alpha_1$ is a sliding block code, and that $\vp = \vp_1 \circ \alpha_1$ holds. 
    Let us show the surjectivity of $\alpha_1$. We first show that $\alpha_1(Y)$ is an open set in $Y_1$. Fix $y_1 \in \alpha_1(Y)$. Then, there exists $y \in Y$ such that $\alpha_1(y) = y_1$. By applying Lemma~\ref{separationlemma} to $\vp_1$, we can choose $\delta > 0$ such that $\vp_1|_{B_{Y_1}(y_1, \delta)}$ is injective on the open $\delta$-ball $B_{Y_1}(y_1, \delta)$. Set $V = \alpha_1^{-1}(B_{Y_1}(y_1, \delta))$. Then $V$ is an open neighborhood of $y$. Since $\vp$ is an open map by Lemma~\ref{openmap}, $\vp(V)$ is an open subset of $X$. Here, let us show that $\alpha_1(V) = B_{Y_1}(y_1, \delta) \cap \vp_1^{-1}(\vp(V))$. Since $\alpha_1(V) \subseteq B_{Y_1}(y_1, \delta)$ and $\vp_1(\alpha_1(V)) = \vp(V)$, it is clear that $\alpha_1(V) \subseteq B_{Y_1}(y_1, \delta) \cap \vp_1^{-1}(\vp(V))$.
    Conversely, fix $w \in B_{Y_1}(y_1, \delta) \cap \vp_1^{-1}(\vp(V))$. 
    Then, there exists $v \in V = \alpha_1^{-1}(B_{Y_1}(y_1, \delta))$ such that $\vp_1(w) = \vp(v) = \vp_1(\alpha_1(v))$. Since $w, \alpha_1(v) \in B_{Y_1}(y_1, \delta)$ and $\vp_1|_{B_{Y_1}(y_1, \delta)}$ is injective, we have $w = \alpha_1(v) \in \alpha_1(V)$. Thus, we obtain $\alpha_1(V) \supseteq B_{Y_1}(y_1, \delta) \cap \vp_1^{-1}(\vp(V))$, which proves that $\alpha_1(V) = B_{Y_1}(y_1, \delta) \cap \vp_1^{-1}(\vp(V))$. 
    Since $\vp(V)$ is an open set, $\alpha_1(V)$ is an open neighborhood of $y_1$. As $\alpha_1(V) \subseteq \alpha_1(Y)$, we see that $\alpha_1(Y)$ is an open set.
    Since $Y_1$ is an irreducible SFT and $\alpha_1(Y)$ is a subshift of $Y_1$, we have $Y_1 = \alpha_1(Y)$ by Lemma~\ref{clopen_subshift}. This completes the proof of the surjectivity of $\alpha_1$. \\
    \indent Finally, we show that $\alpha_1$ is unramified. Fix $x \in X$, and $a, b \in \vp_1^{-1}(x)$. First, we show that $|\alpha_1^{-1}(a)| = |\alpha_1^{-1}(b)|$. Since $\alpha_1$ is surjective, we have $\alpha_1^{-1}(a) \neq \emptyset$ and $\alpha_1^{-1}(b) \neq \emptyset$. Therefore, we can choose elements $y_a \in \alpha_1^{-1}(a)$ and $y_b \in \alpha_1^{-1}(b)$. Since
    \[
    \vp(y_a) = (\vp_1 \circ \alpha_1)(y_a) = \vp_1(a) = x = \vp_1(b) = (\vp_1 \circ \alpha_1)(y_b) = \vp(y_b),
    \]
    we obtain $y_a, y_b \in \vp^{-1}(x)$. Recall that $H = \mathrm{Gal}(Y/X, \vp)$ acts transitively on $\vp^{-1}(x)$. Thus, there exists $\varphi = (\varphi_1, \varphi_2) \in H$ such that $\varphi(y_a) = y_b$. For any $y \in \alpha_1^{-1}(a)$, we have $\alpha_1(y) = a$. Then, 
    \[
    \alpha_1(\varphi(y)) = \varphi_1(\alpha_1(y)) = \varphi_1(a) = \alpha_1(\varphi(y_a)) = \alpha_1(y_b) = b.
    \]
    This implies that $\varphi(\alpha_1^{-1}(a)) \subseteq \alpha_1^{-1}(b)$. By applying a similar argument to $\varphi^{-1} = (\varphi_1^{-1}, \varphi_2^{-1})$, it also follows that $\varphi^{-1}(\alpha_1^{-1}(b)) \subseteq \alpha_1^{-1}(a)$. Consequently, the restrictions $\varphi|_{\alpha_1^{-1}(a)}: \alpha_1^{-1}(a) \to \alpha_1^{-1}(b)$ and $\varphi^{-1}|_{\alpha_1^{-1}(b)}: \alpha_1^{-1}(b) \to \alpha_1^{-1}(a)$ are well-defined and are inverses of each other. That is, $\varphi^{-1}|_{\alpha_1^{-1}(b)}$ is a bijection, which proves that $|\alpha_1^{-1}(a)| = |\alpha_1^{-1}(b)|$. 
    For each $x \in X$, we define $c_x = |\alpha_1^{-1}(a)|$, where $a \in \vp_1^{-1}(x)$. As shown above, $c_x$ is independent of the choice of $a \in \vp_1^{-1}(x)$. If we show that $c_x$ is a constant independent of $x$, it follows that $\alpha_1$ is unramified. Fix $x \in X$. Since $\vp = \vp_1 \circ \alpha_1$, we have
    \[
    \vp^{-1}(x)=\bigsqcup_{a\in\vp_1^{-1}(x)}\alpha_1^{-1}(a). 
    \]
    Therefore
    \[
    \deg(\vp)=|\vp^{-1}(x)|=
        \sum_{a\in\vp_1^{-1}(x)}|\alpha_1^{-1}(a)|
        =c_x|\vp_1^{-1}(x)|
        =c_x\deg(\vp_1)
    \]
    holds. Hence, we have
    \[
    c_x=\frac{ \deg(\vp)}{\deg(\vp_1)}, 
    \]
    which implies that $c_x$ is a constant independent of $x$. Therefore, $\alpha_1$ is unramified, and this completes the proof.
\end{proof}

\begin{lemma}\label{transition_map1}
    Let $(X, x_0)$, $(Y_1, y_{1,0})$, and $(Y_2, y_{2,0})$ be pointed irreducible SFTs. Let $\vp_1: (Y_1, y_{1,0}) \to (X, x_0)$ and $\vp_2: (Y_2, y_{2,0}) \to (X, x_0)$ be pointed Galois factors. Furthermore, let $\alpha: (Y_2, y_{2,0}) \to (Y_1, y_{1,0})$ be a pointed unramified factor map with $\vp_2 = \vp_1 \circ \alpha$. Then, for any $\varphi \in \gal(Y_2/X, \vp_2)$, there exists a unique $\psi \in \gal(Y_1/X, \vp_1)$ such that $\psi \circ \alpha = \alpha \circ \varphi$.
\end{lemma}

\begin{proof}
    Fix $\varphi\in \gal(Y_2/X, \vp_2)$. First, we prove the existence. Since $y_{1,0},\alpha(\varphi(y_{2,0}))\in \vp_1^{-1}(x_0)$, and the Galois group $\gal(Y_1/X, \vp_1)$ acts transitively on the fiber (Proposition~\ref{transitiveaction}), there exists a $\psi\in \gal(Y_1/X, \vp_1)$ such that $\psi(y_{1,0})=\alpha(\varphi(y_{2,0}))$. Then, we have $(\psi\circ\alpha)(y_{2,0})=(\alpha\circ\varphi)(y_{2,0})$. Furthermore, we have
    \[
        \vp_1\circ(\psi\circ\alpha)= 
        (\vp_1\circ\psi)\circ\alpha=
        \vp_1\circ\alpha=\vp_2
    \]
    \[
        \vp_1\circ(\alpha\circ\varphi)=
        (\vp_1\circ\alpha)\circ\varphi=
        \vp_2\circ\varphi=\vp_2.
    \]
    That is, both $\psi\circ\alpha$ and $\alpha\circ\varphi$ are lifts of $\vp_2$ along $\vp_1$. Therefore, by the uniqueness of lifts (Theorem~\ref{liftunique}), we obtain $\psi\circ\alpha=\alpha\circ\varphi$. This proves the existence.\\
    \indent Next, we prove the uniqueness. Suppose that $\psi, \psi' \in \gal(Y_1/X, \vp_1)$ satisfy $\psi \circ \alpha = \alpha \circ \varphi = \psi' \circ \alpha$. Since $\alpha$ is surjective, we have $\psi = \psi'$. This means uniqueness. 
\end{proof}

\begin{definition}\label{def_p_alpha}
    Let $(X,x_0)$, $(Y_1,y_{1,0})$, and $(Y_2,y_{2,0})$ be pointed irreducible SFTs. Let $\vp_1: (Y_1, y_{1,0}) \to (X, x_0)$ and $\vp_2: (Y_2, y_{2,0}) \to (X, x_0)$ be pointed Galois factors. Let $\alpha: (Y_2, y_{2,0}) \to (Y_1, y_{1,0})$ be a pointed unramified factor map such that $\vp_2 = \vp_1 \circ \alpha$.
    Then, by Lemma~\ref{transition_map1}, for any $\varphi \in \gal(Y_2/X, \vp_2)$, there exists a unique $\psi \in \gal(Y_1/X, \vp_1)$ such that $\psi \circ \alpha = \alpha \circ \varphi$. We denote $p_\alpha(\varphi)=\psi$. This defines a map $p_\alpha: \gal(Y_2/X, \vp_2) \to \gal(Y_1/X, \vp_1)$.
\end{definition}

\begin{lemma}
    Under the same assumptions as above, the map 
    \[
    p_\alpha: \gal(Y_2/X, \vp_2) \to \gal(Y_1/X, \vp_1)
    \]
    is a surjective group homomorphism. 
\end{lemma}

\begin{proof}
    First, we show that $p_\alpha$ is a group homomorphism. Fix $\varphi,\varphi'\in \gal(Y_2/X, \vp_2)$. Then we have
    \begin{align*}
        \alpha\circ(\varphi\circ\varphi')&=(\alpha\circ\varphi)\circ\varphi'\\
        &=(p_\alpha(\varphi)\circ\alpha)\circ
        \varphi'\\
        &=p_\alpha(\varphi)\circ(\alpha\circ
        \varphi')\\
        &=p_\alpha(\varphi)\circ (p_\alpha(\varphi')\circ\alpha)\\
        &=(p_\alpha(\varphi)\circ p_\alpha(\varphi'))\circ\alpha.
    \end{align*}
    Therefore, we obtain $p_\alpha(\varphi\circ\varphi')=p_\alpha(\varphi)\circ p_\alpha(\varphi')$. \\
    \indent Next, we show the surjectivity. Fix $\psi\in \gal(Y_1/X, \vp_1)$. Since $\alpha$ is surjective, there exists a $y_2 \in Y_2$ such that $\alpha(y_2) = \psi(y_{1,0})$. Since $\vp_2 = \vp_1 \circ \alpha$ and $\psi(y_{1,0}) \in \vp_1^{-1}(x_0)$, we have $y_2 \in \vp_2^{-1}(x_0)$. Furthermore, since $y_{2,0} \in \vp_2^{-1}(x_0)$ and $\gal(Y_2/X, \vp_2)$ acts transitively on $\vp_2^{-1}(x_0)$ by Proposition~\ref{transitiveaction}, there exists a $\varphi \in \gal(Y_2/X, \vp_2)$ such that $\varphi(y_{2,0}) = y_2$.
    Here, we have
    \[
    (\alpha\circ\varphi)(y_{2,0})=\alpha(y_2)=\psi(y_{1,0})=(\psi\circ\alpha)(y_{2,0}). 
    \]
    By a similar argument to the proof of Lemma~\ref{transition_map1}, both $\alpha \circ \varphi$ and $\psi \circ \alpha$ are lifts of $\vp_2$ along $\vp_1$. Therefore, by the uniqueness of lifts (Theorem~\ref{liftunique}), we obtain $\alpha \circ \varphi = \psi \circ \alpha$. This means that $p_\alpha(\varphi) = \psi$.
\end{proof}

\begin{definition}\label{induce_mor}
    Let $(X, x_0)$, $(Y, y_0)$, and $(Y', y'_0)$ be pointed irreducible SFTs. Assume that $(Y, y_0)$ and $(Y', y'_0)$ are conjugate via a pointed conjugacy $f: (Y, y_0) \to (Y', y'_0)$. Let $\vp: (Y, y_0) \to (X, x_0)$ and $\vp': (Y', y'_0) \to (X, x_0)$ be pointed Galois factors. Then, $f$ induces an isomorphism between the Galois groups $\gal(Y/X,\vp)$ and $\gal(Y'/X,\vp')$. Specifically, the group isomorphism $\gal(Y/X,\vp) \to \gal(Y'/X,\vp')$ is given by $\varphi \mapsto f \circ \varphi \circ f^{-1}$.
\end{definition}

\begin{lemma}\label{well_def_transition_map}
    Let $(X, x_0)$, $(Y_1, y_{1,0})$, $(Y'_1, y'_{1,0})$, $(Y_2, y_{2,0})$, and $(Y'_2, y'_{2,0})$ be pointed irreducible SFTs.
    Assume that $(Y_1, y_{1,0})$ and $(Y'_1, y'_{1,0})$ are conjugate via a pointed conjugacy $f_1: (Y_1, y_{1,0}) \to (Y'_1, y'_{1,0})$.
    Similarly, assume that $(Y_2, y_{2,0})$ and $(Y'_2, y'_{2,0})$ are conjugate via a pointed conjugacy $f_2: (Y_2, y_{2,0}) \to (Y'_2, y'_{2,0})$.
    Let $\vp_1: (Y_1, y_{1,0}) \to (X, x_0)$, $\vp'_1: (Y'_1, y'_{1,0}) \to (X, x_0)$, $\vp_2: (Y_2, y_{2,0}) \to (X, x_0)$, and $\vp'_2: (Y'_2, y'_{2,0}) \to (X, x_0)$ be pointed Galois factors with $\vp_1 = \vp'_1 \circ f_1$ and $\vp_2 = \vp'_2 \circ f_2$.
    Let $\alpha: (Y_2, y_{2,0}) \to (Y_1, y_{1,0})$ and $\alpha': (Y'_2, y'_{2,0}) \to (Y'_1, y'_{1,0})$ be pointed unramified factor maps with $\vp_2 = \vp_1 \circ \alpha$ and $\vp'_2 = \vp'_1 \circ \alpha'$.
    Then, for the group isomorphisms induced by $f_1$ and $f_2$, namely $\Phi_1: \gal(Y_1/X, \vp_1) \to \gal(Y'_1/X, \vp'_1)$ given by $\varphi \mapsto f_1 \circ \varphi \circ f_1^{-1}$ and $\Phi_2: \gal(Y_2/X, \vp_2) \to \gal(Y'_2/X, \vp'_2)$ given by $\varphi \mapsto f_2 \circ \varphi \circ f_2^{-1}$, we have $\Phi_1 \circ p_\alpha = p_{\alpha'} \circ \Phi_2$. That is, the following diagram commutes:
    \[
    \begin{tikzcd}[row sep=large, column sep=large]
        \gal(Y_2/X, \vp_2)\arrow[r, "\Phi_2"] \arrow[d, "p_\alpha"'] & \gal(Y'_2/X, \vp'_2) \arrow[d, "p_{\alpha'}"] \\
        \gal(Y_1/X,\vp_1)\arrow[r, "\Phi_1"] &\gal(Y'_1/X, \vp'_1)
    \end{tikzcd}
    \]
\end{lemma}

\begin{proof}
    First, we show $f_1 \circ\alpha=\alpha'\circ f_2$. A direct calculation shows that $(f_1 \circ \alpha)(y_{2,0})= (\alpha' \circ f_2)(y_{2,0})$. 
    Since both $f_1 \circ \alpha$ and $\alpha' \circ f_2$ are lifts of $\vp_2$ along $\vp'_1$, by Theorem~\ref{liftunique}, we obtain $f_1 \circ \alpha = \alpha' \circ f_2$. \\
    Fix $\varphi\in \gal(Y_2/X, \vp_2)$. Set $\psi = p_\alpha(\varphi)$. By the definition of $p_\alpha$, we have $\alpha \circ \varphi = \psi \circ \alpha$. Furthermore, set $\varphi' = \Phi_2(\varphi) = f_2 \circ \varphi \circ f_2^{-1}$. It immediately follows that $\varphi' \circ f_2 = f_2 \circ \varphi$. Moreover, from $\Phi_1(\psi) = f_1 \circ \psi \circ f_1^{-1}$, we obtain $f_1 \circ \psi = \Phi_1(\psi) \circ f_1$.
    Then, we have
    \begin{align*}
        \alpha'\circ\varphi'\circ f_2&=\alpha'\circ f_2\circ\varphi\\
        &=f_1\circ \alpha \circ\varphi\\
        &=f_1\circ\psi\circ\alpha\\
        &=\Phi_1(\psi)\circ f_1\circ\alpha\\
        &=\Phi_1(\psi)\circ\alpha'\circ f_2.
    \end{align*}
    Since $f_2$ is bijective, we obtain $\alpha' \circ \varphi' = \Phi_1(\psi) \circ \alpha'$. Thus, we have $\alpha' \circ \Phi_2(\varphi) = \Phi_1(p_\alpha(\varphi)) \circ \alpha'$, which implies that $p_{\alpha'}(\Phi_2(\varphi)) = \Phi_1(p_\alpha(\varphi))$. Therefore, $\Phi_1 \circ p_\alpha = p_{\alpha'} \circ \Phi_2$ holds. 
\end{proof}

\begin{lemma}\label{projective_system}
    Let $(X,x_0)$, $(Y_1, y_{1,0})$, $(Y_2, y_{2,0})$, and $(Y_3, y_{3,0})$ be pointed irreducible SFTs. Let $\vp_1: (Y_1, y_{1,0}) \to (X, x_0)$, $\vp_2: (Y_2, y_{2,0}) \to (X, x_0)$, and $\vp_3: (Y_3, y_{3,0}) \to (X, x_0)$ be pointed Galois factors. Let $\alpha: (Y_2, y_{2,0}) \to (Y_1, y_{1,0})$ and $\beta: (Y_3, y_{3,0}) \to (Y_2, y_{2,0})$ be pointed unramified factor maps such that $\vp_2 = \vp_1 \circ \alpha$ and $\vp_3 = \vp_2 \circ \beta$. Then, we have $p_{\alpha \circ \beta} = p_\alpha \circ p_\beta$. 
\end{lemma}

\begin{proof}
    Fix $\varphi\in \gal(Y_3/X,\vp_3)$. 
    Then we have  
    \[
    (\alpha\circ\beta)\circ\varphi=\alpha\circ p_\beta(\varphi)\circ\beta=p_\alpha(p_\beta(\varphi))\circ(\alpha\circ\beta).
    \]
    This means that $p_{\alpha \circ \beta}(\varphi)=(p_\alpha \circ p_\beta)(\varphi)$. 
\end{proof}

\begin{definition}
    Let $(X, x_0)$ be a pointed irreducible SFT. A pointed Galois factor over $(X, x_0)$ is a triple $(Y, y_0, \vp)$ consisting of a pointed irreducible SFT $(Y, y_0)$ and a pointed Galois factor map $\vp: (Y, y_0) \to (X, x_0)$.
    Let $\mathcal{I}(X, x_0)$ be the set of conjugacy classes (isomorphism classes) of pointed Galois factors over $(X, x_0)$. That is, for $[(Y, y_0, \vp)], [(Y', y'_0, \vp')] \in \mathcal{I}(X, x_0)$, we say $[(Y, y_0, \vp)] = [(Y', y'_0, \vp')]$ if there exists a pointed conjugacy $\gamma: (Y, y_0) \to (Y', y'_0)$ such that $\vp = \vp' \circ \gamma$.
    Let $i = [(Y_1, y_{1,0}, \vp_1)]$ and $j = [(Y_2, y_{2,0}, \vp_2)] \in \mathcal{I}(X, x_0)$. We define $i \le j$ if there exists a pointed unramified factor map $\alpha: (Y_2, y_{2,0}) \to (Y_1, y_{1,0})$ such that $\vp_2 = \vp_1 \circ \alpha$. In other words, we write $i \le j$ when $(Y_1, y_{1,0}, \alpha, \vp_1)$ is a pointed unramified intermediate factor of $\vp_2$. This defines a partial order $\le$ on $\mathcal{I}(X, x_0)$ (it is straightforward to verify that this relation is a partial order). Furthermore, by Lemma~\ref{directed_set}, the poset $(\mathcal{I}(X, x_0), \le)$ is a directed set. \\
    \indent For each $i \in \mathcal{I}(X, x_0)$, we fix a representative of $i$ and denote it by $(Y_i, y_{i,0}, \vp_i)$. We denote the corresponding Galois group by $G_i = \gal(Y_i/X)$. When $i \le j$, there exists a pointed unramified factor map from $(Y_j, y_{j,0})$ to $(Y_i, y_{i,0})$, which is unique by Lemma~\ref{medium_unique}. We denote this map by $\alpha_{ij}$. We simply write $p_{ij}$ for the induced surjective group homomorphism $p_{\alpha_{ij}}: G_j \to G_i$. By Lemma~\ref{projective_system}, the pair $(\{G_i\}_{i\in \mathcal{I}(X, x_0)}, \{p_{ij}: G_j \to G_i\}_{i\le j})$ forms a projective system. We define $G_{X,x_0}$ to be the projective limit of this system. That is, we define
    \[
    G_{X,x_0}=\varprojlim G_i
        =\left\{ (\varphi_i)_{i\in \mathcal{I}(X, x_0)}\in \prod_{i\in \mathcal{I}(X, x_0)}G_i\ \Big| \ \varphi_i =p_{ij}(\varphi_j)\text{ for any }i\le j\right\}.
    \]
    We call $G_{X,x_0}$ the absolute Galois group of $(X,x_0)$.
    Since each $G_i$ is a finite group, by endowing it with the discrete topology, the absolute Galois group acquires a topology induced by the product topology. In other words, just as in the Galois theory of fields, the absolute Galois group is a profinite group.\\
    \indent This construction of the absolute Galois group is well-defined. Indeed, for each $i \in \mathcal{I}(X, x_0)$, consider another representative $(Y'_i, y'_{i,0}, \vp'_i)$. Let $G'_i = \gal(Y'_i/X, \vp'_i)$, and construct a projective system $(\{G'_i\}_{i\in \mathcal{I}(X, x_0)}, \{p'_{ij}: G'_j \to G'_i\}_{i\le j})$ following the same procedure as above. By Definition~\ref{induce_mor}, the group isomorphism $\Phi_i: G_i \to G'_i$ is determined for each $i \in \mathcal{I}(X, x_0)$. By Lemma~\ref{well_def_transition_map}, the family $\{\Phi_i: G_i \to G'_i\}_{i\in \mathcal{I}(X, x_0)}$ gives an isomorphism of projective systems between $(\{G_i\}_{i\in \mathcal{I}(X, x_0)}, \{p_{ij}\}_{i\le j})$ and $(\{G'_i\}_{i\in \mathcal{I}(X, x_0)}, \{p'_{ij}\}_{i\le j})$, which implies that their projective limits are isomorphic as topological groups.
\end{definition}

\begin{remark}
    From the characterization of profinite groups, it follows that the absolute Galois group is a compact, totally disconnected, Hausdorff group in which the family of open normal subgroups forms a fundamental system of neighborhoods at $1$. For a proof of this fact, see, for example, \cite[Chapter I, \S 1, Theorem 2]{ShatzShatz2016}. 
\end{remark}

We show that the Galois group is independent of the base point.

\begin{lemma}\label{base_point_change}
    Let $X$ be an irreducible SFT, and $x_0, x'_0 \in X$. Then, the absolute Galois groups $G_{X,x_0}$ and $G_{X,x'_0}$ are isomorphic as topological groups.
\end{lemma}

\begin{proof}
    For each $i \in \mathcal{I}(X, x_0)$, we fix a representative and denote it by $(Y_i, y_{i,0}, \vp_i)$. Let $F_i = \vp_i^{-1}(x'_0)$ be the fiber. When $i \le j$, since $\vp_j = \vp_i \circ \alpha_{ij}$, $\alpha_{ij}$ maps $F_j$ to $F_i$. It directly follows that $(\{F_i\}_{i\in \mathcal{I}(X, x_0)}, \{\alpha_{ij}|_{F_j}: F_j \to F_i\})$ is a projective system. Therefore, we can consider its projective limit $\varprojlim F_i$. Since each $F_i$ is a finite set and the projective limit of finite sets is non-empty, we can choose an element $(y'_{i,0})_{i \in \mathcal{I}(X,x_0)} \in \varprojlim F_i$. Then, for each $i \in \mathcal{I}(X, x_0)$, we have $y'_{i,0} \in \vp_i^{-1}(x'_0)$, and if $i \le j$, we have $\alpha_{ij}(y'_{j,0}) = y'_{i,0}$. This defines a map $\Psi: \mathcal{I}(X, x_0) \to \mathcal{I}(X, x'_0)$ by $\Psi([(Y_i, y_{i,0}, \vp_i)]) = [(Y_i, y'_{i,0}, \vp_i)]$. \\
    \indent First, we show that this map is well-defined. 
    Fix $i \in \mathcal{I}(X, x_0)$, and two representatives $(Y^{(1)}_i, y^{(1)}_{i,0}, \vp^{(1)}_i)$ and $(Y^{(2)}_i, y^{(2)}_{i,0}, \vp^{(2)}_i)$ of $i$. Then, there exists a pointed conjugacy $f: (Y^{(1)}_i, y^{(1)}_{i,0}) \to (Y^{(2)}_i, y^{(2)}_{i,0})$ such that $\vp^{(1)}_i = \vp^{(2)}_i \circ f$. Set $F^{(1)}_i = (\vp^{(1)}_i)^{-1}(x'_0)$ and $F^{(2)}_i = (\vp^{(2)}_i)^{-1}(x'_0)$, and fix elements $(y'^{(1)}_{i,0})_{i \in \mathcal{I}(X,x_0)} \in \varprojlim F_i^{(1)}$ and $(y'^{(2)}_{i,0})_{i \in \mathcal{I}(X,x_0)} \in \varprojlim F_i^{(2)}$. Then, we have 
    \[
    \vp^{(2)}_i(f(y'^{(1)}_{i,0})) = \vp^{(1)}_i(y'^{(1)}_{i,0}) = x'_0,
    \]
    which implies that $f(y'^{(1)}_{i,0}) \in F^{(2)}_i$. Thus, $f: (Y^{(1)}_i, y'^{(1)}_{i,0}) \to (Y^{(2)}_i, f(y'^{(1)}_{i,0}))$ is a pointed conjugacy, which means that 
    \[
    [(Y^{(1)}_i, y'^{(1)}_{i,0}, \vp^{(1)}_i)] = [(Y^{(2)}_i, f(y'^{(1)}_{i,0}), \vp^{(2)}_i)].
    \]
    
    Next, since $f(y'^{(1)}_{i,0}), y'^{(2)}_{i,0} \in F^{(2)}_i$, Proposition~\ref{transitiveaction} ensures that the Galois group $\gal(Y^{(2)}_i/X, \vp^{(2)}_i)$ acts transitively on $F^{(2)}_i$. Thus, there exists a $\varphi \in \gal(Y^{(2)}_i/X, \vp^{(2)}_i)$ such that $\varphi(f(y'^{(1)}_{i,0})) = y'^{(2)}_{i,0}$. This implies that $\varphi: (Y^{(2)}_i, f(y'^{(1)}_{i,0})) \to (Y^{(2)}_i, y^{(2)}_{i,0})$ is a pointed conjugacy satisfying $\vp^{(2)}_i \circ \varphi = \vp^{(2)}_i$.
    Consequently, we have 
    \[
    [(Y^{(2)}_i, f(y'^{(1)}_{i,0}), \vp^{(2)}_i)] = [(Y^{(2)}_i, y'^{(2)}_{i,0}, \vp^{(2)}_i)].
    \]
    From this, we obtain $[(Y^{(1)}_i, y'^{(1)}_{i,0}, \vp^{(1)}_i)] = [(Y^{(2)}_i, y'^{(2)}_{i,0}, \vp^{(2)}_i)]$, which establishes the well-definedness.\\
    \indent Next, we show that $\Psi: \mathcal{I}(X, x_0) \to \mathcal{I}(X, x'_0)$ is an order isomorphism. Now, by reversing the roles of $x_0$ and $x'_0$, we can similarly construct a map $\Psi': \mathcal{I}(X, x'_0) \to \mathcal{I}(X, x_0)$. By construction, $\Psi$ and $\Psi'$ are mutually inverse maps. Thus, $\Psi$ is bijective. Next, we show that $\Psi$ preserves the order. Suppose that $i \le j$. Then, we have $\alpha_{ij}(y'_{j,0}) = y'_{i,0}$, which implies that $\alpha_{ij}: (Y_j, y'_{j,0}) \to (Y_i, y'_{i,0})$ is a pointed factor map. That is, we have $\Psi(i) \le \Psi(j)$. \\
    \indent Consequently, the index sets $\mathcal{I}(X, x_0)$ and $\mathcal{I}(X, x'_0)$ of the projective systems in the definition of the Galois group are isomorphic. Since the Galois groups $G_i$ and the maps $p_{ij}$ are independent of the base point, the isomorphism of the index sets implies that their projective limits coincide. 
\end{proof}

\begin{definition}
    By Lemma~\ref{base_point_change}, the absolute Galois group is independent of the choice of base point. Therefore, when there is no need to specify the base point, we simply denote the absolute Galois group by $G_X$.
\end{definition}

Next, we show that the Galois group is a conjugacy invariant.

\begin{theorem}\label{gal_invariant}
    Let $X$ and $X'$ be irreducible SFTs. If $X$ and $X'$ are topologically conjugate, then their absolute Galois groups $G_X$ and $G_{X'}$ are isomorphic as topological groups. 
\end{theorem}

\begin{proof}
    Let $f: X \to X'$ be a conjugacy. Fix a base point $x_0 \in X$ and set $x'_0 = f(x_0)$. Since the absolute Galois group is independent of the base point, it suffices to show that $G_{X,x_0}$ and $G_{X',x'_0}$ are isomorphic. A straightforward calculation shows that for $[(Y, y_0, \vp)] \in \mathcal{I}(X, x_0)$, the triple $(Y, y_0, f \circ \vp)$ is a pointed irreducible Galois factor over $(X', x'_0)$. Thus, we define a map $\Psi: \mathcal{I}(X, x_0) \to \mathcal{I}(X', x'_0)$ by $\Psi([(Y, y_0, \vp)]) = [(Y, y_0, f \circ \vp)]$. This map is well-defined, and $\Psi$ is an order isomorphism.
    Now, considering the two projective systems $(\{G_i\}_{i\in \mathcal{I}(X, x_0)}, \{p_{ij}: G_j \to G_i\}_{i\le j})$ and $(\{G'_i\}_{i\in \mathcal{I}(X', x'_0)}, \{p'_{ij}: G'_j \to G'_i\}_{i\le j})$, their index sets are isomorphic. Furthermore, for each conjugacy class, we have $\gal(Y_i/X', f \circ \vp_i) = \gal(Y_i/X, \vp_i)$, and for $i\le j$, we have $p_{ij} = p'_{ij}$. Then, these projective systems are identical, and therefore their projective limits also coincide.   
\end{proof}

    We now consider subgroups of the absolute Galois group. For each $i\in \mathcal{I}(X,x_0)$, let $\mathrm{pr}_i:G_{X,x_0}\to G_i$ denote the projection onto the $i$-th coordinate.

\begin{lemma}\label{open_normal_subgroup}
    Let $(X, x_0)$ be a pointed SFT. Then, for any open normal subgroup $N$ of $G_{X,x_0}$, there exists a unique $[(Z, z_0, \vp)] \in \mathcal{I}(X, x_0)$ such that $\gal(Z/X, \vp) \cong G_{X,x_0}/N$. 
\end{lemma}

\begin{proof}
    Let $N$ be an open normal subgroup of $G_{X,x_0}$.
    We first show the existence. Recall from Subsection~\ref{section_profinite_cohom} that $\{\mathrm{Ker}(\mathrm{pr}_i)\}_{i\in\mathcal{I}(X,x_0)}$ is a neighborhood basis at the identity element in $G_{X,x_0}$. Thus, there exists an $i\in\mathcal{I}(X,x_0)$ such that $\mathrm{Ker}(\mathrm{pr}_i)\subseteq N$. Let $(Y_i,y_{i,0},\vp_i)$ be a representative of $i$, and set $G_i=\gal(Y_i/X, \vp_i)$. Since the transition maps of the projective system are surjective, $\mathrm{pr}_i$ is also surjective, and $N_i=\mathrm{pr}_i(N)$ is a normal subgroup of $G_i$. 
    By Theorem~\ref{normal_subgroup}, for the intermediate factor $(Z,\alpha,\beta)$ corresponding to $N_i$, the map $\beta:Z\to X$ is a Galois factor, and $\gal(Z/X,\beta)\cong G_i/N_i$. Set $z_0=\alpha(y_{i,0})$. 
    Then $\beta:(Z,z_0)\to (X,x_0)$ preserves the base point. Let $\pi:G_i\to G_i/N_i$ be the natural projection. Consider the map $\Phi=\pi\circ\mathrm{pr}_i:G_{X,x_0}\to G_i/N_i$. Since $\pi$ and $\mathrm{pr}_i$ are surjective homomorphisms, $\Phi$ is also a surjective homomorphism. Furthermore, a straightforward calculation shows that $\mathrm{Ker}\ \Phi = N$. Applying the first isomorphism theorem to $\Phi$, we have
    \[
    G_{X,x_0}/N \cong G_i/N_i \cong \gal(Z/X, \beta).
    \]
    This proves the existence.\\
    \indent From this proof of existence, we see that a natural surjective homomorphism from $G_{X,x_0}$ to $\gal(Z/X, \vp)$ is constructed. The kernel of this natural surjective homomorphism is $N$.\\
    \indent Next, we show the uniqueness. Let $(Z_1, z_{1,0}, \vp_1)$ and $(Z_2, z_{2,0}, \vp_2)$ be two such pointed Galois factors over $(X, x_0)$, and let $j_1, j_2 \in \mathcal{I}(X, x_0)$ be their respective conjugacy classes. Let $q_1: G_X \to \gal(Z_1/X, \vp_1)$ and $q_2: G_X \to \gal(Z_2/X, \vp_2)$ be the natural surjective homomorphisms from the absolute Galois group. 
    Then we have $\mathrm{Ker}(q_1) = \mathrm{Ker}(q_2) = N$. Since $(\mathcal{I}(X, x_0), \le)$ is a directed set by Lemma~\ref{directed_set}, there exists a $k \in \mathcal{I}(X, x_0)$ such that $j_1 \le k$ and $j_2 \le k$. Let $(Y, y_0, \vp)$ be a representative of this $k$, and set $G = \gal(Y/X, \vp)$. Furthermore, take $\alpha_1: Y \to Z_1$ and $\alpha_2: Y \to Z_2$ such that $(Z_1, z_{1,0}, \alpha_1, \vp_1)$ and $(Z_2, z_{2,0}, \alpha_2, \vp_2)$ are intermediate factors preserving the base points. Letting $p_{j_1, k}: G \to \gal(Z_1/X, \vp_1)$ and $p_{j_2, k}: G \to \gal(Z_2/X, \vp_2)$ be the transition maps, we have $q_1 = p_{j_1, k} \circ \mathrm{pr}_k$ and $q_2 = p_{j_2, k} \circ \mathrm{pr}_k$. Here
    \[
    \mathrm{pr}_k^{-1}(\mathrm{Ker}(p_{j_1, k}))=\mathrm{Ker}(q_1)=N=\mathrm{Ker}(q_2)=\mathrm{pr}_k^{-1}(\mathrm{Ker}(p_{j_2, k}))
    \]
    holds and since $\mathrm{pr}_k$ is surjective, we obtain
    \[
    \mathrm{Ker}(p_{j_1,k}) = \mathrm{pr}_k(N) = \mathrm{Ker}(p_{j_2,k}).
    \]
    Recalling the definition of the transition maps in Definition~\ref{def_p_alpha} and Definition~\ref{def_midfactor_to_gal}, we have
    \[
    H(Z_1,\alpha_1,\vp_1)=\mathrm{Ker}(p_{j_1, k})=\mathrm{Ker}(p_{j_2, k})=H(Z_2,\alpha_2,\vp_2). 
    \]
    Theorem~\ref{conjugatesamegroupequiv} implies that $(Z_1, \alpha_1, \vp_1)$ and $(Z_2, \alpha_2, \vp_2)$ are conjugate as intermediate factors. Since the conjugacy map $f: Z_1 \to Z_2$ as intermediate factors satisfies $f \circ \alpha_1 = \alpha_2$, we have
    \[
    f(z_{1,0})=f(\alpha_1(y_0))=\alpha_2(y_0)=z_{2,0},
    \]
    which implies that $(Z_1, z_{1,0}, \alpha_1, \vp_1)$ and $(Z_2, z_{2,0}, \alpha_2, \vp_2)$ are conjugate as pointed intermediate factors. This completes the proof of uniqueness.
    \end{proof}

\subsection{Calculation of the Absolute Galois Group of the 1-Full Shift}

In general, it seems difficult to calculate the absolute Galois group directly from the definition. However, the absolute Galois group of the 1-full shift, the simplest shift space, can be easily calculated. \\
\indent Let $X = \Sigma_1 = \{0^\infty\}$. The only possible choice for the base point is $0^\infty$. Let us consider the form of a Galois factor over $(\Sigma_1, 0^\infty)$. Let $(Y,y_0 \in Y)$ be a pointed irreducible SFT, and let $\vp: Y \to X$ be a pointed Galois factor with $\deg(\vp) = d$. Since $Y = \vp^{-1}(0^\infty)$, $Y$ is a $d$-point set.
Since $Y$ is irreducible, the shift map must cycle through all the points, that is, $O_{\sigma_Y}(y_0)=Y$. This means that there is only one Galois factor of degree $n$ up to conjugacy. Since the elements of $\aut(Y)$ can only be powers of the shift map, we have
\[
    \aut(Y)=\{\id_Y,\sigma_Y,\sigma_Y^2,\dots,\sigma_Y^{d-1}\}=\langle\sigma_Y\rangle\cong\Z/d\Z.
\]
Thus, $|\aut(Y)|=d$. From $\gal(Y/X,\vp)\subseteq\aut(Y)$ and Proposition~\ref{cj_le_deg}, we obtain
\[
    \gal(Y/X,\vp)=\aut(Y)\cong\Z/d\Z. 
\]
Next, we consider the structure of $\mathcal{I}(\Sigma_1, 0^\infty)$. Let $n, m \in \N$. Let $(Y_n, y_{n,0}, \vp_n)$ and $(Y_m, y_{m,0}, \vp_m)$ be representatives of the conjugacy classes of Galois factors of degrees $n$ and $m$, respectively. Since $Y_n$ and $Y_m$ are sets of $n$ and $m$ points respectively, an unramified factor map $\alpha: Y_m \to Y_n$ that preserves the base points and satisfies $\vp_m = \vp_n \circ \alpha$ exists if and only if $n \mid m$. This means that $\mathcal{I}(\Sigma_1, 0^\infty)$ is order-isomorphic to $\Z$ equipped with the divisibility relation $n \mid m$. Therefore, we obtain
\[
    G_{X}\cong\varprojlim_n\Z/n\Z=\widehat{\Z}.
\]
Thus, we obtain the following theorem.

\begin{theorem}\label{one_fullshift_abgal}
    The absolute Galois group $G_{\Sigma_1}$ is isomorphic to the profinite integers $\widehat{\Z}$. 
\end{theorem}

\begin{remark}
    The calculation in Theorem~\ref{one_fullshift_abgal} is analogous to the calculation of the absolute Galois group of $\mathbb{F}_p$ in the Galois theory of fields. The Galois group of a finite extension over $\mathbb{F}_p$ is generated by the Frobenius map $\mathrm{Frob}_p$. In Theorem~\ref{one_fullshift_abgal}, the Galois group is generated by the shift map $\sigma_Y$, which shows that the Frobenius map and the shift map play similar roles. 
    There is another example where the shift map acts similarly to the Frobenius map. Recall Example~\ref{ex_z_over_nz}. Here, we defined $\vp(y)_i=y_{i+1}-y_i$, which means that
    \[
    \vp(y)=\sigma_Y(y)-y.
    \]
    On the other hand, recall Artin-Schreier theory in the Galois theory of fields. In this theory, the map
    \[
    \wp(x)=x^p-x=\mathrm{Frob}_p(x)-x
    \]
    played a key role. 
\end{remark}

\subsection{Representation of the Automorphism Group}
In this subsection, we construct a representation of the automorphism group of SFTs into the Galois group.

\begin{lemma}\label{forward_gal_factor}
    Let $X$ be an irreducible SFT, $x_0 \in X$ be a base point, and $f \in \aut(X)$ be an automorphism of $X$. Set $x_1 = f(x_0)$. Then, if $(Y, y_0, \vp)$ is a pointed Galois factor over $(X, x_0)$, the composition map $f \circ \vp: (Y, y_0) \to (X, x_1)$ is a pointed Galois factor over $(X, x_1)$, and their Galois groups coincide. That is, we have
    \[ 
    \gal(Y/X,f\circ\vp)=\gal(Y/X, \vp) .
    \]
\end{lemma}

\begin{proof}
    In the proof of Theorem~\ref{gal_invariant}, it suffices to let $X' = X$.
\end{proof}

\begin{definition}
    Let $X$ be an irreducible SFT, $x_0 \in X$, and $f \in \aut(X)$. Set $x_1 = f(x_0)$. We define the order isomorphism $f_*: \mathcal{I}(X, x_0) \to \mathcal{I}(X, x_1)$ induced by $f$ as
    \[
   f_*([(Y,y_0,\vp)])=[(Y,y_0,f\circ\vp)].
   \]
    Similarly to the proof of Theorem~\ref{gal_invariant}, this is well-defined. By Lemma~\ref{forward_gal_factor}, $f_*$ induces a group isomorphism $f_\sharp: G_{X, x_0} \to G_{X, x_1}$ between the absolute Galois groups. That is, for $g = (\varphi_i)_{i \in \mathcal{I}(X, x_0)} \in G_{X, x_0}$, if we write $f_\sharp(g) = (\psi_i)_{i \in \mathcal{I}(X, x_1)}$, then we have $\psi_i=\varphi_{f_*^{-1}(i)}$.
\end{definition}

The following lemma follows directly from this definition.

\begin{lemma}\label{star_functor}
    For $f, f' \in \aut(X)$, we have $(f \circ f')_* = f_* \circ f'_*$ and $(f \circ f')_\sharp = f_\sharp \circ f'_\sharp$.
Moreover, we also have $(f^{-1})_* = (f_*)^{-1}$ and $(f^{-1})_\sharp = (f_\sharp)^{-1}$.
\end{lemma}

\begin{remark}
    By Lemma~\ref{star_functor}, notation such as $f^{-1}_*$ causes no confusion. 
\end{remark}

\begin{definition}\label{action_lim_fiber_def}
    Let $(X,x_0)$ be a pointed irreducible SFT.
    Fix a representative $(Y_i, y_{i,0}, \vp_i)$ for each $i \in \mathcal{I}(X, x_0)$.
    Let $F_i = \vp_i^{-1}(x_0)$ be the fiber, and consider the projective limit $F(x_0) = \varprojlim F_i$ of the projective system $(\{F_i\}_{i \in \mathcal{I}(X, x_0)}, \{\alpha_{ij}: F_j \to F_i\}_{i \le j})$.
    Then, for any $g = (\varphi_i)_{i \in \mathcal{I}(X, x_0)} \in G_{X, x_0}$ and $\xi = (y_i)_{i \in \mathcal{I}(X, x_0)} \in F(x_0)$, we define an action of $G_{X, x_0}$ on $F(x_0)$ by
    \[
        g\cdot\xi=(\varphi_i(y_i))_{i \in \mathcal{I}(X,x_0)}.
    \]
\end{definition}

\begin{lemma}\label{action_lim_fiber_welldef_free_transitive}
    In the above setting, this action is well-defined, free, and transitive.
\end{lemma}

\begin{proof}
    First, we show that the action is well-defined. Fix $g=(\varphi_i)_{i \in \mathcal{I}(X,x_0)} \in G_{X,x_0}$ and $\xi=(y_i)_{i \in \mathcal{I}(X,x_0)}\in F(x_0)$. It suffices to show that $g\cdot\xi\in F(x_0)$. For each $i$, since $\varphi_i\in G_i=\gal(Y_i/X, \vp_i)$ and $y_i \in F_i$, we have $\vp_i(\varphi_i(y_i))=\vp_i(y_i)=x_0$, which implies $\varphi_i(y_i) \in \vp_i^{-1}(x_0)=F_i$. A simple calculation shows that for $i\le j$, we have $\alpha_{ij}(\varphi_j(y_j))=\varphi_i(y_i)$, which means that $g\cdot\xi=(\varphi_i(y_i))_{i\in \mathcal{I}(X,x_0)}\in F(x_0)$.\\
    \indent Next, we show that the action is free. Fix $g=(\varphi_i)_{i \in \mathcal{I}(X,x_0)} \in G_{X,x_0}$. Suppose there exists $\xi = (y_i) \in F(x_0)$ such that $g \cdot \xi = \xi$.
    Then, for any $i \in \mathcal{I}(X,x_0)$, we have $\varphi_i(y_i) = y_i$. By Corollary~\ref{freeaction}, we have $\varphi_i = \id_{Y_i}$. Thus, $g$ is the identity element of $G_{X,x_0}$. Therefore, we have shown that the action is free.\\
    \indent Finally, we show the transitivity of the action. Fix $\xi=(y_i)_{i\in \mathcal{I}(X,x_0)}$ and $\eta=(z_i)_{i\in\mathcal{I}(X,x_0)}$. For each $i\in \mathcal{I}(X,x_0)$, since $y_i, z_i\in F_i$, the transitivity of the Galois group action (Proposition~\ref{transitiveaction}) ensures that there exists $\varphi_i \in G_i$ such that $\varphi_i(y_i)=z_i$. For any $i \le j$, we have 
    \[
    p_{ij}(\varphi_j)(y_i) = \alpha_{ij}(\varphi_j(y_j)) = \alpha_{ij}(z_j) = z_i=\varphi_i(y_i).
    \]
    Since the action of $G_i$ on $F_i$ is free (Corollary~\ref{freeaction}), $p_{ij}(\varphi_j) = \varphi_i$ holds. Set $g=(\varphi_i)_{i\in \mathcal{I}(X,x_0)}\in G_{X,x_0}$, then we have $g\cdot\xi=\eta$.
\end{proof}

So far, we have shown the well-definedness of the absolute Galois group and its independence from the base point. Here, we formally name the maps that appeared in those proofs.

\begin{definition}\label{hogehoge}
    Let $X$ be an irreducible SFT and let $x_0 \in X$. Let $f \in \aut(X)$ and set $x_1 = f(x_0)$. For each $i \in \mathcal{I}(X, x_0)$, fix a representative $(Y_i, y_{i,0}, \vp_i)$, and set $F_i = \vp_i^{-1}(x_1)$. The pair $(\{F_i\}_{i \in \mathcal{I}(X, x_0)}, \{\alpha_{ij}|_{F_j}: F_j \to F_i\})$ forms a projective system, and we denote its projective limit by $F(x_1) = \varprojlim F_i$. Fix $\xi(f) = \xi = (y'_{i,0})_{i \in \mathcal{I}(X, x_0)} \in F(x_1)$. Then, we define the order isomorphism $\Psi_\xi: \mathcal{I}(X, x_0) \to \mathcal{I}(X, x_1)$ by
    \[
    \Psi_\xi([(Y_i, y_{i,0}, \vp_i)]) = [(Y_i, y'_{i,0}, \vp_i)].
    \]
    This is the same as the map $\Psi$ constructed in the proof of Lemma~\ref{base_point_change}.
    Furthermore, let us define $\gamma_\xi: G_{X, x_1} \to G_{X, x_0}$ as follows. For $g = (\psi_i)_{i \in \mathcal{I}(X, x_1)} \in G_{X, x_1}$, we define $\gamma_{\xi}(g) = (\varphi_i)_{i \in \mathcal{I}(X, x_0)} \in G_{X, x_0}$ such that
    \[
    \varphi_i=\psi_{\Psi_\xi(i)}.
    \]
    In other words, if $\gamma_{\xi}(g)_i$ denotes the $i$-th coordinate of $\gamma_{\xi}(g)$, then $\gamma_{\xi}(g)_i = \psi_{\Psi_\xi(i)}$. Since $\gamma_\xi$ simply reindexes the coordinates via $\Psi_\xi$, it is a group isomorphism. The composition of $f_\sharp: G_{X, x_0} \to G_{X, x_1}$ and $\gamma_\xi: G_{X, x_1} \to G_{X, x_0}$ yields an automorphism on $G_{X, x_0}$. We denote this by $\beta_{f, \xi} = \gamma_\xi \circ f_\sharp: G_{X, x_0} \to G_{X, x_0}$.
\end{definition}

This $\beta_{f, \xi}$ represents the action of the automorphism $f$ on the absolute Galois group. This representation depends on the choice of $\xi$, but we will show below that it is unique up to inner automorphism.
We denote the inner automorphism group of a group $G$ by $\inn(G)$, and the outer automorphism group by $\out(G)=\aut(G)/\inn(G)$.

\begin{lemma}\label{out_welldef}
    In the above setting, the equivalence class $[\beta_{f, \xi}]$ of the automorphism $\beta_{f, \xi} \in \aut(G_{X, x_0})$ in the outer automorphism group $\out(G_{X, x_0})=\aut(G_{X, x_0})/\inn(G_{X, x_0})$ is independent of the choice of $\xi\in F(x_1)$.
\end{lemma}

\begin{proof}
    Fix $\xi=(y'_{i,0})_{i \in \mathcal{I}(X, x_0)}, \eta=(z'_{i,0})_{i \in \mathcal{I}(X, x_0)} \in F(x_1)$. By Lemma~\ref{action_lim_fiber_welldef_free_transitive}, since the action of $G_{X, x_0}$ on $F(x_1)$ is free and transitive, there exists a unique $h=(\chi_i) \in G_{X, x_0}$ such that $\eta=h\cdot\xi$. Fix $i\in\mathcal{I}(X, x_0)$. Since $\eta=h\cdot\xi$ implies $\chi_i(y'_{i,0})=z'_{i,0}$, the map $\chi_i:(Y_i, y'_{i,0})\to(Y_i, z'_{i,0})$ is a pointed conjugacy. Furthermore, since $\chi_i\in\gal(Y_i/X,\vp_i)$, we have $\vp_i\circ\chi_i=\vp_i$, which yields
    \[
    [(Y_i, y'_{i,0}, \vp_i)]=[(Y_i, z'_{i,0}, \vp_i)].
    \]
    That is, $\Psi_\xi(i)=\Psi_\eta(i)$.
    Let $j=\Psi_\xi(i)=\Psi_\eta(i)\in\mathcal{I}(X, x_1)$ and fix a representative of $j$, denoting it by $(Y_j, y_{j,0}, \vp_j)$.
    Then, there exist base-point-preserving conjugacies $\theta_{\xi,i}:(Y_i, y'_{i,0})\to(Y_j, y_{j,0})$ and $\theta_{\eta,i}:(Y_i, z'_{i,0})\to(Y_j, y_{j,0})$.
    The map $\theta_{\eta,i}^{-1}\circ\theta_{\xi,i}:(Y_i, y'_{i,0})\to(Y_i, z'_{i,0})$ is a base-point-preserving conjugacy, and by the freeness of the action (or the uniqueness of the lift), we have $\theta_{\eta,i}^{-1}\circ\theta_{\xi,i}=\chi_i$.
    Fix $g=(\psi_k) \in G_{X, x_1}$. By the definition of $\gamma_\xi$, we have $\gamma_\xi(g)_i=\theta_{\xi,i}^{-1}\circ\psi_j \circ \theta_{\xi,i}$. Similarly, we also have $\gamma_\eta(g)_i=\theta_{\eta,i}^{-1} \circ \psi_j \circ \theta_{\eta,i}$. Then, we obtain
    \begin{align*}
        \gamma_\eta(g)_i&=\theta_{\eta,i}^{-1} \circ \psi_j \circ \theta_{\eta,i}\\
        &=\theta_{\eta,i}^{-1} \circ \theta_{\xi,i}\circ\gamma_\xi(g)_i \circ\theta_{\xi,i}^{-1}\circ \theta_{\eta,i}\\
        &=\chi_i\circ\gamma_\xi(g)_i\circ\chi_i^{-1}. 
    \end{align*}
    Since $i$ was chosen arbitrarily, we have $\gamma_\eta(g)=h\cdot\gamma_\xi(g)\cdot h^{-1}$. Thus, for any $k\in G_{X, x_0}$, we have
    \begin{align*}
        \beta_{f, \eta}(k)&=\gamma_\eta(f_\sharp(k))\\
        &=h\circ\gamma_\xi(f_\sharp(k))\circ h^{-1}\\
        &=h\circ\beta_{f, \xi}(k)\circ h^{-1}.
    \end{align*}
    Therefore, we conclude that
    \[
    \beta_{f, \eta}=\beta_{f, \xi}\pmod{\inn(G_{X,x_0})},
    \]
    which completes the proof.
\end{proof}

\begin{definition}
    Let $X$ be an irreducible SFT and $x_0 \in X$. Let $f \in \aut(X)$ and set $x_1 = f(x_0)$. For each $i \in \mathcal{I}(X, x_0)$, fix a representative $(Y_i, y_{i,0}, \vp_i)$ and set $F_i = \vp_i^{-1}(x_1)$. Let $F(x_1) = \varprojlim F_i$. Fix $\xi \in F(x_1)$. Then, we define the map $\rho_X = \rho: \aut(X) \to \out(G_{X, x_0})$ by
    \[
    \rho_X(f)=[\beta_{f,\xi}].
    \]
    By Lemma~\ref{out_welldef}, this map is well-defined.
\end{definition}

\begin{theorem}
    Let $X$ be an irreducible SFT and let $x_0 \in X$. The map $\rho: \aut(X) \to \out(G_{X,x_0})$ is a group homomorphism.
\end{theorem}

\begin{proof}
    Fix $f_1, f_2 \in \aut(X)$. Set $x_1=f_1(x_0)$, $x_2=f_2(x_0)$, and $x_{12}=f_1(x_2)$. Fix $\xi_1 \in F(x_1)=\varprojlim \vp_i^{-1}(x_1)$ and $\xi_2 \in F(x_2)=\varprojlim \vp_i^{-1}(x_2)$. Then, we have $\beta_{f_1, \xi_1}=\gamma_{\xi_1} \circ (f_1)_\sharp$ and $\beta_{f_2, \xi_2}=\gamma_{\xi_2} \circ (f_2)_\sharp$. Fix $i=[(Y_i, y_{i,0},\vp_i)]\in\mathcal{I}(X, x_0)$. A representative of $\Psi_{\xi_1}(i)\in \mathcal{I}(X, x_1)$ is $(Y_i, (\xi_1)_i, \vp_i)$. That is,
    \[
    \Psi_{\xi_1}(i)=[(Y_i, (\xi_1)_i, \vp_i)].
    \]
    Set $k=(f_1)_*^{-1}(\Psi_{\xi_1}(i)) \in \mathcal{I}(X, x_0)$. Then, $(Y_i, (\xi_1)_i, f_1^{-1} \circ \vp_i)$ is a representative of $k$. That is,
    \[
    k=[(Y_i, (\xi_1)_i, f_1^{-1} \circ \vp_i)].
    \]
    Now, considering the $k$-th coordinate of $\xi_2$, we have
    \[
    (\xi_2)_k\in(f_1^{-1}\circ\vp_i)^{-1}(x_2)=\vp_i^{-1}(f_1(x_2))=\vp_i^{-1}(x_{12}).
    \]
    Then, we define $\xi_3=((\xi_3)_i)_{i\in \mathcal{I}(X,x_0)}$ by
    \[
    (\xi_3)_i = (\xi_2)_k = (\xi_2)_{(f_1)_*^{-1}(\Psi_{\xi_1}(i))}.
    \]
    \indent Let us now show that $\xi_3\in F(x_{12})$. Suppose that $l,m\in\mathcal{I}(X,x_0)$ satisfy $l\le m$. Set $l'=\Psi_{\xi_1}(l)$ and $m'=\Psi_{\xi_1}(m)$. Since $\Psi_{\xi_1}$ is order-preserving, we have $l'\le m'$. Since $\xi_1 \in F(x_1)$, we have $\alpha_{ij}((\xi_1)_j)=(\xi_1)_i$. This means $\alpha_{i'j'}=\alpha_{ij}$. Furthermore, set $l''=(f_1)_*^{-1}(l')$ and $m''=(f_1)_*^{-1}(m')$. Since $f_*$ is an order isomorphism, we have $l''\le m''$. Moreover, since
    \[
    (f_1^{-1} \circ \vp_i) \circ \alpha_{ij} = f_1^{-1} \circ (\vp_i \circ \alpha_{ij}) = f_1^{-1} \circ \vp_j,
    \]
    we have $\alpha_{l''m''}=\alpha_{lm}$. Since $\xi_2 \in F(x_2)$, we have $\alpha_{l''m''}((\xi_2)_{m''})=(\xi_2)_{l''}$. By the definition of $\xi_3$, we also have $(\xi_3)_{l}=(\xi_2)_{l''}$ and $(\xi_3)_{m}=(\xi_2)_{m''}$. Therefore, we obtain
    \[
    \alpha_{lm}((\xi_3)_m)=\alpha_{l''m''}((\xi_2)_{m''})=(\xi_2)_{l''}=(\xi_3)_{l},
    \]
    which implies $\xi_3\in F(x_{12})$.\\
    \indent Fix $g \in G_{X, x_0}$. Set $h=\beta_{f_2, \xi_2}(g)$. Then, we have
    \begin{align*}
        (f_2)_*^{-1}(\Psi_{\xi_2}(k))&=(f_2)_*^{-1}(\Psi_{\xi_2}([(Y_i, (\xi_1)_i, f_1^{-1} \circ \vp_i)]))\\
        &=(f_2)_*^{-1}([(Y_i, (\xi_2)_k, f_1^{-1} \circ \vp_i)])\\
        &=(f_2)_*^{-1}([(Y_i, (\xi_3)_i, f_1^{-1} \circ \vp_i)])\\
        &=[(Y_i, (\xi_3)_i, f_2^{-1} \circ f_1^{-1} \circ \vp_i)]\\
        &=[(Y_i, (\xi_3)_i, (f_1 \circ f_2)^{-1} \circ \vp_i)]\\
        &= (f_1 \circ f_2)_*^{-1}(\Psi_{\xi_3}(i)).
    \end{align*}
    Thus, it follows that
    \begin{align*}
        ((\beta_{f_1, \xi_1}\circ\beta_{f_2, \xi_2})(g))_i&=(\beta_{f_1, \xi_1}(h))_i\\
        &=h_{(f_1)_*^{-1}(\Psi_{\xi_1}(i))}\\
        &= h_k\\
        &= g_{(f_2)_*^{-1}(\Psi_{\xi_2}(k))}\\
        &= g_{(f_1 \circ f_2)_*^{-1}(\Psi_{\xi_3}(i))}\\
        &=(\beta_{f_1 \circ f_2, \xi_3}(g))_i.
    \end{align*}
    Since $i$ was arbitrary, we have $\beta_{f_1, \xi_1}\circ\beta_{f_2, \xi_2}=\beta_{f_1 \circ f_2, \xi_3}$. Consequently, we obtain
    \[
    \rho(f_1) \circ \rho(f_2) = [\beta_{f_1, \xi_1} \circ \beta_{f_2, \xi_2}] = [\beta_{f_1 \circ f_2, \xi_3}] = \rho(f_1 \circ f_2).
    \]
\end{proof}

It is a natural question to ask how much information $\rho$ preserves and how much it discards. The following proposition implies that $\rho$ forgets the information of the shift. In other words, $\rho$ can be viewed as eliminating the information of time evolution and extracting the information of spatial symmetry.

\begin{proposition}\label{ker_shift}
    Let $X$ be an irreducible SFT. Then, we have $\sigma_X \in \mathrm{Ker } \rho_X$. 
\end{proposition}

\begin{proof}
    Fix a base point $x_0$. Set $x_1=\sigma_X(x_0)$. Let $i \in \mathcal{I}(X, x_0)$ and fix a representative $(Y_i, y_{i,0}, \vp_i)$. Let $y'_{i,0}=\sigma_{Y_i}(y_{i,0})$ and set $\xi=(y'_{i,0})_{i \in \mathcal{I}(X, x_0)}$. It directly follows that $\xi\in F(x_1)$. Consider $h=\sigma_{Y_i}:(Y_i,y_{i,0}) \to (Y_i,\sigma_{Y_i}(y_{i,0}))$. The map $h$ is a pointed conjugacy and satisfies $\vp_i\circ h=\vp_i\circ\sigma_{Y_i}$, so we have
    \[
    [(Y_i, y_{i,0}, \vp_i \circ \sigma_{Y_i})]=[(Y_i, \sigma_{Y_i}(y_{i,0}), \vp_i)].
    \]
    Therefore, we obtain
    \begin{align*}
        (\sigma_X)_*([(Y_i, y_{i,0}, \vp_i)])&=[(Y_i, y_{i,0}, \sigma_X \circ \vp_i)] \\
        &=[(Y_i, y_{i,0}, \vp_i\circ \sigma_{Y_i})]\\
        &=[(Y_i, \sigma_{Y_i}(y_{i,0}), \vp_i)]\\
        &=[(Y_i, \xi_i, \vp_i)]\\
        &=\Psi_\xi([(Y_i, y_{i,0}, \vp_i)]),
    \end{align*}
    which implies $(\sigma_X)_*(i)=\Psi_\xi(i)$. Since $i$ is arbitrary, we have $(\sigma_X)_*=\Psi_\xi$. Thus, $(\sigma_X)_*^{-1}\circ \Psi_\xi=\id_{\mathcal{I}(X,x_0)}$.
    Therefore, for any $g\in G_{X,x_0}$ and $i\in \mathcal{I}(X,x_0)$, we have 
    \begin{align*}
        (\beta_{\sigma_X, \xi}(g))_i&=g_{(\sigma_X)_*^{-1}(\Psi_\xi(i))}=g_i,
    \end{align*}
    which yields $\beta_{\sigma_X, \xi}=\id_{G_{X,x_0}}$.
    From the above, we conclude that
    \[
    \rho(\sigma_X)=[\beta_{\sigma_X, \xi}]=[\id_{G_{X,x_0}}]=1_{\out(G_{X,x_0})},
    \]
    which proves that $\sigma_X\in\mathrm{Ker }\rho_X$.
\end{proof}

The following two corollaries immediately follow from Proposition~\ref{ker_shift}.

\begin{corollary}\label{cor_shiftker}
    Let $X$ be an irreducible SFT. Then, the subgroup $\langle\sigma_X\rangle$ generated by $\sigma_X$ is a subgroup of $\mathrm{Ker } \rho_X$.
\end{corollary}

\begin{corollary}
    Let $X$ be an irreducible SFT. If $\sigma_X \neq \id_X$, then the map $\rho: \aut(X) \to \out(G_X)$ is not injective. 
\end{corollary}

\begin{remark}
    It is known from \cite{ryan_shiftcenter} that the center of $\aut(X)$ is $\langle \sigma_X\rangle$. Therefore, by Corollary \ref{cor_shiftker}, all information about the center of the group is lost under $\rho_X$.
\end{remark}

\section{First Galois Cohomology and the Automorphism Group}\label{section_gal_cohom}
Let $p$ be a prime number and let $\F_p$ be the finite field of order $p$. In this section, we consider the first Galois cohomology with coefficients in $\F_p$, and construct a representation of the automorphism group on the cohomology. 

\subsection{Galois Cohomology and Representations}
Let $(X, x_0)$ be a pointed SFT, and $G_{X, x_0}$ be its absolute Galois group. Equip the finite field $\F_p$ with the discrete topology to view it as a topological group. Assume that the absolute Galois group $G_{X, x_0}$ acts trivially on $\F_p$. Then, by Lemma~\ref{cohom_homconti}, we have
\[
H^1(G_{X, x_0}, \mathbb{F}_p)=\mathrm{Hom}_{\mathrm{cont}}(G_{X, x_0}, \mathbb{F}_p).
\]
The operations in $\F_p$ induce an $\F_p$-vector space structure on $H^1(G_{X, x_0}, \mathbb{F}_p)$. We call $H^1(G_{X, x_0}, \mathbb{F}_p)$ the first Galois cohomology. \\
\indent The following lemma allows us to simply denote the first Galois cohomology by $H^1(G_X, \F_p)$.

\begin{lemma}\label{cohom_canonical}
    Let $X$ be an irreducible SFT, let $x_0, x_1 \in X$. Then, the map $\gamma_\xi: G_{X,x_1} \to G_{X,x_0}$ defined in Definition~\ref{hogehoge} induces an isomorphism between $H^1(G_{X,x_0}, \mathbb{F}_p)$ and $H^1(G_{X,x_1}, \mathbb{F}_p)$. Furthermore, this isomorphism is independent of the choice of $\xi$.
\end{lemma}

\begin{proof}
    Fix a representative for each element of $\mathcal{I}(X, x_0)$, let $\xi \in F(x_1) = \varprojlim \vp_i^{-1}(x_1)$, and consider the map $\gamma_\xi: G_{X, x_1} \to G_{X, x_0}$ defined in Definition~\ref{hogehoge}.
    This induces an isomorphism $\gamma_\xi^*: H^1(G_{X, x_0}, \mathbb{F}_p) \to H^1(G_{X, x_1}, \mathbb{F}_p)$ defined by $\gamma_\xi^*(\chi) = \chi \circ \gamma_\xi$.
    Let $\eta \in F(x_1)$. Then, from the proof of Lemma~\ref{out_welldef}, there exists $h \in G_{X, x_0}$ such that $\gamma_\eta(g) = h \circ \gamma_\xi(g) \circ h^{-1}$ for any $g \in G_{X, x_1}$. For any $\chi \in H^1(G_{X, x_0}, \mathbb{F}_p)$ and any $g \in G_{X, x_1}$, we have
    \begin{align*}
    (\gamma_\eta^*(\chi))(g) &= \chi(\gamma_\eta(g)) \\
    &= \chi(h \cdot \gamma_\xi(g) \cdot h^{-1}) \\
    &= \chi(h) + \chi(\gamma_\xi(g)) - \chi(h) \\
    &= \chi(\gamma_\xi(g)) \\
    &= (\gamma_\xi^*(\chi))(g),
    \end{align*}
    which implies $\gamma_\eta^* = \gamma_\xi^*$. 
\end{proof}

\begin{remark}
    The fact that the isomorphism induced by $\gamma_\xi$ is canonically determined is essentially due to the commutativity of $\F_p$ and the triviality of the action. Thus, it is expected that this proposition holds for coefficients other than $\F_p$ as well.
\end{remark}

As the absolute Galois group is a conjugacy invariant, we have the following.

\begin{theorem}
    The first Galois cohomology is a conjugacy invariant. 
\end{theorem}

We define a representation of the automorphism group on the Galois cohomology.

\begin{definition}\label{def_phi_p}
    Let $X$ be an irreducible SFT and $x_0 \in X$. Let $f \in \aut(X)$ and set $x_1=f(x_0)$. For each $i \in \mathcal{I}(X, x_0)$, fix a representative $(Y_i, y_{i,0}, \vp_i)$ and set $F_i=\vp_i^{-1}(x_1)$. Let $F(x_1)=\varprojlim F_i$. Let $\beta_{f,\xi}$ be a representative of $\rho_X(f)$, where $\xi\in F(x_1)$. Define a linear map $\Upsilon_p(f): H^1(G_{X,x_0}, \mathbb{F}_p) \to H^1(G_{X,x_0}, \mathbb{F}_p)$ by
    \[
     (\Upsilon_p(f))(\chi)=\chi\circ \beta_{f,\xi}^{-1}.
    \]
    Since $\Upsilon_p(f)$ has the inverse map given by $\chi \mapsto \chi \circ \beta_{f,\xi}$, it is an automorphism of $H^1(G_{X,x_0}, \mathbb{F}_p)$. This yields a map $\Upsilon_p:\aut(X)\to \mathrm{GL}(H^1(G_{X,x_0}, \mathbb{F}_p))$. 
    
    We verify the well-definedness of this construction in the following lemma.
\end{definition}

\begin{lemma}\label{phi_p_welldef}
    In Definition~\ref{def_phi_p}, the linear isomorphism $\Upsilon_p(f)$ is independent of the choice of the representative of $\rho(f)$.
\end{lemma}

\begin{proof}
    Let $\beta_f, \beta'_f \in \rho(f)$. Then, there exists $h\in G_{X,x_0}$ such that $\beta'_f(g)=h\cdot \beta_f(g)\cdot h^{-1}$ for any $g \in G_{X,x_0}$. For any $g\in G_{X,x_0}$ and $\chi\in H^1(G_{X,x_0},\F_p)$, we have
    \begin{align*}
    (\chi \circ (\beta'_f)^{-1})(g) 
    &= \chi(h^{-1} \cdot \beta_f^{-1}(g) \cdot h) \\
    &= \chi(h^{-1}) + \chi(\beta_f^{-1}(g)) + \chi(h) \\
    &= -\chi(h) + \chi(\beta_f^{-1}(g)) + \chi(h) \\
    &= (\chi \circ \beta_f^{-1})(g),
    \end{align*}
    which implies $\chi \circ (\beta'_f)^{-1}=\chi \circ \beta_f^{-1}$.
\end{proof}

The fact that $\rho$ is a group homomorphism immediately implies the following.

\begin{theorem}\label{gal_rep_homo}
    The map $\Upsilon_p:\aut(X)\to \mathrm{GL}(H^1(G_{X,x_0}, \mathbb{F}_p))$ is a group homomorphism.
\end{theorem}

\begin{remark}
    The representation of the automorphism group on the absolute Galois group constructed in the previous section ignored the ambiguity of inner automorphisms. Therefore, $\Upsilon_p$ factors as follows. If we define $\tilde{\Upsilon}_p: \mathrm{Out}(G_{X,x_0}) \to \mathrm{GL}(H^1(G_{X,x_0}, \mathbb{F}_p))$ by
    \[
    [\beta] \mapsto \left( \chi \mapsto \chi \circ \beta^{-1} \right),
    \]
    then we have
    \[
    \Upsilon_p=\tilde{\Upsilon}_p \circ \rho_X.
    \]
    Lemma~\ref{phi_p_welldef} also shows the well-definedness of $\tilde{\Upsilon}_p$.
\end{remark}

\subsection{Frobenius Elements and Periodic Orbit}
    It is well known that analogies exist between dynamical systems and number theory. In the context of SFTs, pioneering work has been done by Parry, Pollicott, and others. In this subsection, following the spirit of their work, we define the Frobenius element. \\
    \indent Let $m \ge 1$ be an integer. For an irreducible SFT $X$, we denote the set of all periodic points with least period $m$ by $P_m(X)$, and denote the set of all periodic orbits of length $m$ by $\mathcal{P}_m(X)$. Depending on the context, we may implicitly assume that $\mathcal{P}_m(X) \ne \emptyset$.
    Let $x_0 \in X$ and $\chi\in H^1(G_X,\F_p)\setminus\{0\}=\mathrm{Hom}_{\mathrm{cont}}(G_X, \F_p)\setminus\{0\}$. Since $\mathrm{Ker}\ \chi$ is an open subgroup, by Lemma~\ref{open_normal_subgroup}, let $(Y, y_0, \vp)$ be a representative of the pointed Galois factor over $(X, x_0)$ corresponding to $\mathrm{Ker}\ \chi$. That is, we have
    \[
    \gal(Y/X,\vp)\cong G_X/\mathrm{Ker}\ \chi.
    \]
    Here, $Y$ depends on $\chi$, but to avoid cumbersome notation, we will not explicitly indicate this dependence in what follows unless necessary.
    Although $\mathrm{Im}\ \chi$ is a subgroup of $\F_p$, the group $\F_p$ has no non-trivial proper subgroups. Since $\chi \ne 0$, it follows that $\mathrm{Im}\ \chi = \F_p$. Thus, the first isomorphism theorem yields
    \[
    \gal(Y/X,\vp)\cong G_X/\mathrm{Ker}\ \chi\cong \F_p.
    \]
    Let $x \in P_m(X)$ and set $F_x = \vp^{-1}(x)$. Fix $y \in F_x$. Then we have
    \[
    \vp(\sigma_Y^m(y))=\sigma_X^m(\vp(y))=\sigma_X^m(x)=x,
    \]
    which implies $\sigma_Y^m(F_x) = F_x$. Since the Galois group acts on the fibers, the group $\F_p$ also acts on the fibers. We denote this action additively by $+$. Because this action is transitive and free (Proposition~\ref{transitiveaction}, Corollary~\ref{freeaction}), there exists a unique $c_{x,y,\chi} \in \F_p$ such that
    \[
    \sigma_Y^m(y)=y+c_{x,y,\chi}.
    \]

\begin{lemma}\label{frob_welldef}
    In the above construction, $c_{x,y,\chi}$ is independent of the choice of the representative $Y$ and the element $y \in F_x$ in the fiber. Furthermore, $c_{x,y,\chi}$ depends only on the periodic orbit $O_{\sigma_X}(x)$.
\end{lemma}

\begin{proof}
    First, we show the independence of the choice of the element in the fiber. Let $y' \in F_x$. Since the action of $\F_p$ on the fiber is transitive, there exists an $a \in \F_p$ such that $y' = y + a$. Then, by the commutativity of the shift and the Galois group action, we have
    \[
    \sigma_Y^m(y')=\sigma_Y^m(y)+a=(y+c_{x,y,\chi})+a=y'+c_{x,y,\chi}.
    \]
    Next, we show the independence of the choice of the representative. Suppose that $(Y', y'_0, \vp')$ and $(Y, y_0, \vp)$ are conjugate, and let $\varphi: Y \to Y'$ be the conjugacy map. Noting that $\varphi$ also commutes with the action of the Galois group, we obtain
    \[
    \sigma_{Y'}^m(\varphi(y))=\varphi(\sigma_Y^m(y))=
    \varphi(y+c_{x,y,\chi})=\varphi(y)+c_{x,y,\chi}. 
    \]
    Finally, we show that it depends only on the orbit. Let $x' \in O_{\sigma_X}(x)$. Then, there exists an $n \in \Z$ such that $x' = \sigma_X^n(x)$. Letting $y' = \sigma_Y^n(y)$, we have
    \[
    \sigma_Y^m(y')=\sigma_Y^m(\sigma_Y^n(y))=
    \sigma_Y^n(\sigma_Y^m(y))=\sigma_Y^n(y+c_{x,y,\chi})
    =y'+c_{x,y,\chi}.
    \]
\end{proof}

\begin{definition}
    Let $\Gamma \in \mathcal{P}_m(X)$ and $x \in \Gamma$. By Lemma~\ref{frob_welldef}, we may write $c_{x,y,\chi}$ as $c_{\Gamma,\chi}$. This is called the Frobenius element of $\chi$ at $\Gamma$. By setting $\mathrm{Frob}_\Gamma(\chi) = c_{\Gamma,\chi}$, we define a map $\mathrm{Frob}_\Gamma: H^1(G_X, \F_p) \to \F_p$. Here, we set $\mathrm{Frob}_\Gamma(0) = 0$.
\end{definition}

In what follows, we prove the linearity and equivariance of $\mathrm{Frob}_\Gamma$. To this end, we first define the absolute Frobenius element.

\begin{definition}\label{def_absolute_frobenius}
    Let $(X, x_0)$ be a pointed irreducible SFT and $m \ge 1$ be an integer. Fix $\Gamma \in \mathcal{P}_m(X)$ and $x \in \Gamma$. 
    For each $i \in \mathcal{I}(X, x_0)$, we fix a representative $(Y_i, y_{i,0}, \vp_i)$. Set $F_{x,i} = \vp_i^{-1}(x)$ and $F(x) = \varprojlim F_{x,i}$. For each $i \in \mathcal{I}(X, x_0)$, the Galois group $\gal(Y_i/X, \vp_i)$ acts transitively and freely on $F_{x,i}$ (Corollary~\ref{freeaction}, Proposition~\ref{transitiveaction}). Therefore, by the same construction as in Definition~\ref{action_lim_fiber_def}, the absolute Galois group $G_X$ acts naturally on $F(x)$, and by the same argument as in Lemma~\ref{action_lim_fiber_welldef_free_transitive}, it can be seen that this action is free and transitive.
    Fix $\tilde{y} = (y_i)_{i \in \mathcal{I}(X,x_0)} \in F(x)$. Then, for each $i \in \mathcal{I}(X,x_0)$, we obtain
    \[
    \vp_i(\sigma_{Y_i}^m(y_i))=\sigma_X^m(\vp_i(y_i))=\sigma_X^m(x)=x.
    \]
    Thus, we have $\sigma_{Y_i}^m(y_i) \in F_{x,i}$. Let $\sigma(\tilde{y}) := (\sigma_{Y_i}(y_i))_{i \in \mathcal{I}(X,x_0)}$ denote the sequence obtained by applying the shift map coordinate-wise. Since the shift map commutes with the transition maps $\alpha_{ij}$ between intermediate factors, it follows that $\sigma^m(\tilde{y}) \in F(x)$. Because the action of $G_X$ on $F(x)$ is free and transitive, there exists a unique group element $g_{x, \tilde{y}} \in G_X$ satisfying
    \begin{align*}
     g_{x, \tilde{y}} \cdot \tilde{y} = \sigma^m(\tilde{y}). 
    \end{align*}
    We call this $g_{x, \tilde{y}}$ the absolute Frobenius element with respect to the periodic point $x$ and the fiber element $\tilde{y}$.
\end{definition}

\begin{lemma}\label{absolute_frobenius_welldef}
    In the above setting, the following hold:
    \begin{enumerate}
    \item For any $x' \in \Gamma$ and $\tilde{z} \in F(x')$, the elements $g_{x', \tilde{z}}$ and $g_{x, \tilde{y}}$ are conjugate in $G_X$.
    \item Suppose that $g' \in G_X$ is conjugate to $g_{x, \tilde{y}}$. Then, there exists a $\tilde{z} \in F(x)$ such that $g' = g_{x,\tilde{z}}$.
    \end{enumerate}
\end{lemma}

\begin{proof}
    Fix $x' \in \Gamma$ and $\tilde{z} \in F(x')$. Then, there exists an integer $k$ such that $x'=\sigma_X^k(x)$. A simple calculation shows that $\sigma^k(\tilde{y})\in F(x')$. By the same argument as in Lemma~\ref{action_lim_fiber_welldef_free_transitive}, the action of $G_X$ on $F(x')$ is transitive. Thus, there exists an $h\in G_X$ such that $h\cdot\sigma^k(\tilde{y})=\tilde{z}$. Then we have
    \begin{align*}
        (g_{x',\tilde{z}}h)\cdot(\sigma^k(\tilde{y}))&=g_{x',\tilde{z}}\cdot (h\cdot\sigma^k(\tilde{y}))\\
        &=\sigma^m(h\cdot\sigma^k(\tilde{y}))\\
        &=h\cdot(\sigma^m(\sigma^k(\tilde{y})))\\
        &=h\cdot(\sigma^k(\sigma^m(\tilde{y})))\\
        &=h\cdot(\sigma^k(g_{x,\tilde{y}}\cdot\tilde{y}))\\
        &=(hg_{x,\tilde{y}})\cdot(\sigma^k(\tilde{y})).
    \end{align*}
    Since the action of $G_X$ on $F(x')$ is free by the same argument as in Lemma \ref{action_lim_fiber_welldef_free_transitive}, we have $g_{x',\tilde{z}}h=hg_{x,\tilde{y}}$. Therefore, $g_{x',\tilde{z}}$ and $g_{x,\tilde{y}}$ are conjugate.\\
    \indent Next, we show the second claim.
    Suppose that $g'$ is conjugate to $g_{x,\tilde{y}}$. Then, there exists an $h \in G_X$ such that $g' = h g_{x, \tilde{y}} h^{-1}$. Let $\tilde{z} = h \cdot \tilde{y}$. By the well-definedness of the action, we have $\tilde{z} \in F(x)$. Then we have
    \begin{align*}
        g'\cdot(\tilde{z})&=g'\cdot(h\cdot \tilde{y})\\
        &=(hg_{x,\tilde{y}}h^{-1}h)\cdot \tilde{y}\\
        &=(hg_{x,\tilde{y}})\cdot\tilde{y}\\
        &=h\cdot(g_{x,\tilde{y}}\cdot\tilde{y})\\
        &=h\cdot(\sigma^m(\tilde{y}))\\
        &=\sigma^m(h\cdot \tilde{y})\\
        &=\sigma^m(\tilde{z})\\
        &=g_{x,\tilde{z}}\cdot(\tilde{z}).
    \end{align*}
    Since the action is free, we obtain $g' = g_{x,\tilde{z}}$. 
\end{proof}

\begin{definition}
    By Lemma~\ref{absolute_frobenius_welldef}, the conjugacy class to which $g_{x, \tilde{y}}$ belongs depends only on the orbit $\Gamma\in\mathcal{P}_m$. Therefore, we denote this conjugacy class by $\mathrm{Frob}_\Gamma^{\mathrm{abs}}\subseteq G_X$ and call it the absolute Frobenius conjugacy class of the orbit $\Gamma$.
\end{definition}

\begin{lemma}\label{teketeke}
    Let $\Gamma \in \mathcal{P}_m(X)$ and let $\chi \in H^1(G_X, \F_p)$. Then the following hold:
    \begin{enumerate}
        \item For any $g, g' \in \mathrm{Frob}_\Gamma^{\mathrm{abs}}$, we have $\chi(g) = \chi(g')$.
        \item For any $g \in \mathrm{Frob}_\Gamma^{\mathrm{abs}}$, we have $\mathrm{Frob}_\Gamma(\chi) = \chi(g)$.
    \end{enumerate}
\end{lemma}

\begin{proof}
    The first claim is clear. We show the second claim. Let $N=\mathrm{Ker}\ \chi$, let $(Y,y_0,\vp)$ be a representative of the pointed intermediate factor corresponding to $N$ in Lemma~\ref{open_normal_subgroup}, and set $G=\gal(Y/X, \vp)$. By the first isomorphism theorem, we have $G\cong G_X/N\cong \F_p$. Let $\bar{\chi}:G\to \F_p$ be the isomorphism giving this relation. Letting $\pi:G_X\to G$ be the natural projection, we have $\chi=\bar{\chi}\circ\pi$.
    Fix $x\in\Gamma$ and $\tilde{y} \in F(x)$, and let $g_{x, \tilde{y}} \in \mathrm{Frob}_\Gamma^\mathrm{abs}$. Set $i=[(Y,y_0,\vp)]$, and let $\pi_i:F(x)\to \vp^{-1}(x)$ be the projection onto the $i$-th coordinate. set $y=\pi_i(\tilde{y})$. Comparing the $i$-th coordinates of both sides of the defining equation of the absolute Frobenius element
    \[
    g_{x, \tilde{y}}\cdot \tilde{y}=\sigma^m(\tilde{y}),
    \]
    we obtain
    \[
    (\pi(g_{x, \tilde{y}}))(y)=\sigma^m_{Y}(y).
    \]
    Thus, we have
    \[
        y+\chi(g)=y+\bar{\chi}(\pi(g_{x, \tilde{y}}))
        =(\pi(g_{x, \tilde{y}}))(y)
        =\sigma^m_{Y}(y).
    \]
    Therefore, we obtain $\mathrm{Frob}_\Gamma(\chi)=\chi(g_{x, \tilde{y}})$.
\end{proof}

    In what follows, for $g \in \mathrm{Frob}_\Gamma^{\mathrm{abs}}$, we write $\chi(\mathrm{Frob}_{\Gamma}^{\mathrm{abs}}) = \chi(g)$. By Lemma~\ref{teketeke}, this is well-defined.

\begin{proposition}\label{frob_linear}
    The map $\mathrm{Frob}_\Gamma:H^1(G_X,\F_p)\to \F_p$ is an $\F_p$-linear map.
\end{proposition}

\begin{proof}
    Fix $\chi_1,\chi_2\in H^1(G_X,\F_p)$, $a\in \F_p$ and $g\in \mathrm{Frob}_\Gamma^{\mathrm{abs}}$. Then
    \[
    \mathrm{Frob}_\Gamma(a\chi_1+\chi_2)=(a\chi_1+\chi_2)(g)=a\chi_1(g)+\chi_2(g)=a\mathrm{Frob}_\Gamma(\chi_1)+\mathrm{Frob}_\Gamma(\chi_2).
    \]
    This completes the proof.
\end{proof}

Since $f \in \aut(X)$ maps $m$-periodic orbits to $m$-periodic orbits, $f$ can be regarded as a permutation on $\mathcal{P}_m(X)$. \\
\indent The following theorem shows the equivariance.

\begin{theorem}\label{frob_equiv}
    Let $f\in\aut(X)$, $\Gamma\in\mathcal{P}_m(X)$, and $\chi\in H^1(G_X,\F_p)$. Let $\beta_{f,\xi}\in \aut(G_X)$ be a representative of $\rho_X(f)\in\out(G_X)$. Then we have
    \[
    \beta_{f,\xi}(\mathrm{Frob}_\Gamma^{\mathrm{abs}})=\mathrm{Frob}_{f(\Gamma)}^{\mathrm{abs}}.
    \]
\end{theorem}

\begin{proof}
    Fix $x \in \Gamma$. set $x' = f(x) \in f(\Gamma)$. For each $i \in \mathcal{I}(X, x_0)$, fix a representative $(Y_i, y_{i,0}, \vp_i)$ and let $G_i = \gal(Y_i/X, \vp_i)$ be its Galois group. Set $F_{x,i} = \vp_i^{-1}(x)$ and $F(x) = \varprojlim F_{x,i}$ . Let $\tilde{y} = (y_i)_{i \in \mathcal{I}(X, x_0)} \in F(x)$. Let $g = g_{x,\tilde{y}} \in G_X$ be the absolute Frobenius element, and let $h = \beta_{f,\xi}(g) \in G_X$. Fix $i \in \mathcal{I}(X, x_0)$. From the definitions of $\gamma_\xi$ and $f_\sharp$ in Definition~\ref{hogehoge}, letting $k = f_*^{-1}(\Psi_\xi(i))$, we have $h_i = g_k$ and $k = [(Y_i, \xi_i, f^{-1} \circ \vp_i)]$. Thus, we have
    \[
    (f^{-1}\circ\vp_i)(y_k)=x.
    \]
    This implies that $y_k \in \vp_i^{-1}(x')$. Let $z_i = y_k \in \vp_i^{-1}(x')$ and $\tilde{z} = (z_i)_{i \in \mathcal{I}(X, x_0)}$. Then we have
    \[
    \sigma_{Y_i}^m(z_i)=\sigma_{Y_i}^m(y_k)=g_k(y_k)=h_i(z_i).
    \]
    Since $i$ was arbitrary, this equation means that $h = g_{x',\tilde{z}}$. That is,
    \[
    \beta_{f,\xi}(\mathrm{Frob}_\Gamma^{\mathrm{abs}}) = \mathrm{Frob}_{f(\Gamma)}^{\mathrm{abs}}
    \]
    holds. 
\end{proof}

\begin{corollary}\label{frob_equiv_cor}
    Let $f \in \aut(X)$, $\Gamma \in \mathcal{P}_m(X)$, and $\chi \in H^1(G_X, \F_p)$. Then we have
    \[
    \mathrm{Frob}_{f(\Gamma)}((\Upsilon_p(f))(\chi))=\mathrm{Frob}_\Gamma(\chi).
    \]
\end{corollary}

\begin{proof}
    By Theorem~\ref{frob_equiv} and Lemma~\ref{teketeke}, we have
    \begin{align*}
        \mathrm{Frob}_{f(\Gamma)}((\Upsilon_p(f))(\chi))&=((\Upsilon_p(f))(\chi))(\mathrm{Frob}_{f(\Gamma)}^{\mathrm{abs}})\\
        &=(\chi\circ\beta_{f,\xi}^{-1})(\beta_{f,\xi}(\mathrm{Frob}_\Gamma^{\mathrm{abs}}))\\
        &=\chi(\mathrm{Frob}_\Gamma^{\mathrm{abs}})\\
        &=\mathrm{Frob}_\Gamma(\chi).
    \end{align*}
\end{proof}

\subsection{Local-Global Evaluation map}
In general, the representation $\Upsilon_p$ is too complex for direct computation. A simple way to extract information is to take restrictions or quotients. Here, we focus only on extracting information about the Frobenius elements. For this purpose, we first introduce a local-global evaluation map.

\begin{definition}
    Let $m \ge 1$ be an integer. Let $V_m(X, p)$ be a finite-dimensional $\F_p$-vector space defined by
    \[
    V_m(X,p)=\bigoplus_{\Gamma \in \mathcal{P}_m(X)} \mathbb{F}_p.
    \]
    Then, we have $\dim(V_m(X,p))=|\mathcal{P}_m(X)|$.
    We define a map 
    \[\mathrm{Ev}_m: H^1(G_X, \mathbb{F}_p) \to V_m(X,p)
    \]
    by
    \[
    \mathrm{Ev}_m(\chi) = (\mathrm{Frob}_\Gamma(\chi))_{\Gamma \in \mathcal{P}_m(X)}.
    \]
    We call $\mathrm{Ev}_m$ the local-global evaluation map of degree $m$. By Proposition~\ref{frob_linear}, $\mathrm{Ev}_m$ is a linear map. 
\end{definition}

In what follows, we prove the surjectivity of $\mathrm{Ev}_m$. Let $X$ be an SFT over the alphabet $\AA$, and let $N \in \N$. Let $f: \BB_{2N+1}(X) \to \F_p$ be a map. We also denote the continuous map $X \to \F_p$ induced by $f$ by $f$. That is,
    \[
    f(x)=f(x_{[-N,N]})\in \F_p.
    \]
We define a map $\sigma_f:X \times \F_p\to X \times \F_p$ by
    \[
    \sigma_f(x, c) = (\sigma_X(x), c+f(x) ).
    \]

\begin{lemma}\label{skew_sft}
    Under the above setting, suppose furthermore that there exists a periodic point $z \in X$ such that
    \[
    \sum_{i=0}^{k-1}f(\sigma_X^i(z))\ne 0,
    \]
    where $k$ is the minimal period of $z$. Then, there exists an irreducible SFT conjugate to the dynamical system $(X \times \F_p, \sigma_f)$.
\end{lemma}

\begin{proof}
    By passing to a higher block presentation if necessary, we may assume that $X$ is a 1-step SFT. (For the concept of 1-step, see, for example, Lind and Marcus \cite{introduction}.) Let $\BB = \AA \times \F_p$. We define a subshift $Y$ of $\BB^{\Z}$ by
    \[
     Y=\{(x,c)\in \BB^{\Z}\mid c_{i+1}-c_i=f(\sigma^i_X(x))\text{ for any }i\in\Z\}.
    \]
    Since the elements of $Y$ are determined by the local rule $f: \BB_{2N+1} \to \F_p$ and the two-term recurrence relation $c_{i+1}-c_i=f(\sigma^i_X(x))$, the set of forbidden words $\FF$ can be constructed solely from blocks of length $2N+2$. Therefore, $Y$ is an SFT. If we define $h: X \times \F_p \to Y$ by
    \[
    h(x,c)=(x,c_0),
    \]
    then $h$ gives a conjugacy. \\
    \indent In what follows, we will prove the irreducibility of $Y$. However, by Lemma~\ref{irred_toptrans_equiv}, it suffices to prove the topological transitivity of the dynamical system $(X \times \mathbb{F}_p, \sigma_f)$.
    Let $\mathcal{U}$ and $\mathcal{V}$ be non-empty open sets in $X \times \mathbb{F}_p$. Since the cylinder sets form a basis for the topology of $X$, and the singleton sets form a basis for the topology of $\F_p$, we may assume that $\mathcal{U}$ and $\mathcal{V}$ are of the following form:
    Let $u, v \in \BB(X)$ with $u \ne \varepsilon$ and $v \ne \varepsilon$, set the cylinder sets $U = [u]_0$ and $V = [v]_0$, and let $a, b \in \F_p$ such that $\mathcal{U} = U \times \{a\}$ and $\mathcal{V} = V \times \{b\}$. By the assumption, we choose a periodic point $z \in X$ of minimal period $k$ such that
     \[
    S=\sum_{i=1}^{k-1}f(\sigma_X^i(z))\ne 0.
    \]
    Choose an integer $c \ge 1$ such that $ck > 2N$ and $c \not\equiv 0 \pmod{p}$. Then, we have $cS \not\equiv 0 \pmod{p}$. Let $w_z = z_{[0, ck-1]} \in \BB_{ck}(X)$ be the periodic block of $z$ of length $ck$. By the irreducibility of $X$, we can choose words $w_1, w_2$ such that $uw_1w_z, w_zw_2v \in \BB(X)$. Furthermore, we choose a word $l \in \BB_N(X)$ such that $luw_1w_z \in \BB(X)$, and a word $r \in \BB_N(X)$ such that $w_zw_2vr \in \BB(X)$. Let $q \in \{2, 3, \dots, p+1\}$. We define $w_q$ by
    \[
    w_q=luw_1w_z^qw_2vr.
    \]
    Since $X$ is a 1-step SFT, we have $w_q \in \BB(X)$.
    Set $m_q = |u| + |w_1| + q|w_z| + |w_2|$, and choose an $x^{(q)} \in X$ such that
    \[
    x^{(q)}_{[-N, m_q+|v|+N-1]}=w_q.
    \]
    By construction, we have $x^{(q)} \in U$ and $\sigma_X^{m_q}(x^{(q)}) \in V$. Let us set $C_1, C_2 \in \F_p$ as
    \begin{align*}
        C_1&=\sum_{j=0}^{|u|+|w_1|+ck-1}f(x^{(q)}_{[j-N,j+N]})\\
        C_2&=\sum_{j=|u|+|w_2|+(q-1)ck}^{m_q-1}f(x^{(q)}_{[j-N,j+N]}).
    \end{align*}
    By the definition of $w_q$, $C_1$ and $C_2$ are constants independent of $q$. Therefore, if we set $C = C_1 + C_2 - 2cS$, then $C$ is a constant independent of $q$.
    Since $p$ is a prime and $cS \not\equiv 0 \pmod{p}$, the congruence equation for $q$
    \[
    a+C+qcS\equiv b\pmod{p}
    \]
    has a solution in the range $2 \le q \le p+1$. In what follows, let $q$ be a solution to this congruence equation. Furthermore, since 
    \begin{align*}
        \sum_{j=0}^{m_q-1}f(\sigma_X^j(x^{(q)}))=&\sum_{j=0}^{m_q-1}f(x^{(q)}_{[j-N,j+N]})\\
        =&\sum_{j=0}^{|u|+|w_1|+ck-1}f(x^{(q)}_{[j-N,j+N]})\\&\qquad+
        \sum_{j=|u|+|w_1|+ck}^{|u|+|w_1|+(q-1)ck-1}f(x^{(q)}_{[j-N,j+N]})\\&\qquad+
        \sum_{j=|u|+|w_1|+(q-1)ck}^{m_q-1}f(x^{(q)}_{[j-N,j+N]})\\
        =&C_1+(q-2)cS+C_2\\
        =&C+qcS,
    \end{align*}
    we have
    \begin{align*}
        \sigma_f^{m_q}(x^{(q)}, a)&=\left(\sigma_X^{m_q}(x^{(q)}), \sum_{j=0}^{m_q-1}f(\sigma_X^j(x^{(q)}))\right)\\
        &=\left(\sigma_X^{m_q}(x^{(q)}), a+C+qcS\right)\\
        &=\left(\sigma_X^{m_q}(x^{(q)}), b\right)\in\mathcal{V},
    \end{align*}
    which implies that $\sigma_f^{m_q}(\mathcal{U})\cap\mathcal{V}\ne\emptyset$. Thus, topological transitivity follows.
\end{proof}

\begin{theorem}\label{ev_surj}
    Let $m \ge 1$ be an integer. Then the local-global evaluation map $\mathrm{Ev}_m: H^1(G_X, \mathbb{F}_p) \to V_m(X, p)$ is surjective.
    
\end{theorem}

\begin{proof} 
    Since $\mathcal{P}_m(X)$ is a finite set, we may write $\mathcal{P}_m(X) = \{\Gamma_1, \dots, \Gamma_r\}$. For each $1 \le i \le r$, fix a representative $x^{(i)} \in \Gamma_i$. Let $S$ be the set of all periodic points in these orbits, given by
    \[
    S=\bigcup_{i=1}^r\Gamma_i=\{ \sigma_X^j(x^{(i)})\mid 1 \le i \le r, \, 0 \le j < |\Gamma_i|\}.
    \]
    Since $S$ is a finite set, there exists an $N \ge 1$ such that for any $s, t \in S$ with $s \ne t$, we have $s_{[-N,N]} \ne t_{[-N,N]}$. For each orbit $\Gamma_i$, we denote the cylinder set containing the representative $x^{(i)}$ by
    \[
    U_i=[(x^{(i)})_{[-N,N]}]_{-N}=\{x\in X\mid x_{[-N,N]}=(x^{(i)})_{[-N,N]}\}.
    \]
    Then, the following hold:
    \begin{itemize}
        \item If $i \ne j$, then $U_i \cap U_j = \emptyset$.
    \item For any $1 \le i \le r$ and $1 \le j \le |\Gamma_i|-1$, we have $\sigma_X^j(x^{(i)}) \notin U_i$.
    \end{itemize}
    \indent Fix $v = (v_1, v_2, \dots, v_r) \in \bigoplus_{i=1}^r \mathbb{F}_p$. Our goal is to find a $\chi$ such that $\mathrm{Ev}_m(\chi) = v$. Since $\mathrm{Ev}_m(0) = 0$, we may assume that $v \ne 0$.
    We define a continuous map $f:X\to \F_p$ by
    \begin{align*}
    f(x)=\left\{ 
    \begin{aligned} 
    & v_i &(x \in U_i)\\
    & 0 &(\text{otherwise})
    \end{aligned}
    \right. .
\end{align*}
    Since the sets $\{U_i\}_{1\le i\le r}$ are pairwise disjoint, $f$ is well-defined. Let $Y=X\times \F_p$, and define $\sigma_Y:Y\to Y$ by
    \[
    \sigma_Y(x,c)=(\sigma_X(x),c+f(x)).
    \]
    Since $v\ne 0$, by applying Lemma~\ref{skew_sft}, the dynamical system $(Y,\sigma_Y)$ can be identified with an irreducible SFT. If we define a map $\vp: Y\to X$ by $\vp(x,c)=x$, then $\vp$ is an unramified factor map with $\deg(\vp)=p$. For $a\in\F_p$, we define an automorphism $\tau_a:Y\to Y$ by
    \[
    \tau_a(x,c)=(x,c+a).
    \]
    Then we have $\vp\circ\tau_a=\vp$, which implies $\tau_a\in\aut(Y/X,\vp)$. Therefore, we obtain
    \[
    \aut(Y/X,\vp)=\{\tau_a\mid a\in\F_p\}.
    \]
    Since $|\aut(Y/X,\vp)|=p=\deg(\vp)$, $\vp$ is a Galois factor. Furthermore, the map $F:\F_p\to \gal(Y/X,\vp)$ given by $a\mapsto\tau_a$ provides an isomorphism $\gal(Y/X,\vp)\cong \F_p$. 
    Since the natural projection $\pi:G_X\to \gal(Y/X,\vp)$ is a continuous surjective homomorphism, the map $\chi=F^{-1}\circ\pi:G_{X}\to \F_p$ is also a continuous surjective homomorphism. That is, $0\ne\chi\in H^1(G_X,\F_p)$. For each $1\le i\le r$, we have
    \begin{align*}
        \sigma_Y^{|\Gamma_i|}(x^{(i)},0)&=\left(\sigma_X^{|\Gamma_i|}(x^{(i)}),\sum_{j=0}^{|\Gamma_i|-1}f(\sigma_X^j(x^{(i)}))\right)\\
        &=(x^{(i)},v_i+0+0+\cdots+0)\\
        &=(x^{(i)}, v_i)\\
        &=\tau_{v_i}(x^{(i)},0)\\
        &=(x^{(i)},0)+v_i.
    \end{align*}
    This implies that $\mathrm{Frob}_{\Gamma_i}(\chi)=v_i$. From the above, we have shown that $\mathrm{Ev}_m(\chi)=v$, which implies the surjectivity.
\end{proof}

\subsection{Permutations of Periodic Orbits and Galois Representations}\label{permu_of_per_orb}
In this subsection, we factor $\Upsilon_p$ through $\mathrm{Ker}(\mathrm{Ev}_m)$. As a result, we can extract from $\Upsilon_p$ only the information regarding permutations of periodic orbits. One of the pioneering works on permutations of periodic orbits is due to Boyle and Krieger \cite{BoyleKrieger}. The sign homomorphism and the gyration number homomorphism play central roles in their work. In what follows, we show that the rich information contained in $\Upsilon_p$ includes the information of the sign homomorphism. We denote the sign homomorphism by $\mathrm{OS}_m \colon \aut(X) \to \{-1,1\}.$

\begin{lemma}\label{gal_rep_ker_inv}
    Let $f \in \aut(X)$, and let $m \ge 1$ be an integer. Then we have
    \[
    (\Upsilon_p(f))(\mathrm{Ker}(\mathrm{Ev}_m))=\mathrm{Ker}(\mathrm{Ev}_m).
    \]
\end{lemma}

\begin{proof}
    Fix $\chi \in \mathrm{Ker}(\mathrm{Ev}_m)$ and $\Gamma \in \mathcal{P}_m(X)$. By Corollary~\ref{frob_equiv_cor}, we have
    \[
    \mathrm{Frob}_\Gamma((\Upsilon_p(f))(\chi))=\mathrm{Frob}_{f^{-1}(\Gamma)}(\chi)=0,
    \]
    which implies that $(\Upsilon_p(f))(\chi) \in \mathrm{Ker}(\mathrm{Ev}_m)$.
    This means that $\mathrm{Ker}(\mathrm{Ev}_m) \subseteq (\Upsilon_p(f))^{-1}(\mathrm{Ker}(\mathrm{Ev}_m))$, and hence $(\Upsilon_p(f))(\mathrm{Ker}(\mathrm{Ev}_m)) \subseteq \mathrm{Ker}(\mathrm{Ev}_m)$ holds.
    For the reverse inclusion, since $\Upsilon_p$ is a group homomorphism (Theorem~\ref{gal_rep_homo}), we have $\Upsilon_p(f^{-1}) = (\Upsilon_p(f))^{-1}$. Thus, it suffices to apply the same argument to $f^{-1}$.
\end{proof}

\begin{definition}
    Let $f \in \aut(X)$, and let $m \ge 1$ be an integer. The map $\Upsilon_p(f): H^1(G_X, \F_p) \to H^1(G_X, \F_p)$ induces a linear isomorphism on $H^1(G_X, \F_p)/\mathrm{Ker}(\mathrm{Ev}_m)$.
    That is, we define
    \[
    \overline{\Upsilon_p}(f)_m:H^1(G_X,\F_p)/\mathrm{Ker}(\mathrm{Ev}_m)\to H^1(G_X,\F_p)/\mathrm{Ker}(\mathrm{Ev}_m)
    \]
    by
    \[
    (\overline{\Upsilon_p}(f)_m)([\chi]) := [(\Upsilon_p(f))(\chi)].
    \]
    This map is well-defined by Lemma~\ref{gal_rep_ker_inv}, and since $\Upsilon_p(f)$ is a linear isomorphism, $\overline{\Upsilon_p}(f)_m$ is also a linear isomorphism.
\end{definition}

\begin{lemma}\label{ker_welldef}
    Let $m \ge 1$ be an integer, $\Gamma \in \mathcal{P}_m(X)$, and $\chi, \chi' \in H^1(G_X, \F_p)$. If $[\chi] = [\chi'] \in H^1(G_X, \F_p)/\mathrm{Ker}(\mathrm{Ev}_m)$, then we have $\mathrm{Frob}_{\Gamma}(\chi)=\mathrm{Frob}_{\Gamma}(\chi')$.
\end{lemma}

\begin{proof}
    Suppose that $[\chi]=[\chi']\in H^1(G_X,\F_p)/\mathrm{Ker}(\mathrm{Ev}_m)$. Then, we have $\chi-\chi'\in \mathrm{Ker}(\mathrm{Ev}_m)$. Therefore, $\mathrm{Frob}_{\Gamma}(\chi-\chi')=0$. By the linearity of $\mathrm{Frob}_{\Gamma}$ (Proposition~\ref{frob_linear}), we obtain the desired conclusion.
\end{proof}

\begin{definition}
    Let $m \ge 1$ be an integer, and let $\Gamma \in \mathcal{P}_m(X)$. For $[\chi] \in H^1(G_X, \F_p)/\mathrm{Ker}(\mathrm{Ev}_m)$, we write
    \[
    \mathrm{Frob}_{\Gamma}([\chi])=\mathrm{Frob}_{\Gamma}(\chi).
    \]
    This is well-defined by Lemma~\ref{ker_welldef}.
\end{definition}

\begin{theorem}\label{det_os}
    Let $f \in \aut(X)$, and let $m \ge 1$ be an integer. Then, there exists a basis for $H^1(G_X, \F_p)/\mathrm{Ker}(\mathrm{Ev}_m)$ such that the matrix representation of $\overline{\Upsilon_p}(f)_m$ with respect to this basis is the permutation matrix of the permutation on $\mathcal{P}_m(X)$ induced by $f$.
\end{theorem}

\begin{proof}
    Let $\mathcal{P}_m(X) = \{\Gamma_1, \dots, \Gamma_r\}$. By Theorem~\ref{ev_surj}, the $m$-th local-global evaluation map $\mathrm{Ev}_m$ is surjective. Thus, by the first isomorphism theorem, we obtain an isomorphism
    \[
    \overline{\mathrm{Ev}_m}:H^1(G_X,\F_p)/\mathrm{Ker}(\mathrm{Ev}_m)\to V_m(X,p).
    \]
    Let $e_1, \dots, e_r$ be the standard basis for $V_m(X, p)$. For each $1 \le i \le r$, let $\chi_i \in H^1(G_X, \F_p)$ be a representative of $(\overline{\mathrm{Ev}_m})^{-1}(e_i)$. This determines a basis $[\chi_1], \dots, [\chi_r]$ for $H^1(G_X, \F_p)/\mathrm{Ker}(\mathrm{Ev}_m)$.
    Then, for any $1 \le i, j \le r$, we have $\mathrm{Frob}_{\Gamma_j}([\chi_i]) = \delta_{ij}$, where $\delta_{ij}$ is the Kronecker delta. Let $\tau \in \mathfrak{S}_r$ be the permutation on $\mathcal{P}_m(X)$ induced by $f$. Fix $1 \le j \le r$. By Corollary~\ref{frob_equiv_cor}, for each $1 \le i \le r$, we have
    \begin{align*}
        \mathrm{Frob}_{\Gamma_i}((\overline{\Upsilon_p}(f)_m)([\chi_j]))
        &=\mathrm{Frob}_{f^{-1}(\Gamma_i)}([\chi_j]) \\
        &=\mathrm{Frob}_{\Gamma_{\tau^{-1}(i)}}([\chi_j])\\
        &=\left\{\begin{aligned}
            &1   &(i=\tau(j)) \\
            &0   &(i\ne\tau(j))
        \end{aligned}
        \right.\\
        &=\mathrm{Frob}_{\Gamma_i}([\chi_{\tau(j)}]).
\end{align*}
    That is,
    \[
    \overline{\mathrm{Ev}_m}((\overline{\Upsilon_p}(f)_m)([\chi_j]))=\overline{\mathrm{Ev}_m}([\chi_{\tau(j)}]).
    \]
    Since $\overline{\mathrm{Ev}_m}$ is injective, it follows that
    \[
    (\overline{\Upsilon_p}(f)_m)([\chi_j])=[\chi_{\tau(j)}].
    \]
    This yields the conclusion.
\end{proof}

Since the determinant of a permutation matrix is the sign of the permutation, we obtain the following corollary.

\begin{corollary}\label{cor_det_os}
    Let $f \in \aut(X)$, and let $m \ge 1$ be an integer. Then we have
    \[
    \det(\overline{\Upsilon_p}(f)_m) \equiv \mathrm{OS}_m(f) \pmod p.
    \]
\end{corollary}

\subsection{Application: Conjugacy Problem for Automorphism Groups}\label{subsection_application_conj}

Determining whether two automorphisms are conjugate is always a problem of interest in group theory. Note that although the term ``conjugacy'' is used in both dynamical systems and group theory, the conjugacy referred to here is in the group-theoretic sense. 
The following theorem shows that the Galois representation serves as a conjugacy invariant. The proof is straightforward and is therefore omitted.

\begin{theorem}
    The characteristic polynomial, determinant, and trace of $\overline{\Upsilon_p}(\cdot)_m:\aut(X)\to \mathrm{GL}(H^1(G_X,\F_p)/\mathrm{Ker}(\mathrm{Ev}_m))$ are conjugacy invariants of automorphisms. 
\end{theorem}

In the following, we give an example of determining that two specific automorphisms are not conjugate to each other by using the representation $\overline{\Upsilon_p}(\cdot)_m$. By Theorem~\ref{det_os}, $\overline{\Upsilon_p}(\cdot)_m$ inherently contains the information of the permutations of periodic orbits. Therefore, we remark that the non-conjugacy of the automorphisms in the following example can be determined simply by observing the permutations of the periodic orbits, without calculating $\overline{\Upsilon_p}(\cdot)_m$. However, while many existing algebraic invariants, such as representations to dimension groups, lose non-abelian information like the commutator subgroup, the representation $\Upsilon_p$ retains such non-abelian information abundantly. The following example can be viewed as an illustration of this very fact. By factoring $\Upsilon_p$ through a subgroup different from $\mathrm{Ker}(\mathrm{Ev}_m)$ or by appropriately restricting it, it seems possible to achieve more non-trivial conjugacy determinations and to understand the structure of automorphism groups. We leave these as future tasks.

\begin{example}
    Consider $\AA=\{1,2,3,4,5,6,7,8\}$ and $X=\AA^\Z=\Sigma_8$. Let $\tau, \tau' \in \mathfrak{S}_8$ be permutations on the alphabet given by
    \begin{align*}
        &\tau=(1,2)(3,4)\\
        &\tau'=(1,2)(3,4)(5,6)(7,8),
    \end{align*}
    and let $f = \tau_\infty$ and $g = \tau'_\infty$ be the automorphisms induced by them. We aim to prove that $f$ and $g$ are not conjugate, that is, there exists no $h$ such that $f = hgh^{-1}$.
    Since the alternating group $\mathfrak{A}_8$ is a perfect group, the commutator subgroup $[\mathfrak{A}_8, \mathfrak{A}_8]$ is equal to $\mathfrak{A}_8$. Because both $\tau$ and $\tau'$ are even permutations, they can be written as products of commutators. This means that such information vanishes under representations such as the sign homomorphism or the action on the dimension group. Indeed, for any $m \ge 1$, we have $\mathrm{OS}_m(f) = \mathrm{OS}_m(g) = 1$.
    Now, let us consider the trace of $\overline{\Upsilon_3}(\cdot)_1$ for $m=1$ and $p=3$. First, the set of fixed points of $X$ is given by $\mathcal{P}_1=\{\Gamma_i=\{i^\infty\}\mid 1\le i\le 8\}$.
    By Theorem~\ref{det_os}, the matrix representation of $\overline{\Upsilon_3}(f)_1$ is equal to the permutation matrix associated with $\tau$, and similarly for $g$. That is, we have
    \[
    \overline{\Upsilon_3}(f)_1=
    \begin{bmatrix}
        0 & 1 & 0 & 0 & 0 & 0 & 0 & 0 \\
1 & 0 & 0 & 0 & 0 & 0 & 0 & 0 \\
0 & 0 & 0 & 1 & 0 & 0 & 0 & 0 \\
0 & 0 & 1 & 0 & 0 & 0 & 0 & 0 \\
0 & 0 & 0 & 0 & 1 & 0 & 0 & 0 \\
0 & 0 & 0 & 0 & 0 & 1 & 0 & 0 \\
0 & 0 & 0 & 0 & 0 & 0 & 1 & 0 \\
0 & 0 & 0 & 0 & 0 & 0 & 0 & 1
    \end{bmatrix},
     \overline{\Upsilon_3}(g)_1=
    \begin{bmatrix}
    0 & 1 & 0 & 0 & 0 & 0 & 0 & 0 \\
1 & 0 & 0 & 0 & 0 & 0 & 0 & 0 \\
0 & 0 & 0 & 1 & 0 & 0 & 0 & 0 \\
0 & 0 & 1 & 0 & 0 & 0 & 0 & 0 \\
0 & 0 & 0 & 0 & 0 & 1 & 0 & 0 \\
0 & 0 & 0 & 0 & 1 & 0 & 0 & 0 \\
0 & 0 & 0 & 0 & 0 & 0 & 0 & 1 \\
0 & 0 & 0 & 0 & 0 & 0 & 1 & 0
    \end{bmatrix}.
    \]
    Thus, we obtain
    \[
    \mathrm{Tr}(\overline{\Upsilon_3}(f)_1)=4\equiv 1\not\equiv0=\mathrm{Tr}(\overline{\Upsilon_3}(g)_1)\pmod{3},
    \]
    which shows that $f$ and $g$ are not conjugate.
\end{example}

\begin{remark}
    The choice of $\mathfrak{A}_8$ in the above example is not essential; one can construct infinitely many analogous examples via a similar construction. Furthermore, it is likely that the marker homomorphisms using $\mathfrak{A}_8$ can also be distinguished by $\overline{\Upsilon_p}(\cdot)_m$.
\end{remark}

\section{Questions}
We conclude this paper by presenting some questions that naturally arise from our results. 

\begin{question}
    Fix a specific irreducible SFT $X$, such as a full shift or the golden mean shift. Then, for any finite group $G$, does there exist a Galois factor $\vp: Y \to X$ over $X$ whose Galois group is isomorphic to $G$? More strongly, does a Galois factor of the form $\vp: X \to X$ always exist? For this question, see Subsection~\ref{subsection_inverse galois}.
\end{question}

\begin{question}
Can the Galois theory of irreducible SFTs be generalized further?
In this paper, we have restricted our attention to irreducible SFTs, but what about broader classes of dynamical systems? Examples include sofic shifts, countable state Markov shifts, or dynamical systems satisfying certain specific conditions. Furthermore, the Galois factors considered in this paper are exclusively finite-to-one, a situation analogous to considering only finite extensions in field theory. Is it possible to consider an analogue of infinite extensions? In particular, can we identify a class of dynamical systems that admits an object corresponding to the algebraic closure in the Galois theory of fields or the universal cover in the theory of covering spaces? \\
\indent Note that, drawing an analogy from the contents of Subsection~\ref{section_three} and the theory of covering spaces, it seems intuitively reasonable to simply replace the $n$-to-$1$ condition with a local homeomorphism condition. At this stage, however, the correct formulation of this condition remains unclear.
\end{question}

\begin{question}
    Can we consider analogues of concepts already existing in the Galois theory of fields or the Galois theory of covering spaces within the Galois theory of SFTs? For example, given an unramified factor map $\alpha: Y \to X$ that is not necessarily a Galois factor, does there always exist a Galois factor that can be regarded as its Galois closure?
\end{question}

\begin{question}
   Can we compute the absolute Galois group or the first Galois cohomology group for specific SFTs? Beyond specific computations, what are the general topological and algebraic properties of $G_X$ as a profinite group? For example, what is its cohomological dimension? Does $G_X$ become a free profinite group?
\end{question}

\begin{question}
    In this paper, we considered the first Galois cohomology with coefficients in $\F_p$. Can we extract other information by changing the coefficients? Furthermore, is it also possible to extract information from higher-degree cohomology groups?
\end{question}

\begin{question}
    In number theory or algebra, class field theory describes the maximal abelian extension of a field.
    Is there a ``Class Field Theory'' for SFTs? In other words, can we explicitly describe the abelianization $G_X^{\mathrm{ab}} = G_X / \overline{[G_X, G_X]}$ of the absolute Galois group? 
\end{question}

\begin{question}
    Are there relationships between the absolute Galois group (as well as Galois cohomology) and other invariants? In Corollary~\ref{cor_det_os}, we showed that the information regarding the sign of permutations is contained in $\Upsilon_p$. To what extent does $\Upsilon_p$ capture information from other invariants, such as the dimension group or the Bowen–Franks group, and conversely, what information does it fail to capture? In relation to this, Corollary~\ref{cor_shiftker} demonstrated that at least $\langle\sigma_X\rangle$ is contained in $\mathrm{Ker }\rho_X$, but how large is $\mathrm{Ker }\rho_X$ in general?
    Furthermore, does the topological entropy impose any constraints on the structure of the absolute Galois group? Intuitively, one might expect that a smaller entropy leads to a simpler structure of the absolute Galois group (or Galois cohomology), whereas a larger entropy causes them to become smaller. Is this intuition correct?
\end{question}

\begin{question}
    The study of shift equivalence and strong shift equivalence in SFTs has a long history. Many standard invariants of SFTs are shift equivalence invariants. This implies a natural trade-off: while these invariants are highly tractable to compute, they naturally capture a coarser level of information. How fine or coarse is the information captured by the absolute Galois group and Galois cohomology? For instance, is it possible to construct two irreducible SFTs $X_1$ and $X_2$ that are shift equivalent but not strong shift equivalent, yet have different absolute Galois groups $G_{X_1}$ and $G_{X_2}$? Furthermore, although it is probably false, could the absolute Galois group serve as a complete invariant?
\end{question}

\begin{question}
    SFTs have long been regarded as an important class since they arise from Markov partitions of dynamical systems. What dynamical meaning does a Galois factor of an irreducible SFT carry for the original dynamical system? Furthermore, can the Galois group, the absolute Galois group, or Galois cohomology impose any constraints on the structure of the original dynamical system?
\end{question}

\begin{question}
    Does the Galois theory of SFTs satisfy the axioms of a Galois category?
\end{question}

\section*{Acknowledgment}
I am deeply grateful to Professor Tomoo Yokoyama for his many helpful suggestions regarding this work.

\section*{AI Usage Declaration}
The artificial intelligence tool Gemini was utilized during the conceptualization and outlining phases of this study to assist with initial brainstorming, text refinement, English translation, and grammatical checking. The generated ideas were independently verified and further developed by the author, who also thoroughly reviewed and revised all outputs. The final responsibility for the accuracy and originality of this manuscript rests entirely with the author.

\section*{Appendix}
We present the proof of Lemma~\ref{orbitspaceSFT}. We first recall some concepts.
Let $G=(V,E)$ be a finite directed graph, and let $H$ be a finite subgroup of $\aut(G)$. The group $H$ acts naturally on $G$. For $v, w \in V$, we write $v \sim w$ if there exists $h \in H$ such that $h(v) = w$. Similarly, for $e, f \in E$, we write $e \sim f$ if there exists $h \in H$ such that $h(e) = f$. These define equivalence relations $\sim$ on $V$ and $E$. We denote their equivalence classes by $[v]$ and $[e]$, respectively. Considering the quotient sets $V/H = V/\sim = \{[v] \mid v \in V\}$ and $E/H = E/\sim = \{[e] \mid e \in E\}$, the pair $G/H = (V/H, E/H)$ naturally forms a directed graph. This is called the quotient graph of $G$ by $H$.
Next, let $G'=(V',E')$ be another graph, and let $p:G'\to G$ be a surjective graph homomorphism. We say that $(G', p)$ is a covering of $G$ if for every $v' \in V'$, the map $p$ induces bijections from $E_{v'}(G')$ to $E_{p(v')}(G)$ and from $E^{v'}(G')$ to $E^{p(v')}(G)$.
We are now ready to prove Lemma~\ref{orbitspaceSFT}.

\begin{proof}
    Fix $\id_{Y}\ne\varphi\in H$, and define $f_\varphi:Y\to \R$ by
    \[
    f_\varphi(y)=d_Y(\varphi(y),y),
    \]
    where $d_Y$ is a metric on $Y$. Since $\varphi$ is continuous, $f_\varphi$ is also continuous. Because $Y$ is compact, $f_\varphi$ attains its minimum value $\varepsilon_\varphi\ge 0$. Since the action of $H$ is free, we have $\varphi(y)\ne y$ for all $y\in Y$, which implies $\varepsilon_\varphi>0$. Let $\varepsilon>0$ be defined by
    \[
    \varepsilon=\min\{\varepsilon_\varphi\mid \varphi\in H\setminus\{\id_Y\}\}
    \]
    and choose $N\in\N$ such that $2^{-N}<\varepsilon$. 
    Since $H$ is a finite group, let $L$ be the maximum window size of all $\varphi\in H$ (i.e., each $\varphi$ is induced by a local rule $\BB_L(Y)\to\BB_L(Y)$). Let $M=\max\{2N+1,L\}$, and let $Y^{[M]}$ be the $M$-th higher block shift of $Y$.
    Since $Y^{[M]}$ is an irreducible SFT, we may assume that $Y^{[M]}$ is an edge shift $Y^{[M]}=\mathsf{X}_G$ of some strongly connected finite directed graph $G=(V,E)$. Because $Y^{[M]}$ is topologically conjugate to $Y$, we can regard $H$ as a subgroup of $\aut(Y^{[M]})$. By taking the $M$-th higher block shift, we may assume that the window size of $\varphi$ is $1$. That is, each element $\varphi\in H$ is induced by a block code $\hat{\varphi}:E\to E$. By naturally extending this correspondence on the edge set to the vertex set, we obtain a graph endomorphism $\hat{\varphi}:G\to G$. Furthermore, since $\varphi$ is an automorphism of the shift space, $\hat{\varphi}$ must be a graph automorphism.
    Now, let
    \[
    \hat{H}=\{\hat{\varphi}\in \aut(G)\mid \varphi\in H\}.
    \]
    This is a finite subgroup of $\aut(G)$ and acts naturally on $G$. We claim that this action of $\hat{H}$ is free. Indeed, fix $\hat{\varphi}\in \hat{H}\setminus\{\id_G\}$ and $e\in E$. There exists $y\in Y^{[M]}$ such that $y_{[0]}=e$. By the choice of $N$ and the definition of the higher block representation, we have $d_{Y^{[M]}}(\varphi(y),y)\ge1$, which means, by the definition of the metric on the shift space, that $(\varphi(y))_{[0]}\ne y_{[0]}$. This implies $\hat{\varphi}(e)\ne e$, showing that the action of $\hat{H}$ on $E$ is free. The freeness of the action of $\hat{H}$ on $V$ follows similarly.\\
    \indent Based on the above, by replacing $Y$ with $Y^{[M]}$, redefining the graph $G=(V,E)$ accordingly, and using the definition of $L$, we may assume from the beginning that $Y$ satisfies the following:
    \begin{itemize}
        \item $Y$ is an edge shift of a strongly connected finite directed graph $G=(V,E)$.
        \item The window size of each $\varphi\in H$ is $1$.
        \item $H$ acts freely on $G$.
    \end{itemize}
    Let $p:G\to G/\hat{H}$ be the projection onto the quotient graph. We show that $(G,p)$ is a covering of $G/\hat{H}$. Fix $v\in V$. It suffices to show that both $p^+=p|_{E_v(G)}:E_v(G)\to E_{p(v)}(G/\hat{H})$ and $p^-=p|_{E^v(G)}:E^v(G)\to E^{p(v)}(G/\hat{H})$ are bijections. Since the arguments are similar, we only prove this for $p^+$. \\
    \indent First, we show surjectivity. Fix $\tilde{e}\in E_{p(v)}(G/\hat{H})$. There exists $e'\in E$ such that $\tilde{e}=[e']$. Since $i([e'])=p(v)=[v]$, we have $i(e')\in [v]$. Thus, there exists $\hat{\varphi}\in \hat{H}$ such that $i(e')=\hat{\varphi}(v)$. Set $e=\hat{\varphi}^{-1}(e')$. Then we have
    \[
    i(e)=i(\hat{\varphi}^{-1}(e'))=\hat{\varphi}^{-1}(i(e'))=\hat{\varphi}^{-1}(\hat{\varphi}(v))=v.
    \]
    This shows that $e\in E_v(G)$. Furthermore,
    \[
    p^+(e)=[\hat{\varphi}^{-1}(e')] = [e'] = \tilde{e},
    \]
    which implies that $p^+$ is surjective. Next, we prove injectivity. Suppose $e_1,e_2 \in E_v(G)$ satisfy $[e_1]=p^+(e_1)=p^+(e_2)=[e_2]$. Then there exists $\hat{\varphi}\in \hat{H}$ such that $e_2=\hat{\varphi}(e_1)$. We have
    \[
    v=i(e_2)=i(\hat{\varphi}(e_1))=\hat{\varphi}(i(e_1))=\hat{\varphi}(v).
    \]
    Since the action of $\hat{H}$ is free, we have $\hat{\varphi}=\id_G$. Therefore, $e_2=e_1$ holds. \\
    \indent Let us define an SFT by $W=\mathsf{X}_{G/\hat{H}}$. Since $G$ is strongly connected, $G/\hat{H}$ is also strongly connected, meaning $W$ is an irreducible SFT. The projection $p:G\to G/\hat{H}$ provides a correspondence between edges. Because the alphabet of $Y$ is $E$ and the alphabet of $W$ is $E/\hat{H}$, $p$ acts as a block code. This naturally induces a sliding block code $p_\infty:Y\to W$. We will prove the following claim regarding this map:\\
    \textbf{Claim. }Let $y, y' \in Y$. Then the following are equivalent:
    \begin{enumerate}\item There exists $\varphi \in H$ such that $y' = \varphi(y)$.\item $p_\infty(y) = p_\infty(y')$.\end{enumerate}
    \begin{proof}[Proof of Claim]
    First, assume that there exists $\varphi \in H$ such that $y' = \varphi(y)$. Then, for any $i \in \Z$, we have $y'_i = \hat{\varphi}(y_i)$, and therefore $p(y'_i) = [y'_i] = [y_i] = p(y_i)$. This means $p_\infty(y) = p_\infty(y')$. Conversely, suppose $p_\infty(y) = p_\infty(y')$. Fix $i \in \Z$. By assumption, we have $[y_i] = p(y_i) = p(y'_i) = [y'_i]$, this means that there exists $\hat{\varphi}_i \in \hat{H}$ such that $y'_i = \hat{\varphi}_i(y_i)$. Then, we have
    \[
    i(y'_{i+1})=t(y'_i)=t(\hat{\varphi}_i(y_i))=\hat{\varphi}_i(t(y_i))=\hat{\varphi}_i(i(y_{i+1}))=i(\hat{\varphi}_i(y_{i+1})).
    \]
    This implies that $y'_{i+1}$ and $\hat{\varphi}_i(y_{i+1})$ share the same initial vertex, say $v \in V$. Furthermore, we have 
    \[
    p(y'_{i+1})=p(y_{i+1})=[y_{i+1}]=
    [\hat{\varphi}_i(y_{i+1})]=p(\hat{\varphi}_i(y_{i+1})).
    \]
    Since $(G, p)$ is a covering of $G/\hat{H}$, the restriction 
    \[
    p|_{E_v(G)} : E_v(G) \to E_{p(v)}(G/\hat{H})
    \]
    is a bijection. Because both $y'_{i+1}$ and $\hat{\varphi}_i(y_{i+1})$ belong to $E_v(G)$ and map to the same element under $p$, we have $y'_{i+1} = \hat{\varphi}_i(y_{i+1})$.
    Combining this with $y'_{i+1} = \hat{\varphi}_{i+1}(y_{i+1})$, we obtain $\hat{\varphi}_i(y_{i+1}) = \hat{\varphi}_{i+1}(y_{i+1})$. Thus, $(\hat{\varphi}_{i+1}^{-1} \circ \hat{\varphi}_i)(y_{i+1}) = y_{i+1}$. Since $\hat{H}$ acts freely on $G$, this implies $\hat{\varphi}_{i+1}^{-1} \circ \hat{\varphi}_i = \id_G$, meaning $\hat{\varphi}_{i+1} = \hat{\varphi}_i$.
    Consequently, we have $\dots = \hat{\varphi}_{-1} = \hat{\varphi}_0 = \hat{\varphi}_1 = \dots$. By setting $\hat{\varphi} = \hat{\varphi}_0$, we see that $y'_j = \hat{\varphi}(y_j)$ holds for any $j \in \Z$. This $\hat{\varphi}$ corresponds to an element $\varphi \in H$ satisfying $y' = \varphi(y)$, which completes the proof of the claim.
\end{proof}
    The map $p_\infty: Y \to W$ induces a map $\overline{p_\infty}: Y/H \to W$ defined by $\overline{p_\infty}([y]) = p_\infty(y)$. By the claim, this map is well-defined. Since $p_\infty$ is surjective, $\overline{p_\infty}$ is also surjective. We now show that $\overline{p_\infty}$ is injective. Suppose $\overline{p_\infty}([y]) = \overline{p_\infty}([y'])$. By definition, we have $p_\infty(y) = p_\infty(y')$. By the claim, there exists $\varphi \in H$ such that $y' = \varphi(y)$, which implies $[y] = [y']$. Thus, injectivity follows.
    Next, we show continuity. Let $q_H: Y \to Y/H$ be the natural projection. From the definitions, we have $\overline{p_\infty} \circ q_H = p_\infty$. The natural projection $q_H$ is a quotient map, and $p_\infty$ is continuous since it is a sliding block code. Therefore, by the universal property of quotient maps (Lemma~\ref{universalityofquotientmap}), $\overline{p_\infty}$ is continuous. Since $\overline{p_\infty}: Y/H \to W$ is a continuous bijection from a compact space to a Hausdorff space, it is a homeomorphism. Finally, we show that $\overline{p_\infty}$ commutes with the shift maps. Fix $[y] \in Y/H$. We have
    \begin{align*}
        (\overline{p_\infty}\circ\sigma^*_Z)([y])&=\overline{p_\infty}(\sigma^*_Z([y]))\\
        &=\overline{p_\infty}([\sigma_Y(y)])\\
        &=p_\infty(\sigma_Y(y))\\
        &=\sigma_W(p_\infty(y))\\
        &=\sigma_W(\overline{p_\infty}([y]))\\
        &=(\sigma_W\circ \overline{p_\infty})([y]).
    \end{align*}
    Therefore, $\overline{p_\infty}\circ\sigma^*_Z = \sigma_W\circ \overline{p_\infty}$. We conclude that $\overline{p_\infty}: Z = Y/H \to W$ gives a topological conjugacy.
    \end{proof}

\bibliographystyle{plain}
\bibliography{bibbib}

@book {introduction,
    AUTHOR = {Lind, Douglas and Marcus, Brian},
     TITLE = {An introduction to symbolic dynamics and coding},
 PUBLISHER = {Cambridge University Press, Cambridge},
      YEAR = {1995},
     PAGES = {xvi+495},
      ISBN = {0-521-55124-2; 0-521-55900-6},
   MRCLASS = {58F03 (15A48 54H20 58F20 94A24 94B60)},
  MRNUMBER = {1369092},
MRREVIEWER = {Petr\ K\.urka},
       DOI = {10.1017/CBO9780511626302},
       URL = {https://doi.org/10.1017/CBO9780511626302},
}

@article {cyr_kra1,
    AUTHOR = {Cyr, Van and Kra, Bryna},
     TITLE = {The automorphism group of a shift of subquadratic growth},
   JOURNAL = {Proc. Amer. Math. Soc.},
  FJOURNAL = {Proceedings of the American Mathematical Society},
    VOLUME = {144},
      YEAR = {2016},
    NUMBER = {2},
     PAGES = {613--621},
      ISSN = {0002-9939,1088-6826},
   MRCLASS = {37B50 (37B10 68R15)},
  MRNUMBER = {3430839},
MRREVIEWER = {Dominik\ Kwietniak},
       DOI = {10.1090/proc12719},
       URL = {https://doi.org/10.1090/proc12719},
}

@article {cyr_kra2,
    AUTHOR = {Cyr, Van and Kra, Bryna},
     TITLE = {The automorphism group of a minimal shift of stretched
              exponential growth},
   JOURNAL = {J. Mod. Dyn.},
  FJOURNAL = {Journal of Modern Dynamics},
    VOLUME = {10},
      YEAR = {2016},
     PAGES = {483--495},
      ISSN = {1930-5311,1930-532X},
   MRCLASS = {37B10 (54H20 68R15)},
  MRNUMBER = {3565928},
MRREVIEWER = {Ronnie\ Pavlov},
       DOI = {10.3934/jmd.2016.10.483},
       URL = {https://doi.org/10.3934/jmd.2016.10.483},
}

@article {donoso,
    AUTHOR = {Donoso, Sebasti\'an and Durand, Fabien and Maass, Alejandro
              and Petite, Samuel},
     TITLE = {On automorphism groups of low complexity subshifts},
   JOURNAL = {Ergodic Theory Dynam. Systems},
  FJOURNAL = {Ergodic Theory and Dynamical Systems},
    VOLUME = {36},
      YEAR = {2016},
    NUMBER = {1},
     PAGES = {64--95},
      ISSN = {0143-3857,1469-4417},
   MRCLASS = {37B10 (68R15)},
  MRNUMBER = {3436754},
MRREVIEWER = {Bryna\ Kra},
       DOI = {10.1017/etds.2015.70},
       URL = {https://doi.org/10.1017/etds.2015.70},
}

@article {cyr_franks_kra_petite,
    AUTHOR = {Cyr, Van and Franks, John and Kra, Bryna and Petite, Samuel},
     TITLE = {Distortion and the automorphism group of a shift},
   JOURNAL = {J. Mod. Dyn.},
  FJOURNAL = {Journal of Modern Dynamics},
    VOLUME = {13},
      YEAR = {2018},
     PAGES = {147--161},
      ISSN = {1930-5311,1930-532X},
   MRCLASS = {37B10 (22F50 37B15)},
  MRNUMBER = {3918261},
MRREVIEWER = {Alfredo\ M. G. Costa},
       DOI = {10.3934/jmd.2018015},
       URL = {https://doi.org/10.3934/jmd.2018015},
}

@article {pavlov_schmieding,
    AUTHOR = {Pavlov, Ronnie and Schmieding, Scott},
     TITLE = {Local finiteness and automorphism groups of low complexity
              subshifts},
   JOURNAL = {Ergodic Theory Dynam. Systems},
  FJOURNAL = {Ergodic Theory and Dynamical Systems},
    VOLUME = {43},
      YEAR = {2023},
    NUMBER = {6},
     PAGES = {1980--2001},
      ISSN = {0143-3857,1469-4417},
   MRCLASS = {37B10 (20F50 20F65)},
  MRNUMBER = {4583802},
MRREVIEWER = {Alfredo\ M. G. Costa},
       DOI = {10.1017/etds.2022.7},
       URL = {https://doi.org/10.1017/etds.2022.7},
}

@article {Hedlund,
    AUTHOR = {Hedlund, G. A.},
     TITLE = {Endomorphisms and automorphisms of the shift dynamical system},
   JOURNAL = {Math. Systems Theory},
  FJOURNAL = {Mathematical Systems Theory. An International Journal on
              Mathematical Computing Theory},
    VOLUME = {3},
      YEAR = {1969},
     PAGES = {320--375},
      ISSN = {0025-5661},
   MRCLASS = {54.82},
  MRNUMBER = {259881},
MRREVIEWER = {W.\ R.\ Utz},
       DOI = {10.1007/BF01691062},
       URL = {https://doi.org/10.1007/BF01691062},
}

@article {ryan_shiftcenter,
    AUTHOR = {Ryan, J. Patrick},
     TITLE = {The shift and commutativity},
   JOURNAL = {Math. Systems Theory},
  FJOURNAL = {Mathematical Systems Theory. An International Journal on
              Mathematical Computing Theory},
    VOLUME = {6},
      YEAR = {1972},
     PAGES = {82--85},
      ISSN = {0025-5661},
   MRCLASS = {54H15 (54H20)},
  MRNUMBER = {305376},
MRREVIEWER = {R.\ L.\ Adler},
       DOI = {10.1007/BF01706077},
       URL = {https://doi.org/10.1007/BF01706077},
}

@article {boyle_lind_rudolph,
    AUTHOR = {Boyle, Mike and Lind, Douglas and Rudolph, Daniel},
     TITLE = {The automorphism group of a shift of finite type},
   JOURNAL = {Trans. Amer. Math. Soc.},
  FJOURNAL = {Transactions of the American Mathematical Society},
    VOLUME = {306},
      YEAR = {1988},
    NUMBER = {1},
     PAGES = {71--114},
      ISSN = {0002-9947,1088-6850},
   MRCLASS = {54H20 (20B27 28D15 34C35 58F11)},
  MRNUMBER = {927684},
MRREVIEWER = {Meir\ Smorodinsky},
       DOI = {10.2307/2000831},
       URL = {https://doi.org/10.2307/2000831},
}

@article {bowen_franks,
    AUTHOR = {Bowen, Rufus and Franks, John},
     TITLE = {Homology for zero-dimensional nonwandering sets},
   JOURNAL = {Ann. of Math. (2)},
  FJOURNAL = {Annals of Mathematics. Second Series},
    VOLUME = {106},
      YEAR = {1977},
    NUMBER = {1},
     PAGES = {73--92},
      ISSN = {0003-486X},
   MRCLASS = {58F15},
  MRNUMBER = {458492},
       DOI = {10.2307/1971159},
       URL = {https://doi.org/10.2307/1971159},
}

@article {williams_shift_equiv,
    AUTHOR = {Williams, R. F.},
     TITLE = {Classification of subshifts of finite type},
   JOURNAL = {Ann. of Math. (2)},
  FJOURNAL = {Annals of Mathematics. Second Series},
    VOLUME = {98},
      YEAR = {1973},
     PAGES = {120--153; errata, ibid. (2) 99 (1974), 380--381},
      ISSN = {0003-486X},
   MRCLASS = {58F20 (28A65)},
  MRNUMBER = {331436},
MRREVIEWER = {C.\ B.\ Thomas},
       DOI = {10.2307/1970908},
       URL = {https://doi.org/10.2307/1970908},
}

@article {dimensiongroup,
    AUTHOR = {Krieger, Wolfgang},
     TITLE = {On dimension functions and topological {M}arkov chains},
   JOURNAL = {Invent. Math.},
  FJOURNAL = {Inventiones Mathematicae},
    VOLUME = {56},
      YEAR = {1980},
    NUMBER = {3},
     PAGES = {239--250},
      ISSN = {0020-9910,1432-1297},
   MRCLASS = {28D20 (54H20)},
  MRNUMBER = {561973},
MRREVIEWER = {Douglas\ Lind},
       DOI = {10.1007/BF01390047},
       URL = {https://doi.org/10.1007/BF01390047},
}

@article {Nasu,
    AUTHOR = {Nasu, Masakazu},
     TITLE = {Constant-to-one and onto global maps of homomorphisms between
              strongly connected graphs},
   JOURNAL = {Ergodic Theory Dynam. Systems},
  FJOURNAL = {Ergodic Theory and Dynamical Systems},
    VOLUME = {3},
      YEAR = {1983},
    NUMBER = {3},
     PAGES = {387--413},
      ISSN = {0143-3857,1469-4417},
   MRCLASS = {58F35 (34C35)},
  MRNUMBER = {741394},
MRREVIEWER = {R.\ L.\ Adler},
       DOI = {10.1017/S0143385700002042},
       URL = {https://doi.org/10.1017/S0143385700002042},
}

@article {Jung,
    AUTHOR = {Jung, Uijin},
     TITLE = {Open maps between shift spaces},
   JOURNAL = {Ergodic Theory Dynam. Systems},
  FJOURNAL = {Ergodic Theory and Dynamical Systems},
    VOLUME = {29},
      YEAR = {2009},
    NUMBER = {4},
     PAGES = {1257--1272},
      ISSN = {0143-3857,1469-4417},
   MRCLASS = {37B10 (54H20)},
  MRNUMBER = {2529648},
MRREVIEWER = {Michael\ H.\ Schraudner},
       DOI = {10.1017/S0143385708000692},
       URL = {https://doi.org/10.1017/S0143385708000692},
}

@book{ShatzShatz2016,
url = {https://doi.org/10.1515/9781400881857},
title = {Profinite Groups, Arithmetic, and Geometry},
author = {Stephen S. Shatz and Stephen S. Shatz},
publisher = {Princeton University Press},
address = {Princeton},
doi = {doi:10.1515/9781400881857},
isbn = {9781400881857},
year = {2016},
lastchecked = {2026-04-03}
}

@book{parry1990zeta,
  title={Zeta functions and the periodic orbit structure of hyperbolic dynamics},
  author={Parry, William and Pollicott, Mark},
  series={Ast{\'e}risque},
  number={187-188},
  year={1990},
  publisher={Soci{\'e}t{\'e} math{\'e}matique de France},
  pages={272},
  url={https://www.numdam.org/item/AST_1990__187-188__1_0/}
}

@article {BoyleKrieger,
    AUTHOR = {Boyle, Mike and Krieger, Wolfgang},
     TITLE = {Periodic points and automorphisms of the shift},
   JOURNAL = {Trans. Amer. Math. Soc.},
  FJOURNAL = {Transactions of the American Mathematical Society},
    VOLUME = {302},
      YEAR = {1987},
    NUMBER = {1},
     PAGES = {125--149},
      ISSN = {0002-9947,1088-6850},
   MRCLASS = {54H20 (28D05 54H15)},
  MRNUMBER = {887501},
MRREVIEWER = {Manfred\ Denker},
       DOI = {10.2307/2000901},
       URL = {https://doi.org/10.2307/2000901},
}

@article {TQFT,
    AUTHOR = {Gilmer, Patrick M.},
     TITLE = {Topological quantum field theory and strong shift equivalence},
   JOURNAL = {Canad. Math. Bull.},
  FJOURNAL = {Canadian Mathematical Bulletin. Bulletin Canadien de
              Math\'ematiques},
    VOLUME = {42},
      YEAR = {1999},
    NUMBER = {2},
     PAGES = {190--197},
      ISSN = {0008-4395,1496-4287},
   MRCLASS = {57R56 (57M12 57M27)},
  MRNUMBER = {1692009},
MRREVIEWER = {Olivier\ Collin},
       DOI = {10.4153/CMB-1999-023-4},
       URL = {https://doi.org/10.4153/CMB-1999-023-4},
}

\end{document}